\documentclass[a4paper, reqno, 11pt]{amsart}

\usepackage[english]{babel}
\usepackage{amsmath}
\usepackage{amssymb}
\usepackage{enumerate}
\usepackage{ifthen}
\usepackage{bbm}
\usepackage{color}
\usepackage{latexsym}
\provideboolean{shownotes} 
\setboolean{shownotes}{true}
\usepackage{hyperref}
\usepackage{graphicx}

\usepackage{geometry}
\usepackage{inputenc}

\newcommand{\margnote}[1]{
\ifthenelse{\boolean{shownotes}}%
{\marginpar{\raggedright\tiny\texttt{#1}}}%
{}%
}

\newcommand{\hole}[1]{
\ifthenelse{\boolean{shownotes}}%
{\begin{center} \fbox{ \rule {.25cm}{0cm}
\rule[-.1cm]{0cm}{.4cm} \parbox{.85\textwidth}{\begin{center}
\texttt{#1}\end{center}} \rule {.25cm}{0cm}}\end{center}}
{}
}

\newtheorem{theorem}{Theorem}[section]
\newtheorem{proposition}[theorem]{Proposition}
\newtheorem{lemma}[theorem]{Lemma}

\theoremstyle{remark}
\newtheorem{remark}[theorem]{Remark}

\newcommand{\R}{\mathbb{R}}

\newcommand{\dive}{\mathop{\mathrm {div}}}

\newcommand{\cof}{\mathop{\mathrm{Cof}}}
\newcommand{\trace}{\mathrm{Tr}}
\newcommand{\We}{We}
\numberwithin{equation}{section}

\keywords{Splash singularity, Oldroyd-B, Viscoelasticity, Existence and Stability}

\begin{document}

\title[]{Splash singularity for a free-boundary incompressible viscoelastic fluid model}

\author{Elena Di Iorio}
\address{GSSI - Gran Sasso Science Institute\\
67100 L'Aquila - 
Italy}

\email{elena.diiorio@gssi.it}

\author{Pierangelo Marcati}
\address{GSSI - Gran Sasso Science Institute\\
67100 L'Aquila -
Italy}
\email{pierangelo.marcati@gssi.it}

\author{Stefano Spirito}
\address{Department of Information Engineering, Computer Science and Mathematics\\ University of L'Aquila\\
67100 L'Aquila - Italy}
\email{stefano.spirito@univaq.it}

\begin{abstract}
In this paper we analyze a 2D free-boundary viscoelastic fluid model of Oldroyd-B type at infinite Weissenberg number. Our main goal is to show the existence of the so-called splash singularities, namely  points where the boundary remains smooth but self-intersects. The combination of existence and stability results allows us to construct a special class of initial data, which evolve in time into self-intersecting configurations. To this purpose we apply the classical conformal mapping method and later we move to the Lagrangian framework, as a consequence we deduce the existence of splash singularities. This result extends the result obtained for the Navier-Stokes equations in \cite{CCFGG2}
\end{abstract}


\maketitle

\section{Introduction}
\noindent The purpose of  this paper is to initiate a study on the formation of splash type singularities for viscoelastic fluids of Oldroyd-B type at high Weissenberg number. In particular this paper is devoted to the limit case of infinite Weiessenberg number, namely the system obtained as the result of a infinite relaxation limit for stress relaxation time of the fluid. 
Many complex fluids have a quite different behavior with respect to classical Newtonian fluids. A key feature of viscoelastic fluids is the presence of some  memory effects, namely stress tensors in these fluids depend on the flow history, moreover we can observe stress anisotropy. A viscoelastic fluid generates stresses that are not present in a Newtonian fluid, having the same deformation history. Therefore viscoelastic fluids do not flow like their Newtonian counterparts. Viscoelastic fluids are materials which have both viscous and elastic responses to forces, so we need to take care that stresses created in an elastic material stay constant in time for as long as the deformation is present, while stresses in a viscous fluid dissipate on a time scale governed by its viscosity.
There is a natural analogy with models composed by dashpots and springs in particular linear solids and liquids are often represented graphically by a sequence of springs and dashpots, a serial connection of a spring and a dashpot represents a viscoelastic fluid, while the parallel connection represents a viscoelastic solid. Traditionally, these models are called respectively the Maxwell fluid and Kelvin--Voigt solid models, and they can be considered the simplest models of viscoelastic materials.  Unfortunately the equation for the linear Maxwell fluid model is not frame-invariant, hence  to recover the frame invariance Maxwell introduced the  so called upper convective derivation operator which can be understood to be the time derivative calculated in a coordinate system which is translating and deforming with the fluid, such that by definition the upper-convected time derivative of the left Cauchy-Green deformation tensor (Finger tensor) is always zero. The viscoelasticity model obtained in this way is the classical Upper Convective Maxwell (UCM) model. 
\\ 

  The  Oldroyd-B model is then obtained by assuming the total stress tensor as the sum of the pressure, the fluid viscosity and the polymeric contribution to the stress, where the polymeric tensor obeys the UCM model. Numerical simulations have shown the presence of singularities for high Weissenberg number (see for instance \cite{K} and \cite{ThSh}). It can be expected that these emerging singularities lie at the root of some of complications in numerical simulations of viscoelastic fluids using the Oldroyd-B model. There is a vast literature regarding the \emph{high Weissenberg} number problem, see  for instance Chap. 7 of \cite{OP}, for a careful exposition and a very relevant analysis of these problems. 
\noindent  In the case of viscoelastic materials, one possible source of difficulty is related to understand in what extent the presence of the elastic components  could prevent the development of (splash) singularities. 
This paper is actually devoted to show that, in the case of infinite Weissenberg the existence of this type of singularity occurs also for viscoelastic fluids.
\\

  As we said before, following the Maxwell ideas, we may think that incompressible viscoelastic fluids have elastic and viscous components, modeled as linear combination of springs and dashpots and described by the momentum equation,
\begin{equation*}
\bar{\rho}(\partial_t u+(u\cdot\nabla)u)+\nabla p=\dive\tau,
\end{equation*}
\noindent where $\tau=\nu_s(\nabla u+\nabla u^T)+\tau_p$ denotes the stress tensor with $\nu_s$, the solvent viscosity and $\tau_p$, the extra-stress related to the elastic part. From now on $\bar{\rho}=1$.\\
\noindent The stress tensor satisfies the usual Oldroyd-B model structure

\begin{equation}\label{Old-B}
\tau+\lambda\partial_t^{uc}\tau=\nu_0((\nabla u+\nabla u^T) + \lambda_s\partial_t^{uc}(\nabla u+\nabla u^T)),
\end{equation}
where

\begin{itemize}
\item $\partial_t^{uc}\tau=\partial_t\tau+ (u\cdot\nabla)\tau-\nabla u \tau-\tau\nabla u^T$ denotes the upper convective time derivative,
\item $\nu_0=\nu_s+\nu_p$, denotes the material viscosity, $\nu_s$ the solvent viscosity and  $\nu_p$ the polymeric viscosity respectively,
\item $\lambda$ the relaxation time,
\item $\lambda_s=\frac{\nu_s}{\nu_0}\lambda$.
\end{itemize}
\medskip

\noindent Therefore by the definition of the total stress tensor $\tau$, we obtain an equation for the extra-stress $\tau_p$
\begin{equation}\label{extra-stress}
\lambda\partial_t^{uc} \tau_p+\tau_p= \nu_p(\nabla u+\nabla u^T).
\end{equation}
\medskip

\noindent The constitutive law (\ref{extra-stress}) is coupled with the equations of conservation of mass and momentum. So we get the formulation for the Oldroyd-B model.

\begin{equation}\label{viscosys}
\left\{\begin{array}{lll}
\partial_t u+u\cdot\nabla u +\nabla p= \nu_s\Delta u+\textrm{div}\tau_p \\[1mm]
\partial_t^{uc}\tau_p=-\frac{1}{\lambda}\tau_p+\frac{{\nu}_p}{\lambda}(\nabla u+\nabla u^T)\\[1mm]
\textrm{div} u=0.\\
\end{array}\right.
\end{equation}
\medskip

\noindent In the equation (\ref{Old-B}), we introduced the relaxation time $\lambda$. Indeed, we had to consider the presence of memory effects in viscoelastic materials,  since one of the main features of these materials is that whenever the material is deformed it will try to revert to its original shape, hence the relaxation time $\lambda$ appears naturally inside the memory function. Consequently, the behavior of this Non-Newtonian fluid depends on $\lambda$ and specifically on to the ratio between $\lambda$ and $t_f$, the typical time-scale of the flow, given by $t_f\simeq\sqrt{\frac{1}{2}(\trace(\nabla u+\nabla u^T))^2}$. The ratio $\We=\frac{\lambda}{t_f}$ is the so-called Weissenberg number, see \cite{R}. When  $\lambda$ is greater than $t_f$, the elastic effects are dominant on the other way around the viscous do. We are therefore motivated to study the system (\ref{viscosys}) for high Weissenberg number ($\We\rightarrow\infty$). As a good approximation of this problem, we take the limit case given by the following system

\begin{equation}\label{newviscosys}
\left\{\begin{array}{lll}
\displaystyle\partial_t u+u\cdot\nabla u +\nabla p= \nu_s\Delta u+\dive\tau_p \\[3mm]
\displaystyle \partial_t\tau_p+(u\cdot\nabla)\tau_p-(\nabla u)\tau_p-\tau_p(\nabla u)^T=0\\[3mm]
\displaystyle\dive u=0.\\
\end{array}\right.
\end{equation}
\medskip

\noindent Let us denote with $\alpha\in\mathbb{R}^2$ the material point in the reference configuration and let $X(t,\alpha)$ be the flux associated to the velocity $u$. The following system of ODEs defines the particle-trajectories 

\begin{equation}
\left\{\begin{array}{lll}
\displaystyle\frac{d}{dt}{X}(t,\alpha)=u(t,X(t,\alpha))\\ [3mm]
X(0,\alpha)=\alpha,
\end{array}\right.
\end{equation}
\medskip

\noindent furthermore the deformation gradient $G$ is defined by  
\begin{equation}\label{lag-def-grad}
G(t,\alpha)=\frac{\partial X}{\partial\alpha}(t,\alpha).
\end{equation}

\noindent In Eulerian framework we define the deformation tensor by $F(t,X(t,\alpha))=G(t,\alpha)$, then by the chain rule the deformation gradient satisfies the following transport equation
\begin{equation}\label{transport}
\partial_t F+u\cdot\nabla F=\nabla u F.
\end{equation}
\noindent If we set  $\tau_p(t,X)=F\tau_0 F^T$  it follows   
$$\partial_t\tau_p+(u\cdot\nabla)\tau_p-(\nabla u)\tau_p-\tau_p(\nabla u)^T=0,$$
then we can replace the equation in $\tau_p$ by \eqref{transport}. In this way it is possible to have a closed system in $u$ and $F$, moreover as long as $\tau_p(0,X)$ is positive definite so is $\tau_p(t,X)$.

\noindent The system in $u$ and $F$, is complemented with suitable boundary conditions, given by the static equilibrium of the force fields at the interface, namely

\begin{equation}\label{Eulsys}
\left\{\begin{array}{lll}
\displaystyle\partial_t F+u\cdot\nabla F=\nabla u F
\hspace{4cm} \textrm{ in }\Omega(t) \\ [1mm]
\displaystyle\partial_t u +u\cdot\nabla u -\Delta u +\nabla p= \dive(FF^{T}) \\[1mm]
\displaystyle\dive u=0,\\ [1mm]
\displaystyle (-p\mathcal{I}+(\nabla u+\nabla u^{T})+(FF^{T}-\mathcal{I}))n=0 \hspace{0.9cm}\textrm{on}\hspace{0.1cm}\partial\Omega(t)\\ [1mm]
\displaystyle u(t)_{|t=0}=u_0,\hspace{0.1cm} F(t)_{|t=0}=F_0\hspace{3.1cm} \textrm{in}\hspace{0.3cm}\Omega_0. \\
\end{array}\right.
\end{equation}
\medskip

\noindent 
The variable domain is given by $\Omega(t)=X(t,\Omega_0)$, where $\Omega_0 \subset \R^2$ denotes the initial domain.
We use the following notation $(\dive M)_j=\sum_{i}\partial_i M_{ij}$, for any matrix valued function $M$.
Since $u$ is divergence-free the first equation in the previous system implies the conservation of  $\det F$, hence in order to respect the principle of the non interpenetration of matter $\det F>0$, therefore it is not restrictive to assume $\det F_0 =1$.  For further details on the incompressibility of $F$, we refer to \cite{G} and \cite{LeLZ}.
\\

\noindent There is a vast literature regarding the analysis of initial value boundary problem for the Oldroyd-B model \cite{BGS},  \cite{GS}, \cite{LeMeur} and the infinite Weissenberg system \cite{CZ}, \cite{LLZ}, \cite{LW} and \cite{XZZ}. In particular, we recall that for co-rotational Oldroyd model in \cite{LM1}, the global existence of weak solutions is proved for general initial conditions. However, for the general Oldroyd-B system the global existence of weak solutions is still open. Concerning the infinite Weissenberg system we mention the result \cite{LZ}, the local existence and global existence for data sufficiently close to the equilibrium in two and three dimensions are proved.\\
\noindent The main result of this paper is the following theorem
\begin{theorem}
There exists a time $t^*\in [0,T]$ such that the interface $\partial\Omega(t^*)$ self-intersects in one point (splash singularity).
\end{theorem}
\noindent Recent results on the existence of splash singularity have been obtained by Castro, C\'ordoba, Fefferman, Gancedo and G\'omez-Serrano in \cite{CCFGG2}, for the  2D free boundary Navier-Stokes equations and by Coutand and Shkoller in \cite{CS2}, for the same problem in higher dimensions. Similar results are proved also for the free boundary Euler equations  in \cite{CCFGG1}, \cite{CCFGG3} and \cite{CS1}, with different techniques, since this problem is reversible in time. By using the results of this paper, announced in \cite{DMS1}, the authors proved the existence of splash singularities also for a model with a non linear Piola-Kirchhoff stress  in \cite{DMS}. Moreover, the existence of splash singularities in the case of finite Weissenberg number has been proved in \cite{DMS11}\\ 
\medskip

\noindent The key ingredient for the proof is the use of the classical method of the conformal mapping, see \cite{CCG}. This method   has been used recently for this kind of problems by Castro, C\'ordoba, Fefferman, Gancedo and G\'omez-Serrano in \cite{CCFGG2}. We introduce the map $P(z)=\tilde{z}$, for $z\in\mathbb{C} \setminus \Gamma$, is defined as a branch of $\sqrt{z}$, where $\Gamma$ is a line, passed through the splash point (see fig.\ref{spl}). We take $z\in\mathbb{C} \setminus \Gamma$ in order to make $\sqrt{z}$ an analytic function and to have $P^{-1}(\tilde{z})=\tilde{z}^2$, an entire function. The idea to prove our theorem is to reduce the system  (\ref{Eulsys}), in Eulerian coordinates, to a system in Lagrangian coordinates in order to have a fixed boundary, as in the paper of Beale \cite{B}. The second key observation regards the tangential behavior of the deformation gradient at the Lagrangian boundary. Therefore, it shows that the way the viscoelastic deformation acts to the boundary does not prevent the natural tendency of the fluid to form splash singularities. The idea hidden in the proof is inspired by the geometric construction in \cite{CCFGG2}, as explained below.

\begin{itemize}
\item  Let the initial domain  $\Omega_0$ be a non regular domain as (b), for this reason we use the conformal map $P$ and  by projection we get $\tilde{\Omega}_0$, a non-splash type domain.
\item If $\{\tilde{\Omega}_0, \tilde{u}(0,\cdot), \tilde{p}(0,\cdot), \tilde{F}(0,\cdot)\}$ are smooth we can prove the local existence of a solution $\{\tilde{\Omega}(t), \tilde{u}(t,\cdot), \tilde{p}(t,\cdot), \tilde{F}(t,\cdot)\}$, for $t\in[0,T]$, and $T$ small enough. (Section \ref{sec4})
\item By a suitable choice of the initial velocity, in particular $\tilde{u}(0,\tilde{z}_1)\cdot n>0,$ $\tilde{u}(0,\tilde{z}_2)\cdot n>0$ (Section \ref{sec6}) there exists $\bar{t}\in(0,T]$ such that $P^{-1}(\partial\tilde{\Omega}(\bar{t}))$ is \textit{self-intersecting}, as in the case (c). This solution lives only in the complex plane so it cannot be reversed into a solution in the non tilde complex plane, by $P^{-1}$. Hence it is not sufficient to prove the existence of splash singularity.
\item To solve the problem in the non-tilde domain, we take a one-parameter family\\
 $\{\tilde{\Omega}_{\varepsilon}(0),\tilde{u}_{\varepsilon}(0),\tilde{p}_{\varepsilon}, \tilde{F}_{\varepsilon}(0)\}$, with $\tilde{\Omega}_{\varepsilon}(0)=\tilde{\Omega}_0+\varepsilon b$ and  $|b|=1$, such that  $P^{-1}(\partial\tilde{\Omega}_{\varepsilon}(0))$ is regular, then by the inverse mapping there exists a local in time smooth solution $\{{\Omega}_{\varepsilon}(t), {u}_{\varepsilon}(t,\cdot), {p}_{\varepsilon}(t,\cdot), {F}_{\varepsilon}(t,\cdot)\}$, in the non tilde complex plane.
\item Then, for sufficiently small $\varepsilon$, by stability we get\\
\begin{center}
$\textrm{dist}(\partial\tilde{\Omega}_{\varepsilon}(\bar{t}),\partial\tilde{\Omega}(\bar{t}))\leq \varepsilon\quad\textrm{ hence }\quad
P^{-1}(\partial\tilde{\Omega}_{\varepsilon}(\bar{t}))\sim P^{-1}(\partial\tilde{\Omega}(\bar{t}))$
\end{center}
\medskip

\noindent and so $P^{-1}(\partial\tilde{\Omega}_{\varepsilon}(\bar{t}))$ self-intersects.
\item Since $P^{-1}(\tilde{\Omega}_{\varepsilon}(0))$ is regular of type (a) and $P^{-1}(\tilde{\Omega}_{\varepsilon}(\bar{t}))$ is self-intersecting domain of type (c), then there exists a time $t^*_{\varepsilon}$ such that $P^{-1}(\tilde{\Omega}_{\varepsilon}(t^*_{\varepsilon}))$ has a splash singularity.
\end{itemize}

\begin{figure}[htbp]\label{spl}
\centering
\includegraphics[scale=0.8] {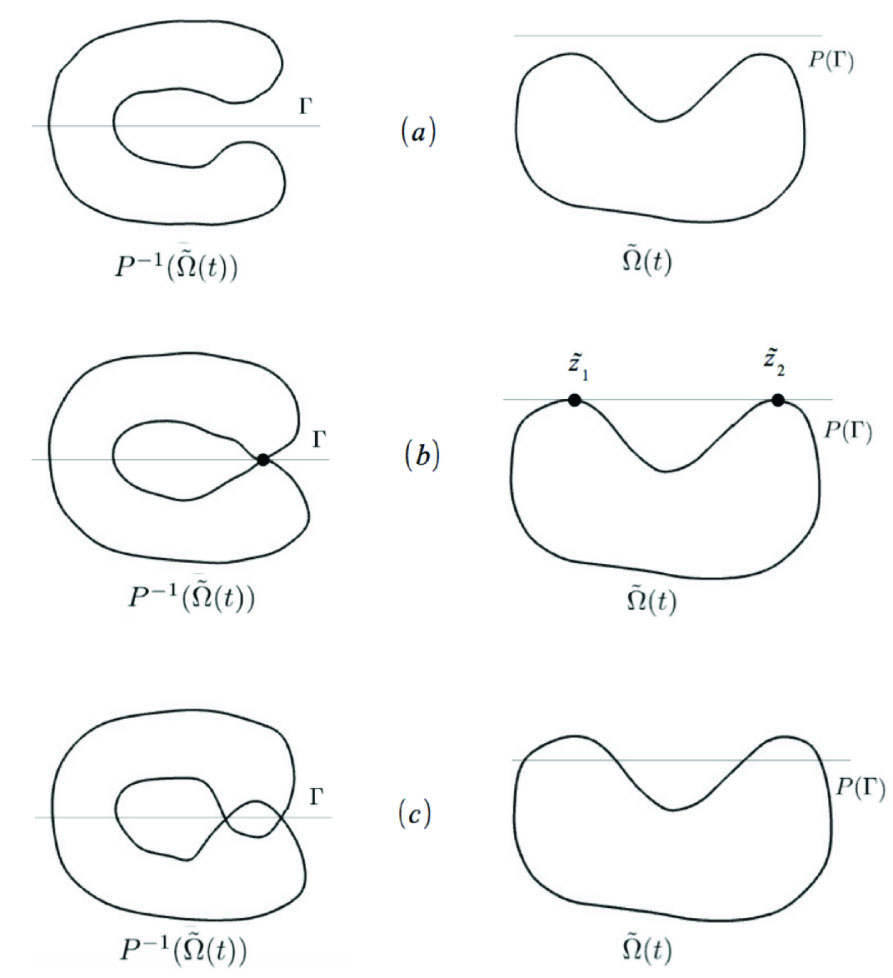}
\caption{Possibilities for $P^{-1}(\tilde{\Omega}(t))$}
\end{figure}
\newpage

\subsection{Outline of the paper}
In Section 2 we introduce the spaces we will use for the estimates and some important lemmas for proving the local existence and the stability estimates. In Section 3 we describe the free boundary system, we define all the variables and the transformations from $\Omega$ into $\tilde{\Omega}$ by using the conformal map and we perform the change from  Eulerian into Lagrangian variables in order to deal with a fixed boundary problem. Section 4 is devoted to solve the Oldroyd-B model through a fixed point argument, using some techniques of \cite{B}. In this section there is also our first important result, that is the proof of the local existence of smooth solutions. In Section 5 we show the stability estimates, necessary for the proof of the existence of splash singularity. In the last section, Section 6, we choose a suitable initial velocity such that the splash occurs, thus we get that even if in viscoelastic model there is the presence of the elastic stress tensor, hidden in the deformation gradient, we can obtain a finite time in which the splash singularity arises.

\section{Conformal and lagrangian transformations}\label{sec2}
\noindent We deal with the following free boundary incompressible viscoelastic fluid model

\begin{equation}\label{ES}
\left\{\begin{array}{lll}
\displaystyle\partial_t F+u\cdot\nabla F=\nabla u F
\hspace{4cm} \textrm{ in }\Omega(t) \\ [1mm]
\displaystyle\partial_t u +u\cdot\nabla u -\Delta u +\nabla p= \dive(FF^{T}) \\[1mm]
\displaystyle\dive u=0,\\ [1mm]
\displaystyle (-p\mathcal{I}+(\nabla u+\nabla u^{T})+(FF^{T}-\mathcal{I}))n=0 \hspace{0.9cm}\textrm{on}\hspace{0.1cm}\partial\Omega(t)\\ [1mm]
\displaystyle u(t)_{|t=0}=u_0,\hspace{0.1cm} F(t)_{|t=0}=F_0\hspace{3.1cm} \textrm{in}\hspace{0.3cm}\Omega_0, \\
\end{array}\right.
\end{equation}
\medskip

\noindent where the domain $\Omega(t)=X(t,\Omega_0)$ moves according to the flux, which solves 
\begin{equation}
\left\{\begin{array}{lll}
\displaystyle\frac{d}{dt}{X}(t,\alpha)=u(t,X(t,\alpha))\\ [3mm]
X(0,\alpha)=\alpha,
\end{array}\right.
\end{equation}
\noindent  and $\Omega_0$, $u_0$ and $F_0$ are prescribed initial data which satisfy the compatibility conditions

\begin{equation}\label{compatibility}
\left\{\begin{array}{lll}
\displaystyle \dive{u_0}=0, \hspace{0.1cm} \det{F_0}=1  \hspace{3.35cm}\textrm{ in } \Omega_0 \\ [1mm]
\displaystyle n_0^{\perp}((\nabla u_0 +\nabla  u_0^T) + (F_0 F_0^T -\mathcal{I}))n_0=0 \hspace{0.4cm} \textrm{ on } \partial\Omega_0.\\
\end{array}\right.
\end{equation}
\medskip

\noindent Let us apply the conformal map $P$ and the change of coordinates from $\Omega\rightarrow\tilde{\Omega}=P(\Omega)$. The trasformed velocity field is defined by
\medskip

$$\tilde{u}(t,\tilde{X})=u(t,P^{-1}(\tilde{X})), \quad\textrm{hence}\quad  u(t,X)=\tilde{u}(t,P(X)).$$

\medskip

\noindent Similarly for the deformation gradient $F$ we have 

$$\tilde{F}(t,\tilde{X})=F(t,P^{-1}(\tilde{X})),\quad\textrm{hence}\quad F(t,X)=\tilde{F}(t,P(X)).$$
\medskip

\begin{remark}
\noindent Defining $J^P_{kj}=\partial_{X_j}P_k(P^{-1}(\tilde{X}))$, for the derivatives we have

$$\partial_{X_j} u_i(t,X)=\partial_{\tilde{X}_k}\tilde{u}_i(t,P(X) )\partial_{X_j} P_k(X)$$
and therefore
\begin{equation}\label{remark1}
\partial_{X_j} u_i(t,P^{-1}(\tilde{X}))=J^P_{kj}\partial_{\tilde{X}_k}\tilde{u}_i(t,\tilde{X}).
\end{equation}
\end{remark}
\medskip

\noindent We start with the transformation in conformal coordinates

\begin{lemma}
Let $P$ the conformal map described above and $\displaystyle Q^2=\left|\frac{dP}{dz}(P^{-1}(\tilde{X}))\right|^2$. Under this transformation the system (\ref{ES}) becomes

\begin{equation}\label{CES}
\left\{\begin{array}{lll}
\displaystyle\partial_t \tilde{F}+(J^P\tilde{u}\cdot\nabla_{\tilde{X}})\tilde{ F}=\nabla_{\tilde{X}}\tilde{u}J^P \tilde{F} \hspace{5.2cm} \textrm{ in } \tilde{\Omega}(t) \\[1mm]
\displaystyle\partial_t\tilde{ u}+(J^P\tilde{u}\cdot\nabla_{\tilde{X}})\tilde{ u} -Q^2\Delta\tilde{ u} +(J^P)^T\nabla_{\tilde{X}}\tilde{ p}= (J^P\tilde{F}\cdot\nabla_{\tilde{X}})\tilde{F}\\[1mm]
\displaystyle \trace(\nabla\tilde{u}J^P)=0\\ [1mm]
\displaystyle (-\tilde{p}\mathcal {I}+(\nabla\tilde{ u}J^P+(\nabla\tilde{ u}J^P)^{T})+(\tilde{F}\tilde{F}^{T}-\mathcal{I}))(J^P)^{-1}\tilde{n}=0 \hspace{1.35cm}\textrm{ on }\partial\tilde{\Omega}(t)\\[1mm]
\displaystyle \tilde{u}(t)_{|t=0}=\tilde{u}_0,\hspace{0.1cm} \tilde{F}(t)_{|t=0}=\tilde{F}_0\hspace{6cm}\textrm{ in } \tilde{\Omega}_0.
\end{array}\right.
\end{equation}
\end{lemma}

\begin{proof}
\noindent The proof follows easily from (\ref{remark1}). For instance $(u\cdot\nabla) u$ becomes\\ $(J^P\tilde{u}\cdot\nabla_{\tilde{X}})\tilde{u}$, and the same for all the other terms with derivatives. The most difficult term is $\Delta u$, by a direct calculation and by using that $P$ satisfies the Cauchy-Riemann equations we get $Q^2\Delta\tilde{u}$. Thus we obtain \eqref{CES}.
\end{proof}
\bigskip

\noindent The next step is to pass from Eulerian to Lagrangian coordinates thus from a free boundary problem to a fixed boundary problem. First of all, we look at the equation for the flux 
\begin{equation}\label{LagFlux}
\left\{\begin{array}{lll}
\displaystyle \frac{d}{dt}\tilde{X}(t,\tilde{\alpha})=J^P(\tilde{X}(t,\tilde{\alpha}))\tilde{u}(t,\tilde{X}(t,\tilde{\alpha}))\hspace{1cm}\textrm{in}\hspace{0.3cm}\tilde{\Omega}(t)\\[3mm]
\displaystyle\tilde{X}(0,\tilde{\alpha})=\tilde{\alpha}\hspace{5.15cm} \textrm{in}\hspace{0.3cm}\tilde{\Omega}(0)
\end{array}\right.
\end{equation}

\noindent The Lagrangian variables are defined as follows
\medskip

\begin{equation}\label{lagvar}
\left\{\begin{array}{lll}
\tilde{v}(t,\tilde{\alpha})=\tilde{u}(t,\tilde{X}(t,\tilde{\alpha}))\\[2mm]
\tilde{q}(t,\tilde{\alpha})=\tilde{p}(t,\tilde{X}(t,\tilde{\alpha}))\\[2mm]
\tilde{G}(t,\tilde{\alpha})=\tilde{F}(t,\tilde{X}(t,\tilde{\alpha})).
\end{array}\right.
\end{equation}
\medskip

\begin{remark}
\noindent Defining by $\tilde{\zeta}_{lj}$ the $lj$-th element of $(\nabla_{\tilde{\alpha}}\tilde{X})^{-1}$, for the derivatives we have 

\begin{equation}\label{remark2}
\partial_{\tilde{X}_j}\tilde{u}_i=\tilde{\zeta}_{lj}\partial_l\tilde{v}_i, 
\end{equation}

\noindent where the derivatives are with respect to $\tilde{\alpha}$.
\end{remark}
\medskip

\begin{lemma}
Under the change of coordinates (\ref{lagvar}), the system (\ref{CES}) on the fixed domain $[0,T]\times \tilde{\Omega}_0$  becomes 

\begin{equation}\label{CLS}
\left\{ \begin{array}{lll}
\displaystyle\partial_t\tilde{G}=J^P(\tilde{X})\tilde{\zeta}\nabla_{\tilde{\alpha}}\tilde{v}\tilde{G} \\[3mm]

\displaystyle \partial_t \tilde{v}-Q^2(\tilde{X})\tilde{\zeta}\nabla_{\tilde{\alpha}}(\tilde{\zeta}\nabla_{\tilde{\alpha}}\tilde{v})+(J^P(\tilde{X}))^T\tilde{\zeta}\nabla_{\tilde{\alpha}}\tilde{q}=J^P(\tilde{X})\tilde{G}\tilde{\zeta}\nabla_{\tilde{\alpha}}\tilde{G} \\[3mm]

\displaystyle\trace(\nabla_{\tilde{\alpha}}\tilde{v}(\nabla_{\tilde{\alpha}}\tilde{X})^{-1}A(\tilde{X}))=0 \\[3mm]

\displaystyle (-\tilde{q}\mathcal{I}+((\nabla_{\tilde{\alpha}}\tilde{v}(\nabla_{\tilde{\alpha}}\tilde{X})^{-1}J^P(\tilde{X}))+(\nabla_{\tilde{\alpha}}\tilde{v}(\nabla_{\tilde{\alpha}}\tilde{X})^{-1}J^P(\tilde{X}))^{T}+\\[2mm]
\hspace{4.4cm}+(\tilde{G}\tilde{G}^{T}-\mathcal{I}))(J^P)^{-1}(\tilde{X})\nabla_{\Lambda}\tilde{X}\tilde{n_0}=0 \\[3mm]
\tilde{v}(\tilde{\alpha},0)=\tilde{v}_0(\tilde{\alpha})=\tilde{u}_0(\tilde{\alpha}),\hspace{0.1cm} \tilde{G}(\tilde{\alpha},0)=\tilde{G}_0(\tilde{\alpha})=\tilde{F}_0(\tilde{\alpha}).
\end{array}\right.
\end{equation}
\medskip

\noindent where $\nabla_{\Lambda}\tilde{X}=-\Lambda\nabla\tilde{X}\Lambda$, with $\Lambda=\left(\begin{matrix} 0 & -1 \\ 1& 0\end{matrix}\right)$,  this is due to the fact that $\tilde{n}=-\Lambda J^P_{|\partial\tilde{\Omega}(t)}\Lambda n$.
\end{lemma}
\medskip

\begin{proof}

\noindent We only consider the equation for the deformation gradient, in more details the transformation of $\partial_t\tilde{F}$. For this term we have
\begin{align*}
\partial_t \tilde{G}(t,\tilde{\alpha})&=\frac{d}{dt}\tilde{F}(t,\tilde{X}(t,\tilde{\alpha}))=\partial_t\tilde{F}(t,\tilde{X}(t,\tilde{\alpha}))+\nabla_{\tilde{X}}\tilde{F}(t,\tilde{X}(t,\tilde{\alpha}))\frac{d}{dt}\tilde{X}(t,\tilde{\alpha})\\[2mm]
&=\partial_t\tilde{F}(t,\tilde{X}(t,\tilde{\alpha}))+\nabla_{\tilde{X}}\tilde{F}(t,\tilde{X}(t,\tilde{\alpha}))J^P(\tilde{X}(t,\tilde{\alpha}))\tilde{u}(t,\tilde{X}(t,\tilde{\alpha})).
\end{align*}
\medskip

\noindent The same is for $\partial_t\tilde{v}(t,\tilde{\alpha})$. For the other  terms when a derivative appears we have to apply (\ref{remark2}). Then  we get  immediately the system in conformal lagrangian coordinate (\ref{CLS}).
\end{proof}

\section{Function spaces and preliminary lemmas}\label{sec3}
\noindent In this section we want to introduce the \textit{Beale spaces} \cite{B} that we will use in the successive sections. 

\begin{align*}
&\mathcal{K}^s_{(0)}([0,T]; \Omega)= L^2([0,T]; H^s( \Omega))\cap H^{\frac{s}{2}}_{(0)}([0,T]; L^2( \Omega)),\\[2mm]
&\mathcal{K}^s_{pr(0)}([0,T]; \Omega)=\{q\in L^{\infty}([0,T]; \dot{H}^1( \Omega)): \nabla q\in\mathcal{K}^{s-1}_{(0)}([0,T]; \Omega), q\in\mathcal{K}^{s-\frac{1}{2}}_{(0)}([0,T]; \partial\Omega)\},\\[2mm]
&\mathcal{\bar{K}}^{s}_{(0)}([0,T]; \Omega)= L^2([0,T];H^s(\Omega)) \cap H^{\frac{s+1}{2}}_{(0)}([0,T]; H^{-1}(\Omega)),\\[2mm]
&\mathcal{F}^{s+1,\gamma}([0,T]; \Omega)= L^{\infty}_{\frac{1}{4}}([0,T]; H^{s+1}( \Omega)) \cap H^2_{(0)}([0,T]; H^{\gamma}( \Omega)),\hspace{0.3cm}\textrm{for}\hspace{0.2cm} s-1-\varepsilon<\gamma<s-1,
\end{align*}
\noindent with 

$$\|f\|_{L^{\infty}_{\frac{1}{4}}H^s}=\sup\limits_{t\in [0,T]}t^{-\frac{1}{4}}\|f(t)\|_{H^s}.$$
\medskip

\noindent In the previous definition there is the presence of the Sobolev space with fractional derivatives in time, $H^s_{(0)}([0,T])$, for $0<s<1$. In \cite{B}, Beale defines it as the interpolated space between $L^2([0,T])$ and $H^1_{(0)}([0,T])$, through the operator $S=(1-\partial_t^2)$, with domain $D(S)=\{v\in H^2([0,T]): v(0)=0, \partial_t v(T)=0\}$ and the $H^s_{(0)}([0,T])-$norm is the graph norm of the operator $S^{\frac{s}{2}}$. In particular $v\in  H^s_{(0)}([0,T])$ if $v\in L^2([0,T])$ and

\begin{align*}
\|v\|_{H^s_{(0)}}^2=\sum_{n=0}^{\infty}\left(\int_0^T v(t)\sin\left(\frac{(2n+1)\pi}{2T}t\right)\sqrt\frac{2}{T} \,dt\right)^2\left(\frac{(2n+1)\pi}{2T}\right)^{2s}<\infty.
\end{align*}

\medskip

\noindent For larger exponents, the definition of $H^{s+m}_{(0)}([0,T])$, for $0<s<1$ and $m\in\mathbb{N}$ is 

\begin{center}
$\{v\in H^{s+m}([0,T]):\partial_t^k v(0)=0, k=0,1,\ldots,m-1\},\hspace{0.3cm}\textrm{with}\hspace{0.3cm}\partial_t^m v\in H^s_{(0)}([0,T])$.
\end{center}
\medskip

\noindent This space has the following norm

$$\|v\|_{H^{s+m}_{(0)}}^2=\sum_{k=0}^{m-1}\|\partial_t^k v\|_{ L^2([0,T])}^2+\|\partial_t^m v\|_{ H^s_{(0)}([0,T])}^2.$$
\bigskip

\noindent These spaces are important for the use of the following embedding theorems and interpolation estimates, in order to get constants independent of time. For details, see \cite{B} and \cite{LM}.
\bigskip

\begin{lemma}\label{extentheorem}
Let $B$ a Hilbert space
\begin{enumerate}
\item For $s\geq 0$, there is a bounded extension operator from $H^s((0,T);B)\rightarrow H^s((-\infty,\infty);B)$.
\item For $0\leq s\leq 2$ , $s-\frac{1}{2}$ is not an integer, there is an extension operator from
$$\left\lbrace v\in H^s((0,T);B); \partial_t^k v(0)=0, 0\leq k < s-\frac{1}{2}\right\rbrace \rightarrow H^s((-\infty,\infty);B)$$
with a norm bounded independently on $T$. Moreover, if $Ev$ is the extention of $v$ then 

$$\|Ev\|_{ H^s((-\infty,\infty);B)}\leq C \|v\|_{ H_{(0)}^s((0,T);B)}$$
\end{enumerate}
\end{lemma}
\bigskip

\begin{lemma}\label{lem1}
Suppose $0\leq r\leq 4$.
\begin{enumerate}
\item The Identity extends to a bounded operator
$$\mathcal{K}^r((0,T);\Omega)\rightarrow H^p(0,T)H^{r-2p}(\Omega),$$
$p\leq \frac{r}{2}$.
\item If $r$ is not an odd integer, the restriction of this operator to the subspace with $\partial_t^k v(0)=0$, $0\leq k< \frac{r-1}{2}$ is bounded independently on $T$, indeed
\begin{equation*}
\|v\|_{H^p_{(0)}H^{r-2p}}\leq C \|v\|_{\mathcal{K}^r_{(0)}}.
\end{equation*}
\end{enumerate}
\end{lemma}
\bigskip

\begin{lemma}\label{lem2}
Let $\bar{T}>0$ be arbitrary, $B$ a Hilbert space and choose $T\leq\bar{T}$.
\begin{enumerate}
\item 
For $v\in L^2((0,T); B)$, we define $V\in H^1((0,T); B)$ by
$$V(t)=\int_0^t v(\tau)\, d\tau.$$
For $0<s<\frac{1}{2}$ and $0\leq\varepsilon<s$, then the map $v\rightarrow V$ is a bounded operator from $H^s((0,T); B)$ to $H^{s+1-\varepsilon}((0,T);B)$, and

$$\|V\|_{H^{s+1-\varepsilon}((0,T); B)}\leq C_0T^{\varepsilon}\|v\|_{H^s((0,T); B)},$$
where $C_0$ is independent of $T$ for $0< T\leq\bar {T}$.
\medskip

\item For $\frac{1}{2}< s<1$, we impose $v(0)=0$ and $0\leq\varepsilon<s$. Then $v\rightarrow V$ is a bounded operator from $H^s_{(0)}((0,T); B)$ to $H^{s+1-\varepsilon}_{(0)}((0,T); B)$ and 

$$\|V\|_{H^{s+1-\varepsilon}_{(0)}((0,T); B)}\leq C_0T^{\varepsilon}\|v\|_{H^s_{(0)}((0,T); B)},$$
where $C_0$ is independent of $T$ for $0< T\leq\bar{T}$.
\end{enumerate}
\end{lemma}
\bigskip

\begin{lemma}\label{lem3}
Suppose $r>1$ and $r\geq s\geq 0.$ If $v\in H^r(\Omega)$ and $w\in H^s(\Omega)$, then $vw\in H^s(\Omega)$ and
$$\|vw\|_{H^s}\leq C\|v\|_{H^r} \|w\|_{H^s}.$$
\end{lemma}
\bigskip

\begin{lemma}\label{lem4}
If $v\in H^{\frac{1}{q}}$ and $w\in H^{\frac{1}{p}}$ with $\frac{1}{p}+\frac{1}{q}=1$ and $1<p<\infty$ then 
$$\|vw\|_{L^2}\leq C\|v\|_{H^{\frac{1}{q}}}\|w\|_{H^{\frac{1}{p}}}.$$
\end{lemma}
\bigskip

\begin{lemma}\label{lem5}
Suppose  $B, Y, Z$ are Hilbert spaces, and $M:B\times Y\rightarrow Z$ is a bounded, bilinear, multiplication operator. Suppose $w\in H^s((0,T); B)$ and $v\in H^s((0,T); Y)$, where $s>\frac{1}{2}$. If $vw$ is defined by $M(v,w)$, then $vw\in H^s((0,T); Z)$ and the following hold
\begin{enumerate}
\item $$\|vw\|_{H^s((0,T); Z)}\leq C\|v\|_{H^s((0,T); Y)} \|w\|_{H^s((0,T); B)}.$$
\medskip

\item In addition, if $s\leq 2$ and $\partial_t^k v(0)=\partial_t^k w(0)=0$, $0\leq k<s-\frac{1}{2}$ and $s-\frac{1}{2}$ is not an integer, then the constant $C$ in $(1)$ can be chosen independently on $T$. Indeed
$$\|vw\|_{H^s_{(0)}((0,T); Z)}\leq C \|v\|_{H^s_{(0)}((0,T); Y)}\|w\|_{H^s_{(0)}((0,T); B)}.$$

\end{enumerate}
\end{lemma}
\bigskip

\begin{lemma}\label{lem6}
For $2<s<\frac{5}{2}$, $\varepsilon, \delta$ positive and small enough and $v\in\mathcal{F}^{s+1,\gamma}$ the following estimates hold
\medskip

\begin{enumerate}
\item $\|v\|_{H_{(0)}^{\frac{s+1}{2}}H^{1-\varepsilon}}\leq C\|v\|_{\mathcal{F}^{s+1,\gamma}}$,\\
\item  $\|v\|_{H_{(0)}^{\frac{s+1}{2}+\varepsilon}H^{1+\delta}}\leq C\|v\|_{\mathcal{F}^{s+1,\gamma}}$,\\
\item  $\|v\|_{H_{(0)}^{\frac{s-1}{2}+\varepsilon}H^{2+\delta}}\leq C\|v\|_{\mathcal{F}^{s+1,\gamma}}$,\\
\item  $\|v\|_{H_{(0)}^{\frac{s}{2}-\frac{1}{4}+\varepsilon}H^{2+\delta}}\leq C\|v\|_{\mathcal{F}^{s+1,\gamma}}$,\\
\item  $\|v\|_{H_{(0)}^1H^{s-1}}\leq C\|v\|_{\mathcal{F}^{s+1,\gamma}}$,\\
\item  $\|v\|_{H_{(0)}^{\frac{1}{2}+2\varepsilon}H^{s}}\leq C\|v\|_{\mathcal{F}^{s+1,\gamma}}$.
\end{enumerate}
\end{lemma}
\medskip

\begin{remark}
We notice that in the same way as Lemma \ref{lem6} we can deduce the embeddings for the space $\mathcal{F}^{s,\gamma-1}$. For instance, $\mathcal{F}^{s,\gamma-1}\subset H_{(0)}^{\frac{s-1}{2}+\delta}H^{1+\eta}$, $\mathcal{F}^{s,\gamma-1}\subset H_{(0)}^{\frac{s}{2}-\frac{1}{4}+\delta}H^{1+\eta}$, for $\delta, \eta>0$ and small enough.
\end{remark}

\bigskip

\begin{lemma}\label{parabolic-trace}
Let $\Omega$ be a bounded set with a sufficient smooth boundary then the following trace theorems hold
\begin{enumerate}
\item Suppose $\frac{1}{2}< s\leq 5$. The mapping $v\rightarrow\partial_n^j v$ extends to a bounded operator\\ $\mathcal{K}^s([0,T];\Omega)\rightarrow \mathcal{K}^{s-j-\frac{1}{2}}([0,T];\partial\Omega)$, where $j$ is an integer $0\leq j<s-\frac{1}{2}$. The mapping $v\rightarrow \partial_t^k v(\alpha,0)$ extends to a bounded operator 
$\mathcal{K}^s([0,T];\Omega)\rightarrow H^{s-2k-1}(\Omega)$, if $k$ is an integer $ 0\leq k<\frac{1}{2}(s-1)$.

\item  Suppose $\frac{3}{2}<s<5$, $s\neq 3$ and $s-\frac{1}{2}$ not an integer. Let
$$\mathcal{W}^s=\prod_{0\leq j\leq s-\frac{1}{2}}\mathcal{K}^{s-j-\frac{1}{2}}([0,T];\partial\Omega)\times \prod_{0\leq k< \frac{s-1}{2}}H^{s-2k-1}(\Omega),$$
and let $\mathcal{W}^s_0$ the subspace consisting of $\{a_j,w_k\}$, which are the traces described in the previous point, so that $ \partial_t^k a_j(\alpha,0)=\partial_n^j w_k(\alpha)$, $\alpha\in\partial\Omega$, for $j+2k<s-\frac{3}{2}$. Then the traces in the previous point form a bounded operator $\mathcal{K}^s([0,T];\Omega)\rightarrow \mathcal{W}^s_0$ and this operator has a bounded right inverse.
\end{enumerate}
\end{lemma}
\bigskip

\section{Local existence of smooth solutions for the system \eqref{CLS}} \label{sec4}

\noindent In order to prove a local existence result for the system \eqref{CLS}, we use the Picard iterations and we show their contractivity. We separate the iteration for $\tilde{v}$ from the iteration for $\tilde{G}$. Thus, for the velocity we have

%

\begin{equation}\label{v-iterativesys}
\left\{ \begin{array}{lll}
\displaystyle \partial_t\tilde{v}^{(n+1)}-Q^2\Delta\tilde{v}^{(n+1)}+(J^P)^T\nabla\tilde{q}^{(n+1)}=\tilde{f}^{(n)}\\[3mm]

\displaystyle\trace(\nabla\tilde{v}^{(n+1)}J^P)=\tilde{g}^{(n)}\\[3mm]

\displaystyle [-\tilde{q}^{(n+1)}\mathcal{I}+((\nabla\tilde{v}^{(n+1)}J^P)+(\nabla\tilde{v}^{(n+1)}J^P)^T)](J^P)^{-1}\tilde{n}_0=\tilde{h}^{(n)}\\[3mm]
\displaystyle\tilde{v}(0,\tilde{\alpha})=\tilde{v}_0(\tilde{\alpha}),
\end{array}\right.
\end{equation}
\bigskip

\noindent where $\tilde{f}^{(n)},\tilde{ g}^{(n)},\tilde{ h}^{(n)}$ collect all the terms at $n^{th}$ time step, namely
\begin{align*} 
&\tilde{f}^{(n)}=-Q^2\Delta\tilde{v}^{(n)}+(J^P)^T\nabla\tilde{q}^{(n)}+ Q^2(\tilde{X}^{(n)})\tilde{\zeta}^{(n)}\nabla(\tilde{\zeta}^{(n)}\nabla\tilde{v}^{(n)})-J^P(\tilde{X}^{(n)})^T\tilde{\zeta}^{(n)}\nabla\tilde{q}^{(n)}\\ 
&\displaystyle\hspace{1.1cm}+J^P(\tilde{X}^{(n)})\tilde{G}^{(n)}\tilde{\zeta}^{(n)}\nabla\tilde{G}^{(n)},\\[3mm]
&\tilde{g}^{(n)}=\trace(\nabla\tilde{v}^{(n)}J^P)-\trace(\nabla\tilde{v}^{(n)}\tilde{\zeta}^{(n)}J^P(\tilde{X}^{(n)})),\\[3mm]
&\tilde{h}^{(n)}=-\tilde{q}^{(n)}(J^P)^{-1}\tilde{n}_0+\tilde{q}^{(n)}(J^P(\tilde{X}^{(n)}))^{-1}\nabla_{\Lambda}\tilde{X}^{(n)}\tilde{n}_0+
((\nabla\tilde{v}^{(n)}J^P)+(\nabla\tilde{v}^{(n)}J^P)^T)(J^P)^{-1}\tilde{n_0}\\
&\hspace{1cm}-((\nabla\tilde{v}^{(n)}\tilde{\zeta}^{(n)}J^P(\tilde{X}^{(n)}))+(\nabla\tilde{v}^{(n)}\tilde{\zeta}^{(n)}J^P(\tilde{X}^{(n)}))^T)(J^P(\tilde{X}^{(n)}))^{-1}\nabla_{\Lambda}\tilde{X}^{(n)}\tilde{n}_0\\
&\hspace{1cm}-(\tilde{G}^{(n)}\tilde{G}^{T(n)}-\mathcal{I})(J^P(\tilde{X}^{(n)}))^{-1}\nabla_{\Lambda}\tilde{X}^{(n)}\tilde{n}_0,
\end{align*}

\bigskip

\noindent while for the deformation gradient,  we consider the following ODE 

\begin{equation}\label{G-iterativesys}
\left\{ \begin{array}{lll}
\displaystyle\partial_t\tilde{G}^{(n+1)}(t,\tilde{\alpha})=J^P(\tilde{X}^{(n)}(t,\tilde{\alpha}))\tilde{\zeta}^{(n)}\nabla\tilde{v}^{(n)}\tilde{G}^{(n)}\\[3mm]
\displaystyle \tilde{G}(0,\tilde{\alpha})=\tilde{G}_0(\tilde{\alpha}).
\end{array}\right.
\end{equation}
\bigskip

\noindent Moreover the flux satisfies

\begin{equation}\label{X-iterativesys}
\left\{\begin{array}{lll}
\displaystyle \frac{d}{dt}\tilde{X}^{(n+1)}(t,\tilde{\alpha})=J^P(\tilde{X}^{(n)}(t,\tilde{\alpha})) \tilde{v}^{(n)}(t,\tilde{\alpha})\\[3mm]
\displaystyle\tilde{X}(0,\tilde{\alpha})=\tilde{\alpha}\hspace{6cm} \textrm{in}\hspace{0.3cm}\tilde{\Omega}_0.
\end{array}\right.
\end{equation}
\medskip

\noindent We study separately the systems \eqref{v-iterativesys}, \eqref{G-iterativesys} and \eqref{X-iterativesys}. For the linear system $(\ref{v-iterativesys})$, we use the methods of \cite{B} adapted to the conformal coordinates as in \cite{CCFGG2}. Consequently we study the following system

\begin{equation}\label{v-homsys}
\left\{\begin{array}{lll}
\displaystyle \partial_t \tilde{v}-Q^2\Delta\tilde{v}+(J^P)^T \nabla\tilde{q}=\tilde{f} \hspace{4.5cm}\textrm{in}\hspace{0.3cm}(0,T)\times\tilde{\Omega}_0\\[2mm]
\displaystyle \trace(\nabla \tilde{v} J^P)=\tilde{g}\hspace{6.8cm}\textrm{in}\hspace{0.3cm}(0,T)\times\tilde{\Omega}_0\\[2mm]
\displaystyle (-\tilde{q}\mathcal{I}+(\nabla\tilde{v} J^P)+(\nabla\tilde{v} J^P)^T)\frac{(J^P)^T}{Q^2} \tilde{n}=\tilde{h}\hspace{2.3cm}\textrm{on}\hspace{0.3cm}(0,T)\times\partial\tilde{\Omega}_0\\[2mm]
\displaystyle \tilde{v}(0,\tilde{\alpha})=\tilde{v}_0(\tilde{\alpha})\hspace{6.7cm}\textrm{on}\hspace{0.3cm}\{t=0\}\times\tilde{\Omega}_0,
\end{array}\right.
\end{equation}
\medskip

\noindent supplemented by the following compatibility conditions for the initial data

\begin{equation}\label{compcond}
\left\{\begin{array}{lll}
\displaystyle \trace(\nabla \tilde{v}_0 J^P)=\tilde{g}(0)\hspace{7.7cm}\textrm{in}\hspace{0.1cm}\tilde{\Omega}_0\\[3mm]
\displaystyle ((J^P)^{-1}\tilde{n})^{\perp}(\nabla\tilde{ v}_0 J^P+(\nabla \tilde{v}_0 J^P)^T)(J^P)^{-1}\tilde{n}=\tilde{h}(0)((J^P)^{-1}\tilde{n})^{\perp}\hspace{0.5cm}\textrm{on}\hspace{0.1cm}\partial\tilde{\Omega}_0
\end{array}\right.
\end{equation}
\medskip

\noindent We define the following functional space of the solution $X_0$, namely
\begin{align*}
&X_0:=\left\lbrace (\tilde{v},\tilde{q})\in\mathcal{K}^{s+1}_{(0)}\times \mathcal{K}^{s}_{pr(0)}\right\rbrace,
\end{align*}
the function space of the data $Y_0$, namely
\begin{align*}
& Y_0:=\{ (\tilde{f},\tilde{g},\tilde{h})\in\mathcal{K}^{s-1}_{(0)}\times \mathcal{\bar{K}}^{s}_{(0)}\times \mathcal{K}^{s-\frac{1}{2}}_{(0)}([0,T];\partial\Omega):\hspace{0.1cm}(\ref{compcond})\hspace{0.1cm} \textrm{are satisfied}\},
\end{align*}
\medskip

\noindent and a linear operator  $L: X_0\rightarrow Y_0$, related to the system (\ref{v-homsys}) by 
\begin{equation}\label{L}
L(\tilde{v},\tilde{q})=(\tilde{f},\tilde{g},\tilde{h},\tilde{v}_0).
\end{equation}
\noindent The well-posedness of the system (\ref{v-homsys}) is guaranteed by the invertibility of the operator $L$, proved in \cite{B} and \cite{CCFGG2}.

\begin{lemma}\label{invL}
The operator $L$ defined in (\ref{L}) is invertible for $2<s<\frac{5}{2}$. Moreover, for any $0<T<\bar{T}$, the bound of $\|L^{-1}\|$ does not depend on $T$. Precisely, the following estimate holds
\begin{equation*}
\|(\tilde{v},\tilde{q})\|_{X_0}\leq C \|(\tilde{f},\tilde{g},\tilde{h}, 0)\|_{Y_0}.
\end{equation*}
\end{lemma}
\medskip

\noindent Now, in order to apply the Lemma \ref{invL}, which give us estimates independent of time, we need a modification of the velocity and pressure to have $(\tilde{v},\tilde{q})\in X_0$ and $(\tilde{f},\tilde{g},\tilde{h})\in Y_0$. For that reason we consider an approximation of the velocity

$$\phi=\tilde{v}_0+t(Q^2\Delta\tilde{v}_0-(J^P)^T\nabla\tilde{q}_{\phi}+J^P\tilde{G}_0\nabla\tilde{G}_0)=\tilde{v}_0+t\hat{\phi},$$

\noindent we choose $\tilde{q}_{\phi}$ in such a way that $\partial_t\tilde{v}^{(n)}(0)=\partial_t\phi(0)$, for all $n$. Specifically $\tilde{q}_{\phi}$ has to satisfies

\begin{equation}
\left\{\begin{array}{lll}
-Q^2\Delta\tilde{q}_{\phi}=\trace(\nabla\tilde{v}_0J^P\nabla\tilde{v}_0J^P)-\trace(\nabla(J^P\tilde{G}_0\nabla\tilde{G}_0))\hspace{2.5cm}\textrm{in}\hspace{0.3cm}\tilde{\Omega}_0\\[3mm]
\tilde{q}_{\phi}(J^P)^{-1}\tilde{n}_0=(J^P\nabla\tilde{v}_0+(J^P\nabla\tilde{v}_0)^T+(\tilde{G}_0\tilde{G}_0^T-\mathcal{I}))(J^P)^{-1}\tilde{n}_0\hspace{0.8cm}\textrm{on}\hspace{0.3cm}\partial\tilde{\Omega}_0
\end{array}\right.
\end{equation}
\medskip

\noindent Now, we define the new velocity field
\begin{equation}\label{newvelocity}
\tilde{w}^{(n)}=\tilde{v}^{(n)}-\phi,
\end{equation}

\noindent and the system  $(\ref{v-iterativesys})$ becomes the following

\begin{equation}\label{systemW}
\left\{\begin{array}{lll}
\displaystyle \partial_t\tilde{w}^{(n+1)}-Q^2\Delta\tilde{w}^{(n+1)}+(J^P)^T\nabla\tilde{q}_w^{(n+1)}=\tilde{f}^{(n)}-\partial_t\phi \\[1mm]
\hspace{1cm}+Q^2\Delta\phi-(J^P)^T\nabla\tilde{q}_{\phi}\\ [3mm]
\displaystyle\trace(\nabla\tilde{w}^{(n+1)}J^P)=\tilde{g}^{(n)}-\trace(\nabla\phi J^P)\\ [3mm]
\displaystyle [-\tilde{q}_w^{(n+1)}\mathcal{I}+((\nabla\tilde{w}^{(n+1)}J^P)+(\nabla\tilde{w}^{(n+1)}J^P)^T)](J^P)^{-1}\tilde{n_0}=\\[1mm]
\hspace{1cm}=\tilde{h}^{(n)}+\tilde{q}_{\phi}(J^P)^{-1}\tilde{n}_0-((\nabla\phi J^P)+(\nabla\phi J^P)^T)(J^P)^{-1}\tilde{n}_0\\ [3mm]
\displaystyle \tilde{w}^{(n+1)}_{|t=0}=0,
\end{array}\right.
\end{equation}
\medskip
where

\begin{align*} 
&\tilde{f}^{(n)}=-Q^2\Delta\tilde{w}^{(n)}+(J^P)^T\nabla\tilde{q}_w^{(n)}+ Q^2(\tilde{X}^{(n)})\tilde{\zeta}^{(n)}\nabla(\tilde{\zeta}^{(n)}\nabla\tilde{w}^{(n)})-J^P(\tilde{X}^{(n)})^T\tilde{\zeta}^{(n)}\nabla\tilde{q}_w^{(n)}\\ 
&\displaystyle\hspace{1cm}+J^P(\tilde{X}^{(n)})\tilde{G}^{(n)}\tilde{\zeta}^{(n)}\nabla\tilde{G}^{(n)}-Q^2\Delta\phi+(J^P)^T\nabla\tilde{q}_{\phi}+Q^2(\tilde{X}^{(n)})\tilde{\zeta}^{(n)}\nabla(\tilde{\zeta}^{(n)}\nabla\phi)\\
&\displaystyle\hspace{1cm}-J^P(\tilde{X}^{(n)})\tilde{\zeta}^{(n)}\nabla\tilde{q}_{\phi},\\[3mm]
&\tilde{g}^{(n)}=\trace(\nabla\tilde{w}^{(n)}J^P)-\trace(\nabla\tilde{w}^{(n)}\tilde{\zeta}^{(n)}J^P(\tilde{X}^{(n)}))+\trace(\nabla\phi J^P)-\trace(\nabla\phi\tilde{\zeta}^{(n)}J^P(\tilde{X}^{(n)})),\\[3mm]
&\tilde{h}^{(n)}=-\tilde{q}_w^{(n)}(J^P)^{-1}\tilde{n}_0+\tilde{q}_w^{(n)}(J^P(\tilde{X}^{(n)}))^{-1}\nabla_{\Lambda}\tilde{X}^{(n)}\tilde{n}_0+
((\nabla\tilde{w}^{(n)}J^P)+(\nabla\tilde{w}^{(n)}J^P)^T)(J^P)^{-1}\tilde{n}_0\\
&\hspace{1cm}-((\nabla\tilde{w}^{(n)}\tilde{\zeta}^{(n)}J^P(\tilde{X}^{(n)}))+(\nabla\tilde{w}^{(n)}\tilde{\zeta}^{(n)}J^P(\tilde{X}^{(n)}))^T)(J^P(\tilde{X}^{(n)}))^{-1}\nabla_{\Lambda}\tilde{X}^{(n)}\tilde{n}_0\\
&\hspace{1cm}-(\tilde{G}^{(n)}\tilde{G}^{T(n)}-\mathcal{I})(J^P(\tilde{X}^{(n)}))^{-1}\nabla_{\Lambda}\tilde{X}^{(n)}\tilde{n}_0 -\tilde{q}_{\phi}(J^P)^{-1}\tilde{n}_0+\tilde{q}_{\phi}(J^P(\tilde{X}^{(n)}))^{-1}\nabla_{\Lambda}\tilde{X}^{(n)}\tilde{n}_0\\
&\hspace{1cm}+((\nabla\phi J^P)+(\nabla\phi J^P)^T)(J^P)^{-1}\tilde{n_0}-((\nabla\phi\tilde{\zeta}^{(n)}J^P(\tilde{X}^{(n)}))\\
&\hspace{1cm}+(\nabla\phi\tilde{\zeta}^{(n)}J^P(\tilde{X}^{(n)}))^T)(J^P(\tilde{X}^{(n)}))^{-1}\nabla_{\Lambda}\tilde{X}^{(n)}\tilde{n}_0.
\end{align*}
\medskip

\noindent For the deformation gradient we have 
\begin{equation}\label{newG}
\begin{aligned}
\tilde{G}^{(n+1)}(t,\tilde{\alpha})&=\tilde{G}_0+\int_0^t (J^P(\tilde{X}^{(n)})\tilde{\zeta}^{(n)}\nabla\tilde{w}^{(n)}\tilde{G}^{(n)})(\tau,\tilde{\alpha})\,d\tau \\
&+\int_0^t (J^P(\tilde{X}^{(n)})\tilde{\zeta}^{(n)}\nabla\phi\tilde{G}^{(n)})(\tau,\tilde{\alpha})\,d\tau,
\end{aligned}
\end{equation}
\medskip

\noindent and for the flux we have 
\begin{equation}\label{newX}
\begin{aligned}
\tilde{X}^{(n+1)}(t,\tilde{\alpha})&=\tilde{\alpha}+\int_0^t (J^P(\tilde{X}^{(n)})\tilde{w}^{(n)})(\tau,\tilde{\alpha})\,d\tau+\int_0^t(J^P(\tilde{X}^{(n)})\phi)(\tau,\tilde{\alpha})\,d\tau.
\end{aligned}
\end{equation}
\medskip

\noindent The main theorem of this section is the following local existence theorem.

\begin{theorem}\label{localex}
Let $2<s<\frac{5}{2}$, $1<\gamma<s-1$ . If $\tilde{v}(0)=\tilde{v}_0$ and $\tilde{G}(0)=\tilde{G}_0$ are in $H^k(\tilde{\Omega}_0)$ for $k$ big enough, then there exist a sufficiently small $T$ and a solution $\{\tilde{X}-\hat{X}, \tilde{w}, \tilde{q}_w, \tilde{G}-\hat{G}\}\in\mathcal{F}^{s+1,\gamma}\times\mathcal{K}^{s+1}_{(0)}\times\mathcal{K}^s_{pr(0)}\times\mathcal{F}^{s,\gamma-1}$ on $(0,T]\times\tilde{\Omega}_0$.
\end{theorem}

\noindent In this theorem we introduced $\hat{X}=\tilde{\alpha}+tJ^P\tilde{v}_0$ and $\hat{G}=\tilde{G}_0+tJ^P\nabla\tilde{v}_0\tilde{G}_0$, since by adding these two terms we have $\tilde{X}-\hat{X}$ and $\tilde{G}-\hat{G}$ belong to $\mathcal{F}^{s+1,\gamma}$ and $\mathcal{F}^{s,\gamma-1}$, respectively. In order to prove Theorem \ref{localex}, we have to show the contractivity of the Picard iterations. 

\begin{proposition}(\textbf{Iterative bounds}) \label{fixed point}
For $2<s<\frac{5}{2}$, $1<\gamma<s-1$ and for  $\tilde{X}^{(n)}-\tilde{\alpha}, \tilde{X}^{(n-1)}-\tilde{\alpha}\in B_1$,
$(\tilde{w}^{(n)},\tilde{q}_w^{(n)}), (\tilde{w}^{(n-1)}, \tilde{q}_w^{(n)})\in B_2$ and $\tilde{G}^{(n)}-\tilde{G}_0, \tilde{G}^{(n-1)}-\tilde{G}_0\in B_3$, where $B_1, B_2, B_3$ are balls that we will define later.

\noindent Then it follows

\begin{align*}
&\left\|\tilde{w}^{(n+1)}-\tilde{w}^{(n)}\right\|_{\mathcal{K}^{s+1}_{(0)}}+\left\|\tilde{q}_w^{(n+1)}-\tilde{q}_w^{(n)}\right\|_{\mathcal{K}^{s}_{pr(0)}}+ \left\|\tilde{X}^{(n+1)}-\tilde{X}^{(n)}\right\|_{\mathcal{F}^{s+1,\gamma}}+\left\|\tilde{G}^{(n+1)}-\tilde{G}^{(n)}\right\|_{\mathcal{F}^{s,\gamma-1}}\\
&\hspace{1cm}\leq C(\tilde{v}_0, \tilde{G}_0)T^{\delta}\left( \left\|\tilde{X}^{(n)}-\tilde{X}^{(n-1)}\right\|_{\mathcal{F}^{s+1,\gamma}}+\left\|\tilde{w}^{(n)}-\tilde{w}^{(n-1)}\right\|_{\mathcal{K}^{s+1}_{(0)}}+
\left\|\tilde{q}_w^{(n)}-\tilde{q}_w^{(n-1)}\right\|_{\mathcal{K}^{s}_{pr(0)}}\right.\\
&\hspace{4cm}\left.+ \left\|\tilde{G}^{(n)}-\tilde{G}^{(n-1)}\right\|_{\mathcal{F}^{s,\gamma-1}}\right).
\end{align*}
\end{proposition}
\medskip

\noindent The proof of Proposition \ref{fixed point} is obtained by investigating separately, the equations \eqref{systemW}, \eqref{newG} and \eqref{newX}. The map related to these equations is $\mathcal{L}: X_0\times\mathcal{F}^{s+1,\gamma}\times\mathcal{F}^{s,\gamma-1}\rightarrow X_0\times\mathcal{F}^{s+1,\gamma}\times\mathcal{F}^{s,\gamma-1}$, defined as follows

\begin{align*}
\mathcal{L}\left((\tilde{w}^{(n+1)},\tilde{q}_w^{(n+1)}), \tilde{X}^{(n+1)}-\hat{X}, \tilde{G}^{(n+1)}-\hat{G}\right)=&L^{-1}\left(\tilde{f}^{(n)},\tilde{g}^{(n)},\tilde{h}^{(n)}\right)+\mathcal{D}\left(\tilde{w}^{(n)},\tilde{X}^{(n)}-\hat{X}\right)\\
&+\mathcal{E}\left(\tilde{w}^{(n)},\tilde{X}^{(n)}-\hat{X},\tilde{G}^{(n)}-\hat{G}\right).
\end{align*}

%

\noindent The operators  $\mathcal{D}$ and $\mathcal{E}$ have the following form
\begin{align*}
&\mathcal{D}(\tilde{w}^{(n)},\tilde{X}^{(n)}-\hat{X})=\tilde{X}^{(n+1)}-\hat{X},\\ \\
&\mathcal{E}(\tilde{w}^{(n)},\tilde{X}^{(n)}-\hat{X},\tilde{G}^{(n)}-\hat{G})=\tilde{G}^{(n+1)}-\hat{G},
\end{align*}
\medskip

\noindent Consequently, with the bound of proposition \ref{fixed point} and by applying the contraction mapping principle, we have the following result

\begin{proposition}
For $T$ small enough and a suitable $\delta>0$, $\mathcal{L}$ is a contraction.
\end{proposition}
\bigskip

\subsection{Proof of Proposition \ref{fixed point}}
In order to get this result we study  separately the three equations \eqref{systemW}, \eqref{newG} and \eqref{newX}. 
We rewrite the RHS of system \eqref{systemW} in the following way to apply lemma \ref{invL}.

\begin{align*}
&L(\tilde{w}^{(n+1)},\tilde{q}_w^{(n+1)})=\left(\tilde{f}^{(n)}-\tilde{f}_{G_0}, \bar{g}^{(n)}, \tilde{h}^{(n)}-\tilde{h}_{G_0}\right)+\left(\tilde{f}_{\phi}^L+\tilde{f}_{G_0},\bar{g}_{\phi}^L,\tilde{h}^L_{\phi}+\tilde{h}_{G_0} \right),
\end{align*}
\noindent where $\tilde{f}_{G_0}=J^P\tilde{G}_0\nabla\tilde{G}_0$ and $\tilde{h}_{G_0}=-(\tilde{G}_0\tilde{G}_0^T-\mathcal{I})(J^P)^{-1}\tilde{n}_0$. In addition, we introduce

\begin{align*}
&\tilde{f}^L_{\phi}=-\partial_t\phi+Q^2\Delta\phi-(J^P)^T\nabla\tilde{q}_{\phi},\\[2mm]
&\bar{g}^{(n)}=\tilde{g}^{(n)}+\trace(\nabla\phi\tilde{\zeta}_{\phi}J^P_{\phi})-\trace(\nabla\phi J^P),\\[2mm]
&\bar{g}^L_{\phi}=\tilde{g}^L_{\phi}-\trace(\nabla\phi\tilde{\zeta}_{\phi}J^P_{\phi})+\trace(\nabla\phi J^P),\\[2mm]
&\tilde{g}^L_{\phi}=-\trace(\nabla\phi J^P),\\[2mm]
&\tilde{h}^L_{\phi}=\tilde{q}_{\phi}(J^P)^{-1}\tilde{n}_0-(\nabla\phi J^P+(\nabla\phi J^P)^T)(J^P)^{-1}\tilde{n}_0,
\end{align*}
\medskip 

\noindent where $\tilde{\zeta}_{\phi}=\mathcal{I}-t\nabla(J^P \tilde{v}_0)$ and $(J^P_{\phi})_{ij}=J^P_{ij}+t\partial_kJ^P_{ij}J^P_{kl}\tilde{v}_{0,l}$. These technical modifications are important because now we have $ \left(\tilde{f}^{(n)}-\tilde{f}_{G_0}, \bar{g}^{(n)}, \tilde{h}^{(n)}-\tilde{h}_{G_0}\right)\in Y_0$ and also\\
$\left(\tilde{f}_{\phi}^L+\tilde{f}_{G_0},\bar{g}_{\phi}^L,\tilde{h}^L_{\phi}+\tilde{h}_{G_0} \right)\in Y_0$ and thus we get the independence of time for all the constants.
\medskip

\noindent We can study separately the equations and we start with \eqref{newX}. For this we can use the result obtained in \cite[Proposition 5.3]{CCFGG2}.

\begin{lemma}
For $2<s<\frac{5}{2}$ and $T>0$ small enough depending on $N, \tilde{v}_0$, we have

\begin{enumerate}
\item Let $\tilde{X}^{(n)}-\hat{X}\in \mathcal{F}^{s+1,\gamma}$, $\tilde{w}^{(n)}\in\mathcal{K}^{s+1}_{(0)}$ and $\tilde{q}_w^{(n)}\in\mathcal{K}^s_{pr(0)}$ and such that

\begin{enumerate}
\item $\tilde{X}^{(n)}-\hat{X}\in\left\lbrace \tilde{X}-\hat{X}\in\mathcal{F}^{s+1,\gamma} :\left\|\tilde{X}-\tilde{\alpha}-\int_0^t J^P\nabla\phi\,d\tau\right\|_{\mathcal{F}^{s+1,\gamma}} \leq N \right\rbrace\equiv B_1,$\\[3mm]
\item $(\tilde{w}^{(n)},\tilde{q}_w^{(n)})\in\left\lbrace (\tilde{w},\tilde{q})\in\mathcal{K}^{s+1}_{(0)}\times\mathcal{K}^s_{pr(0)}:  \right.\\
\left.\hspace{2cm}\left\|(\tilde{w},\tilde{q})-L^{-1}(\tilde{f}_{\phi}^L+\tilde{f}_{G_0},\bar{g}_{\phi}^L,\tilde{h}_{\phi}^L+\tilde{h}_{G_0}))\right\|_{\mathcal{K}^{s+1}_{(0)}\times \mathcal{K}^s_{pr(0)}}\leq N\right\rbrace\equiv B_2.$
\end{enumerate}
\medskip

\noindent Then, $\tilde{X}^{(n+1)}-\hat{X}\in B_1$.
\bigskip

\item Let $\tilde{X}^{(n)}-\tilde{\alpha}, \tilde{X}^{(n-1)}-\tilde{\alpha}\in B_1$ and $(\tilde{w}^{(n)}, \tilde{q}^{(n)}), (\tilde{w}^{(n-1)}, \tilde{q}^{(n-1)})\in B_2.$ Then

\begin{equation*}
\left\|\tilde{X}^{(n+1)}-\tilde{X}^{(n)}\right\|_{\mathcal{F}^{s+1,\gamma}}\leq C(\tilde{v}_0)T^{\delta}\left(\left\|\tilde{w}^{(n)}-\tilde{w}^{(n-1)}\right\|_{\mathcal{K}^{s+1}_{(0)}}+\left\|\tilde{X}^{(n)}-\tilde{X}^{(n-1)}\right\|_{\mathcal{F}^{s+1,\gamma}}\right).
\end{equation*}
\end{enumerate}
\end{lemma}
\bigskip

\noindent One of the new terms we have in this paper is the deformation gradient $\tilde{G}$. We have to prove the iterative bounds for this term and we can show the following result.

\begin{proposition}\label{Gn}
For $2<s<\frac{5}{2}$ and $T>0$ small enough, depending only on $N, \tilde{v}_0, \tilde{G}_0$, we have

\begin{enumerate}
\item Let $\tilde{G}^{(n)}-\hat{G}\in\mathcal{F}^{s,\gamma-1}$, $\tilde{X}^{(n)}-\hat{X}\in\mathcal{F}^{s+1,\gamma}$, and $\tilde{w}^{(n)}\in \mathcal{K}^{s+1}_{(0)}$ and such that
\medskip

\begin{enumerate}
\item $\tilde{X}^{(n)}-\hat{X}\in B_1$,\\[2mm]
\item $(\tilde{w}^{(n)},\tilde{q}_w^{(n)})\in B_2$,\\[2mm]
\item $\displaystyle \tilde{G}^{(n)}-\hat{G}\in\left\lbrace \tilde{G}-\hat{G}\in\mathcal{F}^{s,\gamma-1} : \left\|\tilde{G}-\tilde{G}_0-\int_0^t J^P\nabla\phi \tilde{G}_0\,d\tau\right\|_{\mathcal{F}^{s,\gamma-1}} \leq N\right\rbrace\equiv B_3,$
\end{enumerate}
\medskip

\noindent Then, $\tilde{G}^{(n+1)}-\hat{G}\in B_3.$
\bigskip

\item  Let $\tilde{X}^{(n)}-\tilde{\alpha}, \tilde{X}^{(n-1)}-\tilde{\alpha}\in B_1$, $(\tilde{w}^{(n)}, \tilde{q}^{(n)}), (\tilde{w}^{(n-1)}, \tilde{q}^{(n)}) \in B_2$ and $\tilde{G}^{(n)}-\tilde{G}_0, \tilde{G}^{(n-1)}-\tilde{G}_0\in B_3$. Then, for a suitable $\delta>0$,

\begin{equation}\label{prop-localex-G}
\begin{split}
\left\|\tilde{G}^{(n+1)}-\tilde{G}^{(n)}\right\|_{\mathcal{F}^{s,\gamma-1}}&\leq C(\tilde{v}_0,\tilde{G}_0) T^{\delta}\left( \left\|\tilde{G}^{(n)}-\tilde{G}^{(n-1)}\right\|_{\mathcal{F}^{s,\gamma-1}}+\left\|\tilde{w}^{(n)}-\tilde{w}^{(n-1)}\right\|_{\mathcal{K}^{s+1}_{(0)}}\right.\\[2mm]
&\hspace{5cm}\left.+\left\|\tilde{X}^{(n)}-\tilde{X}^{(n-1)}\right\|_{\mathcal{ F}^{s+1,\gamma}}\right) 
\end{split}
\end{equation}
\end{enumerate}
\end{proposition}

\begin{proof}
\textbf{Part 1.}\\
\begin{align*}
&\left\|\tilde{G}^{(n+1)}-\tilde{G}_0-\int_0^t J^P\nabla\phi \tilde{G}_0\,d\tau\right\|_{\mathcal{ F}^{s,\gamma-1}}\leq \left\|\int_0^t J^P( \tilde{X}^{(n)})\tilde{\zeta}^{(n)}\nabla \tilde{w}^{(n)}\tilde{G}^{(n)}\right\|_{\mathcal{ F}^{s,\gamma-1}}\\[2mm]
&+\left\|\int_0^t J^P(\tilde{X}^{(n)})\tilde{\zeta}^{(n)}\nabla\tilde{v}_0 \tilde{G}^{(n)}- J^P\nabla\tilde{v}_0\tilde{G}_0\right\|_{\mathcal{ F}^{s,\gamma-1}}+\left\|\int_0^t J^P(\tilde{X}^{(n)})\tilde{\zeta}^{(n)}\tau\nabla\hat{\phi}\tilde{G}^{(n)}- J^P\tau\nabla\hat{\phi }\tilde{G}_0\,d\tau\right\|_{\mathcal{ F}^{s,\gamma-1}}\\[2mm]
&=I_1+I_2+I_3.
\end{align*}
\medskip

\noindent Let us start with  the estimate in $L^{\infty}_{\frac{1}{4}}H^s$. In order to use \eqref{flux-estim}, \eqref{defgrad-estim},  Lemma \ref{Jp-est} and Lemma \ref{zeta-est}, we need to split these terms in a right way. For the fist term we have

\begin{align*}
&\left\|\int_0^t J^P( \tilde{X}^{(n)})\tilde{\zeta}^{(n)}\nabla \tilde{w}^{(n)}\tilde{G}^{(n)}\right\|_{L^{\infty}_{\frac{1}{4}}H^s}\leq\sup_{t\in [0,T]} t^{-\frac{1}{4}}\int_0^t \|J^P(\tilde{X}^{(n)})\tilde{\zeta}^{(n)}\nabla \tilde{w}^{(n)} \tilde{G}^{(n)}\|_{H^s}\\[2mm]
&\leq T^{\frac{1}{4}}\|J^P(\tilde{X}^{(n)})\tilde{\zeta}^{(n)}\nabla \tilde{w}^{(n)} (\tilde{G}^{(n)}-\tilde{G}_0)\|_{L^2H^s}+T^{\frac{1}{4}}\|J^P(\tilde{X}^{(n)}) \tilde{\zeta}^{(n)}\nabla \tilde{w}^{(n)} \tilde{G}_0\|_{L^2H^s}=I_{1,1}+I_{1,2}
\end{align*}

\noindent We just show the estimate of $I_{1,1}$, since $I_{1,2}$ can be easily deduced from $I_{1,1}$.

\begin{align*}
I_{1,1}&\leq T^{\frac{1}{4}}\|J^P(\tilde{X}^{(n)})\|_{L^{\infty}H^s}\|\tilde{\zeta}^{(n)}\|_{L^{\infty}H^s}\|\tilde{w}^{(n)}\|_{L^{2}H^{s+1}}\|\tilde{G}^{(n)}-\tilde{G}_0\|_{L^{\infty}H^s}\\[2mm]
&\leq T^{\frac{1}{4}}C(\tilde{v}_0)\|\tilde{w}^{(n)}\|_{\mathcal{K}^{s+1}_{(0)}}\left(T^{\frac{1}{4}}\|\tilde{G}^{(n)}-\hat{G}\|_{\mathcal{F}^{s,\gamma-1}}+T\|\tilde{v}_0\|_{H^{s+1}}\|\tilde{G}_0\|_{H^s}\right)\\[2mm]
&\leq T^{\frac{1}{2}}C(\tilde{v}_0,\tilde{G}_0).\\[3mm]
\end{align*}
\medskip

\noindent For the second term we have 
\begin{equation*}
\begin{aligned}
&\left\|\int_0^t J^P(\tilde{X}^{(n)})\tilde{\zeta}^{(n)}\nabla\tilde{v}_0 \tilde{G}^{(n)}- J^P\nabla\tilde{v}_0\tilde{G}_0\right\|_{L^{\infty}_{\frac{1}{4}}H^s}\leq \sup_{t\in [0,T]} t^{-\frac{1}{4}}\int_0^t \|J^P(\tilde{X}^{(n)})\tilde{\zeta}^{(n)}\nabla\tilde{v}_0 \tilde{G}^{(n)}- J^P\nabla\tilde{v}_0\tilde{G}_0\|_{H^s}\\
&\leq T^{\frac{1}{4}}\|J^P(\tilde{X}^{(n)})\tilde{\zeta}^{(n)}\nabla\tilde{v}_0 \tilde{G}^{(n)}- J^P\nabla\tilde{v}_0\tilde{G}_0\|_{L^2H^s}\leq T^{\frac{1}{4}}\left(\|(J^P(\tilde{X}^{(n)})-J^P)\tilde{\zeta}^{(n)}\nabla\tilde{v}_0 \tilde{G}^{(n)}\|_{L^2H^s}+\right.\\[2mm]
&\hspace{0.5cm}\left.\|J^P(\tilde{\zeta}^{(n)}-\mathcal{I})\nabla\tilde{v}_0 \tilde{G}^{(n)}\|_{L^2H^s}+\|J^P\nabla\tilde{v}_0(\tilde{G}^{(n)}-\tilde{G}_0)\|_{L^2H^s}\right)=I_{2,1}+I_{2,2}+I_{2,3}.
\end{aligned}
\end{equation*}

\noindent We show how to get the result for $I_{2,1}$, since for $I_{2,2}$ and $I_{2,3}$ is the same just by noticing that 
$\tilde{\zeta}^{(n)}-\mathcal{I}=\tilde{\zeta}^{(n)}(\mathcal{I}-(\nabla \tilde{X}^{(n)})^{-1})=\tilde{\zeta}^{(n)}\nabla(\tilde{\alpha}-\tilde{X}^{(n)})$.

\begin{align*}
I_{2,1}&\leq T^{\frac{3}{4}}C(\tilde{v}_0)\|J^P(\tilde{X}^{(n)})-J^P\|_{L^{\infty}H^s}\|\tilde{\zeta}^{(n)}\|_{L^{\infty}H^s}\left(\|\tilde{G}^{(n)}-\tilde{G}_0\|_{L^{\infty}H^s}+\|\tilde{G}_0\|_{H^s}\right)\\[2mm]
&\leq C(\tilde{v}_0)T^{\frac{3}{4}}\left(T^{\frac{1}{4}}\|\tilde{G}^{(n)}-\hat{G}\|_{\mathcal{F}^{s,\gamma-1}}+T\|\tilde{v}_0\|_{H^{s+1}}\|\tilde{G}_0\|_{H^s}+\|\tilde{G}_0\|_{H^s}\right)\leq C(\tilde{v}_0,\tilde{G}_0) T.
\end{align*}
\medskip

\noindent And we have $\sum_{i=2,3}I_{2,i}\leq  C(\tilde{v}_0,\tilde{G}_0) T.$  For the third term $I_3$, we have

\begin{align*}
&\left\|\int_0^t J^P(\tilde{X}^{(n)})\tilde{\zeta}^{(n)}\tau\nabla\hat{\phi}\tilde{G}^{(n)}- J^P\tau\nabla\hat{\phi }\tilde{G}_0\right\|_{L^{\infty}_{\frac{1}{4}}H^s}\leq \sup_{t\in [0,T]} t^{-\frac{1}{4}}\int_0^t\|J^P(\tilde{X}^{(n)})\tilde{\zeta}^{(n)}t\nabla\hat{\phi}\tilde{G}^{(n)}- J^Pt\nabla\hat{\phi }\tilde{G}_0\|_{H^s}\\[2mm]
&\leq T^{\frac{1}{4}}\|J^P(\tilde{X}^{(n)})\tilde{\zeta}^{(n)}t\nabla\hat{\phi}\tilde{G}^{(n)}- J^Pt\nabla\hat{\phi }\tilde{G}_0\|_{L^2H^s}\leq T^{\frac{1}{4}}\left(\|(J^P(\tilde{X}^{(n)})-J^P)\tilde{\zeta}^{(n)}t\nabla\hat{\phi}\tilde{G}^{(n)}\|_{L^2H^s}\right.\\[2mm]
&\left.+\|J^P(\tilde{\zeta}^{(n)}-\mathcal{I})t\nabla\hat{\phi}\tilde{G}^{(n)}\|_{L^2H^s}+\|J^P t\nabla\hat{\phi}(\tilde{G}^{(n)}-\tilde{G}_0\|_{L^2H^s}\right)=I_{3,1}+I_{3,2}+I_{3,3}.
\end{align*}
\noindent As for the other terms it is enough to show the estimate of $I_{3,1}$.

\begin{align*}
I_{3,1}&\leq T^{\frac{1}{4}}\|J^P(\tilde{X}^{(n)})-J^P\|_{L^{\infty}H^s}\|\tilde{\zeta}^{(n)}\|_{L^{\infty}H^s}\|t\nabla\hat{\phi}\|_{L^2H^s}\left(\|\tilde{G}^{(n)}-\tilde{G}_0\|_{L^{\infty}H^s}+\|\tilde{G}_0\|_{H^s}\right)\\[2mm]
&\leq T^{\frac{1}{4}}C(\tilde{v}_0)\|t\|_{L^2}\|\nabla\hat{\phi}\|_{H^s}\left(T^{\frac{1}{4}}\|\tilde{G}^{(n)}-\hat{G}\|_{\mathcal{F}^{s,\gamma-1}}+T\|\tilde{v}_0\|_{H^{s+1}}\|\tilde{G}_0\|_{H^s}+\|\tilde{G}_0\|_{H^s}\right)\\[2mm]
&\leq C(\tilde{v}_0,\tilde{G}_0) T.
\end{align*}
\medskip

\noindent And we have $\sum_{i=2,3}I_{3,i}\leq  C(\tilde{v}_0,\tilde{G}_0) T.$ In this way we obtain the estimates in $L^{\infty}_{\frac{1}{4}}H^s$ and we need to obtain a similar result in $H^2_{(0)}H^{\gamma-1}$.   For the second space, we use Lemma \ref{lem3}, with $1<\gamma< s-1$, Lemma \ref{lem5}, Lemma \ref{lem1} and the flux and deformation gradient estimates \eqref{flux-estim}, \eqref{defgrad-estim} and Lemma \ref{Jp-est}  and Lemma \ref{zeta-est}. We remark that in this case the estimates are more complicated because in order to apply the mentioned lemmas we need all terms to be zero at $t=0$, so it implies to add and subtract the right terms. We start with the first one

\begin{align*}
&\left\|\int_0^t J^P( \tilde{X}^{(n)})\tilde{\zeta}^{(n)}\nabla \tilde{w}^{(n)}\tilde{G}^{(n)}\right\|_{H^2_{(0)}H^{\gamma-1}}\leq \left\|J^P( \tilde{X}^{(n)})\tilde{\zeta}^{(n)}\nabla \tilde{w}^{(n)}\tilde{G}^{(n)}\right\|_{H^1_{(0)}H^{\gamma-1}}=I_1,
\end{align*}

\noindent by using Lemma \ref{lem2} with $\varepsilon=0$. Now we can rewrite this term in a more convenient way

\begin{align*}
I_1&=(J^P( \tilde{X}^{(n)})-J^P)(\tilde{\zeta}^{(n)}-\mathcal{I})\nabla \tilde{w}^{(n)}(\tilde{G}^{(n)}-\tilde{G}_0)+( J^P( \tilde{X}^{(n)})-J^P)(\tilde{\zeta}^{(n)}-\mathcal{I})\nabla \tilde{w}^{(n)}\tilde{G}_0\\
&+ (J^P( \tilde{X}^{(n)})-J^P)\nabla\tilde{w}^{(n)}(\tilde{G}^{(n)}-\tilde{G}_0)+(J^P( \tilde{X}^{(n)})-J^P)\nabla\tilde{w}^{(n)}\tilde{G}_0\\
&+J^P\nabla\tilde{w}^{(n)}(\tilde{G}^{(n)}-\tilde{G}_0)+J^P\nabla\tilde{w}^{(n)}\tilde{G}_0+J^P(\tilde{\zeta}^{(n)}-\mathcal{I})\nabla \tilde{w}^{(n)}(\tilde{G}^{(n)}-\tilde{G}_0)\\
&+J^P(\tilde{\zeta}^{(n)}-\mathcal{I})\nabla \tilde{w}^{(n)}\tilde{G}_0=\sum_{i=1}^8 I_{1,i}.
\end{align*}

\noindent We show the estimate of $I_{1,1}$, the others are similar or easier. 

\begin{align*}
\|I_{1,1}\|_{H^1_{(0)}H^{\gamma-1}}&\leq \|J^P( \tilde{X}^{(n)})-J^P\|_{H^1_{(0)}H^{\gamma}}\|(\tilde{\zeta}^{(n)}-\mathcal{I})\nabla \tilde{w}^{(n)}(\tilde{G}^{(n)}-\tilde{G}_0)\|_{H^1_{(0)}H^{\gamma-1}}\\[2mm]
&\leq C(\tilde{v}_0)\|\tilde{X}^{(n)}-\tilde{\alpha}\|_{H^1_{(0)}H^{\gamma-1}}\|\tilde{\zeta}^{(n)}-\mathcal{I}\|_{H^1_{(0)}H^{\gamma-1}}\|\tilde{w}^{(n)}\|_{H^1_{(0)}H^{\gamma}}\|\tilde{G}^{(n)}-\tilde{G}_0\|_{H^1_{(0)}H^{\gamma-1}}\\[2mm]
&\leq C(\tilde{v}_0,\tilde{G}_0)T^{\delta_1}.
\end{align*}
\medskip

\noindent And  $\sum_{i=2}^{8}\|I_{1,i}\|_{H^1_{(0)}H^{\gamma-1}}\leq  \sum_{i=2}^{8} C(\tilde{v}_0,\tilde{G}_0) T^{\delta_i}$. The second term can be  estimated in the same manner as the previous one.

\begin{align*}
&\left\|\int_0^t J^P( \tilde{X}^{(n)})\tilde{\zeta}^{(n)}\nabla \tilde{v}_0\tilde{G}^{(n)}-J^P\nabla\tilde{v}_0\tilde{G}_0\right\|_{H^2_{(0)}H^{\gamma-1}}\leq \left\|J^P( \tilde{X}^{(n)})\tilde{\zeta}^{(n)}\nabla \tilde{v}_0\tilde{G}^{(n)}-J^P\nabla\tilde{v}_0\tilde{G}_0\right\|_{H^1_{(0)}H^{\gamma-1}}=I_2.
\end{align*}

\noindent In order to use the mentioned Lemmas we need to do some adjustements, thus the term can be splitted as follows

\begin{align*}
I_2&=( J^P( \tilde{X}^{(n)})-J^P)(\tilde{\zeta}^{(n)}-\mathcal{I})\nabla \tilde{v}_0(\tilde{G}^{(n)}-\tilde{G}_0)+( J^P( \tilde{X}^{(n)})-J^P)(\tilde{\zeta}^{(n)}-\mathcal{I})\nabla \tilde{v}_0\tilde{G}_0\\
&+ (J^P( \tilde{X}^{(n)})-J^P)\nabla\tilde{v}_0(\tilde{G}^{(n)}-\tilde{G}_0)+(J^P( \tilde{X}^{(n)})-J^P)\nabla\tilde{v}_0\tilde{G}_0\\
&+J^P(\tilde{\zeta}^{(n)}-\mathcal{I})\nabla\tilde{v}_0(\tilde{G}^{(n)}-\tilde{G}_0)+J^P(\tilde{\zeta}^{(n)}-\mathcal{I})\nabla\tilde{v}_0\tilde{G}_0+J^P\nabla \tilde{v}_0(\tilde{G}^{(n)}-\tilde{G}_0)=\sum_{i=1}^7=I_{2,i}.
\end{align*}

\noindent Also in this case we will show the estimate of $I_{2,1}$.

\begin{align*}
\|I_{2,1}\|_{H^1_{(0)}H^{\gamma-1}}&\leq \|J^P( \tilde{X}^{(n)})-J^P\|_{H^1_{(0)}H^{\gamma}}\|(\tilde{\zeta}^{(n)}-\mathcal{I})\nabla \tilde{v}_0(\tilde{G}^{(n)}-\tilde{G}_0)\|_{H^1_{(0)}H^{\gamma-1}}\\[2mm]
&\leq C(\tilde{v}_0)\|\tilde{X}^{(n)}-\tilde{\alpha}\|_{H^1_{(0)}H^{\gamma-1}}\|\tilde{\zeta}^{(n)}-\mathcal{I}\|_{H^1_{(0)}H^{\gamma-1}}\|\nabla \tilde{v}_0\|_{H^s}\|(\tilde{G}^{(n)}-\tilde{G}_0)\|_{H^1_{(0)}H^{\gamma-1}}\\[2mm]
&\leq C(\tilde{v}_0,\tilde{G}_0) T^{\theta_1}.
\end{align*}
\medskip

\noindent And  $\sum_{i=2}^{7}\|I_{2,i}\|_{H^1_{(0)}H^{\gamma-1}}\leq  \sum_{i=2}^{7} C(\tilde{v}_0,\tilde{G}_0) T^{\theta_i}.$ Finally, to conclude the first part of the proof we estimate the third term, as follows

\begin{align*}
\left\|\int_0^t J^P(\tilde{X}^{(n)})\tilde{\zeta}^{(n)}\tau\nabla\hat{\phi}\tilde{G}^{(n)}- J^P\tau\nabla\hat{\phi }\tilde{G}_0\,d\tau\right\|_{H^2_{(0)}H^{\gamma-1}}&\leq \left\|J^P(\tilde{X}^{(n)})\tilde{\zeta}^{(n)}\tau\nabla\hat{\phi}\tilde{G}^{(n)}- J^P\tau\nabla\hat{\phi }\tilde{G}_0\right\|_{H^1_{(0)}H^{\gamma-1}}\\[1mm]
&=I_3.
\end{align*}
\medskip

\noindent This term as the previous one can be splitted

\begin{align*}
I_3&=( J^P( \tilde{X}^{(n)})-J^P)(\tilde{\zeta}^{(n)}-\mathcal{I})t\nabla \hat{\phi}(\tilde{G}^{(n)}-\tilde{G}_0)+( J^P( \tilde{X}^{(n)})-J^P)(\tilde{\zeta}^{(n)}-\mathcal{I})t\nabla \hat{\phi}\tilde{G}_0\\
&+ (J^P( \tilde{X}^{(n)})-J^P)t\nabla\hat{\phi}(\tilde{G}^{(n)}-\tilde{G}_0)+(J^P( \tilde{X}^{(n)})-J^P)t \hat{\phi}\tilde{G}_0\\
&+J^P(\tilde{\zeta}^{(n)}-\mathcal{I})t\nabla\hat{\phi}(\tilde{G}^{(n)}-\tilde{G}_0)+J^P(\tilde{\zeta}^{(n)}-\mathcal{I})t\nabla\hat{\phi}_0\tilde{G}_0+J^Pt\nabla \hat{\phi}(\tilde{G}^{(n)}-\tilde{G}_0)=\sum_{i=1}^7 I_{3,i}.
\end{align*}

\noindent In order to estimate these terms, we will use the preliminary lemmas and we want to point out the fact that $\hat{\phi}$ does not depend on time but just on the initial data $\tilde{v}_0, \tilde{G}_0$ and $\|t\|_{H^1_{(0)}}\leq T^{\frac{1}{2}}$.

\begin{align*}
&\|I_{3,1}\|_{H^1_{(0)}H^{\gamma-1}}\leq \|J^P( \tilde{X}^{(n)})-J^P\|_{H^1_{(0)}H^{\gamma}}\|(\tilde{\zeta}^{(n)}-\mathcal{I})t\nabla \hat{\phi}(\tilde{G}^{(n)}-\tilde{G}_0)\|_{H^1_{(0)}H^{\gamma-1}}\\[2mm]
&\leq C(\tilde{v}_0)\|\tilde{X}^{(n)}-\tilde{\alpha}\|_{H^1_{(0)}H^{\gamma-1}}\|\tilde{\zeta}^{(n)}-\mathcal{I}\|_{H^1_{(0)}H^{\gamma-1}}\|t\nabla \hat{\phi}\|_{H^1_{(0)}H^{\gamma-1}}\|\tilde{G}^{(n)}-\tilde{G}_0\|_{H^1_{(0)}H^{\gamma-1}}\\[2mm]
&\leq C(\tilde{v}_0,\tilde{G}_0) T^{\rho_1}\|t\|_{H^1_{(0)}}\|\nabla\hat{\phi}\|_{H^{\gamma-1}}
\end{align*}

\noindent For the other estimates we can state $\sum_{i=2}^{7}\|I_{3,i}\|_{H^1_{(0)}H^{\gamma-1}}\leq \sum_{i=2}^{7}  C(\tilde{v}_0,\tilde{G}_0) T^{\rho_i}.$

\noindent We prove the first part of the Proposition  by taking $\delta=\min\{\frac{1}{4}, \delta_i,\theta_j,\rho_j\}$, for $i=1,\ldots, 8$ and $j=1,\ldots,7$.
\medskip

\textbf{Part 2.}\\

\noindent We consider the following difference

\begin{align}\label{Gn-diff}
\tilde{G}^{(n+1)}-\tilde{G}^{(n)}&=
\int_0^t J^P (\tilde{X}^{(n)})\tilde{\zeta}^{(n)}\nabla \tilde{v}^{(n)}\tilde{G}^{(n)}-J^P (\tilde{X}^{(n-1)})\tilde{\zeta}^{(n-1)}\nabla \tilde{v}^{(n-1)}\tilde{G}^{(n-1)}
\end{align}

\noindent Thus we rewrite the norm of (\ref{Gn-diff}) as follows

\begin{align*}
& \left\|\int_0^t J^P (\tilde{X}^{(n)})\tilde{\zeta}^{(n)}\nabla \tilde{v}^{(n)}\tilde{G}^{(n)}-J^P (\tilde{X}^{(n-1)})\tilde{\zeta}^{(n-1)}\nabla \tilde{v}^{(n-1)}\tilde{G}^{(n-1)}\right\|_{\mathcal{F}^{s,\gamma-1}}\\[2mm]
&\leq\left\|\int_0^t J^P (\tilde{X}^{(n)})\tilde{\zeta}^{(n)}\nabla \tilde{v}^{(n)}\tilde{G}^{(n)}-J^P( \tilde{X}^{(n-1)})\tilde{\zeta}^{(n-1)}\nabla \tilde{v}^{(n-1)}\tilde{G}^{(n-1)}\right\|_{L^{\infty}_{\frac{1}{4}}H^s}\\[2mm]
&+\left\|\int_0^t J^P(\tilde{X}^{(n)})\tilde{\zeta}^{(n)}\nabla \tilde{v}^{(n)}\tilde{G}^{(n)}-J^P( \tilde{X}^{(n-1)})\tilde{\zeta}^{(n-1)}\nabla \tilde{v}^{(n-1)}\tilde{G}^{(n-1)}\right\|_{H^2_{(0)}H^{\gamma-1}}.
\end{align*}
\medskip

\noindent We start with the estimate in $L^{\infty}_{\frac{1}{4}}H^{s}$ and we have

\begin{align*}
&\left\|\tilde{G}^{(n+1)}-\tilde{G}^{(n)}\right\|_{L^{\infty}_{\frac{1}{4}}H^s}\\[2mm]
&\hspace{0.3cm}\leq \sup_{0\leq t\leq T}t^{-\frac{1}{4}}\left\|\int_0^t J^P( \tilde{X}^{(n)})\tilde{\zeta}^{(n)}\nabla\tilde{v}^{(n)}\tilde{G}^{(n)}-J^P( \tilde{X}^{(n-1)})\tilde{\zeta}^{(n-1)}\nabla \tilde{v}^{(n-1)}\tilde{G}^{(n-1)}\right\|_{H^s}\\[2mm]
&\hspace{0.3cm}\leq T^{\frac{1}{4}}\left\|J^P(\tilde{X}^{(n)})\tilde{\zeta}^{(n)}\nabla \tilde{v}^{(n)}\tilde{G}^{(n)}-J^P(\tilde{X}^{(n-1)})\tilde{\zeta}^{(n-1)}\nabla \tilde{v}^{(n-1)}\tilde{G}^{(n-1)}\right\|_{L^2H^s}=T^{\frac{1}{4}}\|I_1\|_{L^2H^s}.
\end{align*}

\noindent We split $I_1$ as follows

\begin{align*}
I_1=&(J^P( \tilde{X}^{(n)})-J^P( \tilde{X}^{(n-1)}))\tilde{\zeta}^{(n)}\nabla \tilde{v}^{(n)}\tilde{G}^{(n)}+J^P( \tilde{X}^{(n-1)})(\tilde{\zeta}^{(n)}-\tilde{\zeta}^{(n-1)})\nabla \tilde{v}^{(n)}\tilde{G}^{(n)}\\[2mm]
&+J^P( \tilde{X}^{(n-1)})\tilde{\zeta}^{(n-1)}(\nabla \tilde{v}^{(n)}-\nabla \tilde{v}^{(n-1)})\tilde{G}^{(n)}+J^P( \tilde{X}^{(n-1)})\tilde{\zeta}^{(n-1)}\nabla\tilde{ v}^{(n-1)}(\tilde{G}^{(n)}-\tilde{G}^{(n-1)})=\sum_{i=1}^4 I_{1,i}.
\end{align*}

\noindent We estimate $I_{1,1}$ by taking into account the definition of the velocity $\tilde{v}=\tilde{w}+\tilde{v}_0+t\hat{\phi}$ and the fact that $\|\tilde{v}^{(n)}\|_{L^2H^{s+1}}\leq \|\tilde{w}^{(n)}\|_{L^2H^{s+1}}+\|\tilde{v}_0\|_{L^2H^{s+1}}+\|t\hat{\phi}\|_{L^2H^{s+1}}$. Furthermore we apply estimates \eqref{flux-estim}, \eqref{defgrad-estim} and Lemma \ref{Jp-dif-est}, Lemma \ref{zeta-est}.

\begin{align*}
&T^{\frac{1}{4}}\|I_{1,1}\|_{L^2H^s}\leq T^{\frac{1}{4}}\|J^P( \tilde{X}^{(n)})-J^P(\tilde{X}^{(n-1)})\|_{L^{\infty}H^s}\|\tilde{\zeta}^{(n)}\|_{L^{\infty}H^s}\|\nabla\tilde{v}^{(n)}\|_{L^2H^s}\\
&\hspace{3cm}\cdot\left( \|\tilde{G}^{(n)}-\tilde{G}_0\|_{L^{\infty}H^s}+\|\tilde{G}_0\|_{H^s}\right)\\[2mm]
&\leq C(\tilde{v}_0,\tilde{G}_0) T^{\frac{1}{2}} \| \tilde{X}^{(n)}-\tilde{X}^{(n-1)}\|_{L^{\infty}_{\frac{1}{4}}H^s} \|\tilde{v}^{(n)}\|_{L^2H^{s+1}}\leq C(\tilde{v}_0,\tilde{G}_0) T^{\frac{1}{2}} \| \tilde{X}^{(n)}-\tilde{X}^{(n-1)}\|_{\mathcal{F}^{s+1,\gamma}}.
\end{align*}

\noindent We notice that also for $I_{1,2}$ the estimate can be obtained in the same manner as $I_{1,1}$.

\begin{align*}
&T^{\frac{1}{4}}\|I_{12}\|_{L^2H^s}\leq  C(\tilde{v}_0,\tilde{G}_0) T^{\frac{1}{2}} \| \tilde{X}^{(n)}-\tilde{X}^{(n-1)}\|_{\mathcal{F}^{s+1,\gamma}}.
\end{align*}
\medskip

\noindent We want to focus on $I_{1,3}, I_{1,4}$ which give us the other differences that appear in \eqref{prop-localex-G}.\\
\noindent $I_{1,3}=J^P( \tilde{X}^{(n-1)})\tilde{\zeta}^{(n-1)}(\nabla \tilde{w}^{(n)}-\nabla \tilde{w}^{(n-1)})\tilde{G}^{(n)}$ by the definition of the velocity. Thus its estimate follows from Lemma \ref{Jp-est}, Lemma \ref{zeta-est} and estimates \eqref{flux-estim}, \eqref{defgrad-estim}.

\begin{align*}
&T^{\frac{1}{4}}\|J^P( \tilde{X}^{(n-1)})\tilde{\zeta}^{(n-1)}(\nabla \tilde{w}^{(n)}-\nabla \tilde{w}^{(n-1)})\tilde{G}^{(n)}\|_{L^2H^s}\leq T^{\frac{1}{4}}\|J^P( \tilde{X}^{(n-1)})\|_{L^{\infty}H^s}\|\tilde{\zeta}^{(n-1)}\|_{L^{\infty}H^s}\\[2mm]
&\cdot \|\nabla \tilde{w}^{(n)}-\nabla \tilde{w}^{(n-1)}\|_{L^2H^s} \left(\|\tilde{G}^{(n)}-\tilde{G}_0\|_{L^{\infty}H^s}+\|\tilde{G}_0\|_{H^s}\right)\leq C T^{\frac{1}{4}} \|\tilde{w}^{(n)}- \tilde{w}^{(n-1)}\|_{\mathcal{K}^{s+1}}.
\end{align*}
\medskip

\noindent  For the last term we have

\begin{align*}
&T^{\frac{1}{4}}\|I_{1,4}\|_{L^2H^s}\leq T^{\frac{1}{4}}\|J^P(\tilde{ X}^{(n-1)})\|_{L^{\infty}H^s}\|\tilde{\zeta}^{(n-1)}\|_{L^{\infty}H^s} \|\nabla \tilde{v}^{(n-1)}\|_{L^2H^s} \|\tilde{G}^{(n)}-\tilde{G}^{(n-1)}\|_{L^{\infty}H^s}\\[2mm]
&\hspace{2.1cm}\leq C(\tilde{v}_0) T^{\frac{1}{2}} \|\tilde{G}^{(n)}- \tilde{G}^{(n-1)}\|_{\mathcal{F}^{s,\gamma-1}}.\\
\end{align*}
\medskip

\noindent The estimate of $\tilde{G}^{(n)}-\tilde{G}^{(n-1)}$ in $H^2_{(0)}H^{\gamma-1}$ is more complicated since we will use Lemma \ref{lem5} which requires all terms to be zero at $t=0$ then we need to do some adjustements, as we did for the first part of the proof. First of all we observe that by using Lemma \ref{lem2} with $\varepsilon=0$ and the definition of the velocity we get

\begin{align*}
\left\|\tilde{G}^{(n+1)}-\tilde{G}^{(n)}\right\|_{H^2_{(0)}H^{\gamma-1}}&\leq \left\|J^P(\tilde{X}^{(n)})\tilde{\zeta}^{(n)}\nabla \tilde{v}^{(n)}\tilde{G}^{(n)}-J^P( \tilde{X}^{(n-1)})\tilde{\zeta}^{(n-1)}\nabla \tilde{v}^{(n-1)}\tilde{G}^{(n-1)}\right\|_{H^1_{(0)}H^{\gamma-1}}\\[2mm]
&\leq\left\|J^P(\tilde{X}^{(n)})\tilde{\zeta}^{(n)}\nabla \tilde{w}^{(n)}\tilde{G}^{(n)}-J^P( \tilde{X}^{(n-1)})\tilde{\zeta}^{(n-1)}\nabla \tilde{w}^{(n-1)}\tilde{G}^{(n-1)}\right\|_{H^1_{(0)}H^{\gamma-1}}\\[2mm]
&+\left\|J^P(\tilde{X}^{(n)})\tilde{\zeta}^{(n)}\nabla \tilde{v}_0\tilde{G}^{(n)}-J^P( \tilde{X}^{(n-1)})\tilde{\zeta}^{(n-1)}\nabla \tilde{v}_0\tilde{G}^{(n-1)}\right\|_{H^1_{(0)}H^{\gamma-1}}\\[2mm]
&+\left\|J^P(\tilde{X}^{(n)})\tilde{\zeta}^{(n)}t\nabla \hat{\phi}\tilde{G}^{(n)}-J^P( \tilde{X}^{(n-1)})\tilde{\zeta}^{(n-1)}t\nabla \hat{\phi}\tilde{G}^{(n-1)}\right\|_{H^1_{(0)}H^{\gamma-1}}\\[2mm]
&=\|I_{21}\|_{H^1_{(0)}H^{\gamma-1}}+\|I_{22}\|_{H^1_{(0)}H^{\gamma-1}}+\|I_{23}\|_{H^1_{(0)}H^{\gamma-1}}.
\end{align*}
\medskip

\noindent In order to do estimates in this space we need more modifications to apply the preliminary lemmas. The term $I_{21}$ can be splitted as follows

\begin{align*}
I_{21}&=(J^P( \tilde{X}^{(n)})-J^P( \tilde{X}^{(n-1)}))(\tilde{\zeta}^{(n)}-\mathcal{I})\nabla \tilde{w}^{(n)} (\tilde{G}^{(n)}-\tilde{G}_0)\\[1mm]
&+(J^P( \tilde{X}^{(n)})-J^P( \tilde{X}^{(n-1)}))(\tilde{\zeta}^{(n)}-\mathcal{I})\nabla \tilde{w}^{(n)} \tilde{G}_0\\[1mm]
&+(J^P( \tilde{X}^{(n)})-J^P( \tilde{X}^{(n-1)}))\nabla \tilde{w}^{(n)} (\tilde{G}^{(n)}-\tilde{G}_0)+(J^P( \tilde{X}^{(n)})-J^P( \tilde{X}^{(n-1)}))\nabla \tilde{w}^{(n)}\tilde{G}_0\\[1mm]
&+(J^P( \tilde{X}^{(n-1)})-J^P)(\tilde{\zeta}^{(n)}-\tilde{\zeta}^{(n-1)})\nabla \tilde{w}^{(n)} (\tilde{G}^{(n)}-\tilde{G}_0)+J^P(\tilde{\zeta}^{(n)}-\tilde{\zeta}^{(n-1)})\nabla \tilde{w}^{(n)} (\tilde{G}^{(n)}-\tilde{G}_0)\\[1mm]
&+(J^P( \tilde{X}^{(n-1)})-J^P)(\tilde{\zeta}^{(n)}-\tilde{\zeta}^{(n-1)})\nabla \tilde{w}^{(n)} \tilde{G}_0+J^P(\tilde{\zeta}^{(n)}-\tilde{\zeta}^{(n-1)})\nabla \tilde{w}^{(n)}\tilde{G}_0\\[1mm]
&+(J^P( \tilde{X}^{(n-1)})-J^P)(\tilde{\zeta}^{(n-1)}-\mathcal{I})(\nabla \tilde{w}^{(n)}-\nabla \tilde{w}^{(n-1)})(\tilde{G}^{(n)}-\tilde{G}_0)\\[1mm]
&+(J^P( \tilde{X}^{(n-1)})-J^P)(\tilde{\zeta}^{(n-1)}-\mathcal{I})(\nabla \tilde{w}^{(n)}-\nabla \tilde{w}^{(n-1)})\tilde{G}_0
\end{align*}
\begin{align*}
&+(J^P( \tilde{X}^{(n-1)})-J^P)(\nabla \tilde{w}^{(n)}-\nabla \tilde{w}^{(n-1)})(\tilde{G}^{(n)}-\tilde{G}_0)\\[1mm]
&+(J^P( \tilde{X}^{(n-1)})-J^P)(\nabla \tilde{w}^{(n)}-\nabla \tilde{w}^{(n-1)})\tilde{G}_0+J^P(\nabla \tilde{w}^{(n)}-\nabla \tilde{w}^{(n-1)})(\tilde{G}^{(n)}-\tilde{G}_0)\\[1mm]
&+J^P(\nabla \tilde{w}^{(n)}-\nabla \tilde{w}^{(n-1)})\tilde{G}_0+J^P(\tilde{\zeta}^{(n-1)}-\mathcal{I})(\nabla \tilde{w}^{(n)}-\nabla \tilde{w}^{(n-1)})(\tilde{G}^{(n)}-\tilde{G}_0)\\[1mm]
&+J^P(\tilde{\zeta}^{(n-1)}-\mathcal{I})(\nabla \tilde{w}^{(n)}-\nabla \tilde{w}^{(n-1)})\tilde{G}_0+(J^P( \tilde{X}^{(n-1)})-J^P)(\tilde{\zeta}^{(n-1)}-\mathcal{I})\nabla \tilde{w}^{(n-1)}(\tilde{G}^{(n)}-\tilde{G}^{(n-1)})\\[1mm]
&+(J^P( \tilde{X}^{(n-1)})-J^P)\nabla \tilde{w}^{(n-1)}(\tilde{G}^{(n)}-\tilde{G}^{(n-1)})+J^P\nabla \tilde{w}^{(n-1)}(\tilde{G}^{(n)}-\tilde{G}^{(n-1)})\\[1mm]
&+J^P(\tilde{\zeta}^{(n-1)}-\mathcal{I})\nabla \tilde{w}^{(n-1)}(\tilde{G}^{(n)}-\tilde{G}^{(n-1)})=\sum_{i=1}^{19} I_{21,i}.
\end{align*}
\medskip

\noindent We observe that whenever there is a term that does not depend on time then we can take it out of the norm $H^{1}_{(0)}$, by using the properties of this space and the preliminary lemmas. We will show the estimates of $I_{21,1}$ which is the most complicated term. 

\begin{align*}
\|I_{21,1}\|_{H^1_{(0)}H^{\gamma-1}}&\leq \|J^P( \tilde{X}^{(n)})-J^P( \tilde{X}^{(n-1)})\|_{H^{1}_{(0)}H^{\gamma}}\|(\tilde{\zeta}^{(n)}-\mathcal{I})\nabla \tilde{w}^{(n)} (\tilde{G}^{(n)}-\tilde{G}_0)\|_{H^{1}_{(0)}H^{\gamma-1}}\\[2mm]
&\leq C(\tilde{v}_0)\|\tilde{X}^{(n)}-\tilde{X}^{(n-1)}\|_{H^{1}_{(0)}H^{\gamma}}\|\tilde{\zeta}^{(n)}-\mathcal{I}\|_{H^{1}_{(0)}H^{\gamma-1}}\|\tilde{w}^{(n)} \|_{H^{1}_{(0)}H^{\gamma}}\|\tilde{G}^{(n)}-\tilde{G}_0\|_{H^{1}_{(0)}H^{\gamma-1}}\\[2mm]
&\leq C(\tilde{v}_0,\tilde{G}_0)\left\|\int_0^t \tilde{X}^{(n)}-\tilde{X}^{(n-1)}\right\|_{H^{1+\eta-\delta_1}_{(0)}H^{\gamma}}\leq C(\tilde{v}_0,\tilde{G}_0) T^{\delta_1} \|\tilde{X}^{(n)}-\tilde{X}^{(n-1)}\|_{\mathcal{F}^{s+1,\gamma}},
\end{align*}
\medskip

\noindent where we used lemma \ref{lem3} with $\gamma>1$, lemma  \ref{lem5} several times, lemma \ref{zeta-est} and estimate \eqref{flux-estim}, \eqref{defgrad-estim}. The last estimate follows from lemma \ref{lem1} for $\tilde{w}^{(n)}$ and lemma \ref{lem2} with $\eta>\delta_1$. For the remaining terms 

\begin{equation*}
\begin{split}
&\sum_{i=2}^{8} \|I_{21,i}\|_{H^1_{(0)}H^{\gamma-1}} \leq \sum_{i=2}^{8} C(\tilde{v}_0,\tilde{G}_0) T^{\delta_i} \|\tilde{X}^{(n)}-\tilde{X}^{(n-1)}\|_{\mathcal{F}^{s+1,\gamma}},\\[3mm]
&\sum_{i=9}^{15} \|I_{21,i}\|_{H^1_{(0)}H^{\gamma-1}}\leq \sum_{i=9}^{15} C(\tilde{v}_0,\tilde{G}_0) T^{\delta_i} \|\tilde{w}^{(n)}-\tilde{w}^{(n-1)}\|_{\mathcal{K}^{s+1}_{(0)}},\\[3mm]
&\sum_{i=16}^{19} \|I_{21,i}\|_{H^1_{(0)}H^{\gamma-1}}\leq \sum_{i=16}^{19} C(\tilde{v}_0,\tilde{G}_0)T^{\delta_i} \|\tilde{G}^{(n)}-\tilde{G}^{(n-1)}\|_{\mathcal{F}^{s,\gamma-1}}.
\end{split}
\end{equation*}
\medskip

\noindent We have to estimate $I_{22}$ and $I_{23}$, by dividing them in a similar way as $I_{21}$. We notice that by applying the same lemmas, we get also the same results as before, without the velocity's differences. Thus

\begin{equation*}
\begin{split}
&\|I_{22}\|_{H^1_{(0)}H^{\gamma-1}}+\|I_{23}\| _{H^1_{(0)}H^{\gamma-1}}\leq C(\tilde{v}_0,\tilde{G}_0) T^{\theta}\left(\|\tilde{X}^{(n)}-\tilde{X}^{(n-1)}\|_{\mathcal{F}^{s+1,\gamma}}+ \|\tilde{G}^{(n)}-\tilde{G}^{(n-1)}\|_{\mathcal{F}^{s,\gamma-1}}\right).
\end{split}
\end{equation*}
\medskip

\noindent Thus the proof is done for $\delta=\min\{\frac{1}{4}, \delta_i,\theta\}$, for $i=1,\ldots, 19$.
\end{proof}
\bigskip

\noindent We need to prove  similar estimates also for $(\tilde{w},\tilde{q_w})$.

\begin{proposition}\label{estim-(v,q)}
For $2<s<\frac{5}{2}$ and $T>0$ small enough, depending only on $N, \tilde{v}_0, \tilde{G}_0$, we have
\begin{enumerate}

\item Let $\tilde{X}^{(n)}-\hat{X}\in\mathcal{F}^{s+1,\gamma}$, $\tilde{w}^{(n)}\in\mathcal{K}^{s+1}_{(0)}$, and $\tilde{q}_w^{(n)}\in \mathcal{K}^s_{pr(0)}$  and such that 

\begin{enumerate}
\item $\tilde{X}^{(n)}-\hat{X}\in B_1$,\\

\item  $(\tilde{w}^{(n)},\tilde{q}_w^{(n)})\in B_2$,\\

\item $\tilde{G}^{(n)}-\hat{G}\in B_3$.
\end{enumerate}
\medskip

\noindent Then 
$$(\tilde{w}^{(n+1)},\tilde{q}_w^{(n+1)})\in B_2.$$

\bigskip

\item Let $\tilde{X}^{(n)}-\tilde{\alpha}, \tilde{X}^{(n-1)}-\tilde{\alpha}\in B_1$, $\tilde{w}^{(n)}, \tilde{w}^{(n-1)}\in B_2$ and $\tilde{G}^{(n)}-\tilde{G}_0,\tilde{G}^{(n-1)}-\tilde{G}_0\in B_3$. Then, for a suitable $\varrho>0$

\begin{align*}
&\left\|\tilde{w}^{(n+1)}-\tilde{w}^{(n)}\right\|_{\mathcal{K}^{s+1}_{(0)}}+\left\|\tilde{q}_w^{(n+1)}-\tilde{q}_w^{(n)}\right\|_{\mathcal{K}^{s}_{pr(0)}} \leq C(\tilde{v}_0,\tilde{G}_0)T^{\varrho}\left( \left\|\tilde{X}^{(n)}-\tilde{X}^{(n-1)}\right\|_{\mathcal{F}^{s+1,\gamma}}\right.\\
&\hspace{3cm}\left.+\left\|\tilde{w}^{(n)}-\tilde{w}^{(n-1)}\right\|_{\mathcal{K}^{s+1}_{(0)}}+
\left\|\tilde{q}_w^{(n)}-\tilde{q}_w^{(n-1)}\right\|_{\mathcal{K}^{s}_{pr(0)}} + \left\|\tilde{G}^{(n)}-\tilde{G}^{(n-1)}\right\|_{\mathcal{F}^{s,\gamma-1}}\right) 
\end{align*}
\end{enumerate}
\end{proposition}

\begin{proof}
\textbf{Part 1.}\\

\noindent In this proof we use Lemma \ref{invL}, in particular we have

\begin{align*}
&\|(\tilde{w}^{(n+1)},\tilde{q}_w^{(n+1)})-L^{-1}(\tilde{f}_{\phi}^L+\tilde{f}_{G_0},\bar{g}_{\phi}^L,\tilde{h}_{\phi}^L+\tilde{h}_{G_0})\|_{X_0}\\[2mm]
&\leq C \left\|L^{-1}\left(\tilde{f}^{(n)}-\tilde{f}_{G_0},\bar{g}^{(n)},\tilde{ h}^{(n)}-\tilde{h}_{G_0}\right)\right\|_{X_0}\\[2mm]
&\leq C\left(\|\tilde{f}^{(n)}-\tilde{f}_{G_0}\|_{\mathcal{K}^{s-1}_{(0)}}+\|\bar{g}^{(n)}\|_{\mathcal{\bar{K}}^{s}_{(0)}}+\|\tilde{h}^{(n)}-\tilde{h}_{G_0}\|_{\mathcal{K}^{s-\frac{1}{2}}_{(0)}}\right).
\end{align*}
\noindent  Thus it is sufficient to prove
\begin{align*}
&\|\tilde{f}^{(n)}-\tilde{f}_{G_0}\|_{\mathcal{K}^{s-1}_{(0)}}\leq C\left(\tilde{v}_0,\tilde{G}_0, \|\tilde{w}^{(n)}\|_{\mathcal{K}^{s+1}_{(0)}}, \|\tilde{q}^{(n)}_{w}\|_{\mathcal{K}^{s}_{pr(0)}},\|\tilde{X}^{(n)}-\hat{X}\|_{\mathcal{F}^{s+1,\gamma}},\|\tilde{G}^{(n)}-\hat{G}\|_{\mathcal{F}^{s,\gamma-1}}\right)T^{\delta},\\
&\|\bar{g}^{(n)}\|_{\mathcal{\bar{K}}^{s}_{(0)}}\leq C\left(\tilde{v}_0, \tilde{G}_0, \|\tilde{w}^{(n)}\|_{\mathcal{K}^{s+1}_{(0)}}\right)T^{\theta},\\
&\|\tilde{h}^{(n)}-\tilde{h}_{G_0}\|_{\mathcal{K}^{s-\frac{1}{2}}_{(0)}}\leq C\left(\tilde{v}_0,\tilde{G}_0, \|\tilde{w}^{(n)}\|_{\mathcal{K}^{s+1}_{(0)}}, \|\tilde{q}_w^{(n)}\|_{\mathcal{K}^{s}_{pr(0)}},\|\tilde{X}^{(n)}-\hat{X}\|_{\mathcal{F}^{s+1,\gamma}},\|\tilde{G}^{(n)}-\hat{G}\|_{\mathcal{F}^{s,\gamma-1}}\right)T^{\beta},
\end{align*}
\medskip

\noindent for all $(\tilde{w}^{(n)},\tilde{q}_w^{(n)})\in B_2$. 
\medskip

\noindent\underline{\textbf{Estimate for $\tilde{f}^{(n)}$}}\\

For technical reasons we rewrite $\tilde{f}^{(n)}$

\begin{equation}\label{split-f}
\begin{split}
&\tilde{f}^{(n)}_w=Q^2(\tilde{X}^{(n)})\tilde{\zeta}^{(n)}\partial(\tilde{\zeta}^{(n)}\partial\tilde{w}^{(n)})- Q^2\Delta\tilde{w}^{(n)},\\[2mm]
&\tilde{f}_{\phi}=Q^2(\tilde{X}^{(n)})\tilde{\zeta}^{(n)}\partial(\tilde{\zeta}^{(n)}\partial\phi)- Q^2\Delta\phi,\\[2mm]
&\tilde{f}^{(n)}_q=(J^P)^T\nabla\tilde{q}^{(n)}-J^P(\tilde{X}^{(n)})^T\tilde{\zeta}^{(n)}\nabla\tilde{q}^{(n)},\\[2mm]
&\tilde{f}^{(n)}_G=J^P(\tilde{X}^{(n)})\tilde{G}^{(n)}\tilde{\zeta}^{(n)}\nabla\tilde{G}^{(n)}.
\end{split}
\end{equation}
\medskip

\noindent We notice that $\tilde{f}^{(n)}_w, \tilde{f}_{\phi}, \tilde{f}^{(n)}_q$ have already been studied in \cite[Proposition 5.4]{CCFGG2}, but while $\tilde{f}^{(n)}_w, \tilde{f}^{(n)}_q$ are exactly the same, concerning $\tilde{f}_{\phi}$, we mark that our definition of $\phi$ is different since there is the presence of the initial data for the elastic part. However it does not significantly affect the estimates already obtained by Castro et al. but the constant will depend on $\tilde{G}_0$. For this motivation we show only the estimate of $\tilde{f}^{(n)}_G-\tilde{f}_{G_0}$.

\begin{align*}
\|\tilde{f}^{(n)}_G-\tilde{f}_{G_0}\|_{\mathcal{K}^{s-1}_{(0)}}\leq \|\tilde{f}^{(n)}_G-\tilde{f}_{G_0}\|_{L^2H^{s-1}}+\|\tilde{f}^{(n)}_G-\tilde{f}_{G_0}\|_{H^{\frac{s-1}{2}}_{(0)}L^2}.
\end{align*}
\medskip

\noindent For the estimate in $L^2H^{s-1}$ we use the following splitting

\begin{align*}
\tilde{f}^{(n)}_G-\tilde{f}_{G_0}&=J^P(\tilde{X}^{(n)})(\tilde{G}^{(n)}-\tilde{G}_0)\tilde{\zeta}^{(n)}(\nabla\tilde{G}^{(n)}-\nabla\tilde{G}_0)+J^P(\tilde{X}^{(n)})(\tilde{G}^{(n)}-\tilde{G}_0)\tilde{\zeta}^{(n)}\nabla\tilde{G}_0\\[2mm]
&+J^P(\tilde{X}^{(n)})\tilde{G}_0\tilde{\zeta}^{(n)}(\nabla\tilde{G}^{(n)}-\nabla\tilde{G}_0)+(J^P(\tilde{X}^{(n)})-J^P)\tilde{G}_0\nabla{G}_0\\[2mm]
&+J^P\tilde{G}_0(\tilde{\zeta}^{(n)}-\mathcal{I})\nabla\tilde{G}_0=\sum_{i=1}^5 d_{i}^f.
\end{align*}
\medskip

\noindent Now,  Lemma \ref{Jp-est} and Lemma \ref{zeta-est} with estimates \eqref{flux-estim} and \eqref{defgrad-estim} give the following results

\begin{align*}
&\|d_1^f\|_{L^2H^{s-1}}\leq \|J^P(\tilde{X}^{(n)})\|_{L^{\infty}H^{s-1}}\|\tilde{G}^{(n)}-\tilde{G}_0\|_{L^2H^{s-1}}\|\tilde{\zeta}^{(n)}\|_{L^{\infty}H^{s-1}}\|\nabla\tilde{G}^{(n)}-\nabla\tilde{G}_0\|_{L^{\infty}H^{s-1}}\\
&\hspace{1.9cm}\leq C(\tilde{v}_0,\tilde{G}_0)T^{\frac{1}{4}}\\[4mm]
&\|d_5^f\|_{L^2H^{s-1}}\leq t^{\frac{1}{2}}\|d_5^f\|_{L^{\infty}H^{s-1}}\leq \sup_{t\in[0,T]}t^{\frac{1}{2}}\|J^P\tilde{G}_0(\tilde{\zeta}^{(n)}-\mathcal{I})\nabla\tilde{G}_0\|_{H^{s-1}}\\
&\hspace{1.9cm}\leq T^{\frac{1}{2}} \|\tilde{G}_0\|_{H^{s-1}}\|\nabla\tilde{G}_0\|_{H^{s-1}}\|\tilde{\zeta}^{(n)}-\mathcal{I}\|_{L^{\infty}H^{s-1}}\leq C(\tilde{v}_0,\tilde{G}_0) T^{\frac{3}{4}}
\end{align*}
\medskip

\noindent We shown only the estimates of  $d_1^f$ and $d_5^f$, because the first has all terms depending on time. On the contrary $d_5^f$ has almost all terms independent of time.\\
\medskip

\noindent The estimates in $H^{\frac{s-1}{2}}_{(0)}L^2$ are more delicated and we need to pay attention in adding and subtracting the right terms in order to use the preliminaries lemmas.

\begin{align*}
\tilde{f}^{(n)}_G-\tilde{f}_{G_0}&=(J^P(\tilde{X}^{(n)})-J^P)(\tilde{G}^{(n)}-\tilde{G}_0)(\tilde{\zeta}^{(n)}-\mathcal{I})(\nabla\tilde{G}^{(n)}-\nabla\tilde{G}_0)\\[1mm]
&+J^P(\tilde{G}^{(n)}-\tilde{G}_0)(\tilde{\zeta}^{(n)}-\mathcal{I})(\nabla\tilde{G}^{(n)}-\nabla\tilde{G}_0)\\[1mm]
&+(J^P(\tilde{X}^{(n)})-J^P)\tilde{G}_0(\tilde{\zeta}^{(n)}-\mathcal{I})(\nabla\tilde{G}^{(n)}-\nabla\tilde{G}_0)\\[1mm]
&+J^P\tilde{G}_0(\tilde{\zeta}^{(n)}-\mathcal{I})(\nabla\tilde{G}^{(n)}-\nabla\tilde{G}_0)+(J^P(\tilde{X}^{(n)})-J^P)(\tilde{G}^{(n)}-\tilde{G}_0)(\tilde{\zeta}^{(n)}-\mathcal{I})\nabla\tilde{G}_0\\[1mm]
&+J^P(\tilde{G}^{(n)}-\tilde{G}_0)(\tilde{\zeta}^{(n)}-\mathcal{I})\nabla\tilde{G}_0+(J^P(\tilde{X}^{(n)})-J^P)\tilde{G}_0(\tilde{\zeta}^{(n)}-\mathcal{I})\nabla\tilde{G}_0\\[1mm]
&+J^P\tilde{G}_0(\tilde{\zeta}^{(n)}-\mathcal{I})\nabla\tilde{G}_0+(J^P(\tilde{X}^{(n)})-J^P)(\tilde{G}^{(n)}-\tilde{G}_0)\nabla\tilde{G}_0+J^P(\tilde{G}^{(n)}-\tilde{G}_0)\nabla\tilde{G}_0\\[1mm]
&+(J^P(\tilde{X}^{(n)})-J^P)(\tilde{G}^{(n)}-\tilde{G}_0)(\nabla\tilde{G}^{(n)}-\nabla\tilde{G}_0)+J^P(\tilde{G}^{(n)}-\tilde{G}_0)(\nabla\tilde{G}^{(n)}-\nabla\tilde{G}_0)\\[1mm]
&+(J^P(\tilde{X}^{(n)})-J^P)\tilde{G}_0\nabla\tilde{G}_0+(J^P(\tilde{X}^{(n)})-J^P)\tilde{G}_0(\nabla\tilde{G}^{(n)}-\nabla\tilde{G}_0),\\[1mm]
&+J^P\tilde{G}_0(\nabla\tilde{G}^{(n)}-\nabla\tilde{G}_0)=\sum_{i=1}^{15} d_i^f.
\end{align*}
\medskip

\noindent Since there are a lot of terms, we focus on one with the largest number of terms depending on time as $d_1^f$ and one with the least number of terms depending on time as $d_{10}^f$. We use lemma \ref{lem3}, lemma \ref{lem4}, with $1-\eta'=\frac{1}{p}$ lemma \ref{lem5}, lemma \ref{Jp-est} and lemma \ref{zeta-est}. Moreover, we need estimates \eqref{flux-estim} and \eqref{defgrad-estim} for $d_1^f$. On the contrary for $d_{10}^f$ we use the property of the space $H^{\frac{s-1}{2}}_{(0)}([0,T])$, with lemma \ref{lem4} and \eqref{defgrad-estim}.

\begin{align*}
&\|d_1^f\|_{H^{\frac{s-1}{2}}_{(0)}L^2}\leq \|(J^P(\tilde{X}^{(n)})-J^P)(\tilde{\zeta}^{(n)}-\mathcal{I})\|_{H^{\frac{s-1}{2}}_{(0)}H^{1+\eta}}\|(\tilde{G}^{(n)}-\tilde{G}_0)(\nabla\tilde{G}^{(n)}-\nabla\tilde{G}_0)\|_{H^{\frac{s-1}{2}}_{(0)}L^2}\\[2mm]
&\leq\|J^P(\tilde{X}^{(n)})-J^P\|_{H^{\frac{s-1}{2}}_{(0)}H^{1+\eta}}\|\tilde{\zeta}^{(n)}-\mathcal{I}\|_{H^{\frac{s-1}{2}}_{(0)}H^{1+\eta}}\|\tilde{G}^{(n)}-\tilde{G}_0\|_{H^{\frac{s-1}{2}}_{(0)}H^{1-\eta'}}\|\nabla\tilde{G}^{(n)}-\nabla\tilde{G}_0\|_{H^{\frac{s-1}{2}}_{(0)}H^{\eta'}}\\[2mm]
&\leq C(\tilde{v}_0, \tilde{G}_0) T^{\delta_1},\\[5mm]
&\|d_{10}^f\|_{H^{\frac{s-1}{2}}_{(0)}L^2}\leq \|\tilde{G}^{(n)}-\tilde{G}_0\|_{H^{\frac{s-1}{2}}_{(0)}H^{1-\eta'}}\|\nabla\tilde{G}_0\|_{H^{\eta'}}\leq C(\tilde{G}_0) T^{\delta_{10}}.
\end{align*}
\medskip

\noindent We can state that for $i=1,\ldots,15$, we have $\|d_{i}^f\|_{H^{\frac{s-1}{2}}_{(0)}L^2}\leq C(\tilde{v}_0, \tilde{G}_0) T^{\delta_i}$. In that way we can estimate all the terms and by gathering all together we have

\begin{equation}\label{ball-f}
\|\tilde{f}^{(n)}_w\|_{\mathcal{K}^{s-1}_{(0)}}+\|\tilde{f}^{(n)}_{\phi}\|_{\mathcal{K}^{s-1}_{(0)}}+\|\tilde{f}^{(n)}_q\|_{\mathcal{K}^{s-1}_{(0)}}+\|\tilde{f}^{(n)}_{G}-\tilde{f}_{G_0}\|_{\mathcal{K}^{s-1}_{(0)}}\leq C(\tilde{v}_0,\tilde{G}_0)T^{\delta},
\end{equation}

\noindent where $\delta>0$ is the minimum among all the exponents.
\bigskip

\noindent\underline{\textbf{Estimate for $\bar{g}^{(n)}$}}\\

\noindent We have to estimate this term in $\mathcal{\bar{K}}^s_{(0)}$ and we recall, as we did before, the idea for splitting $\bar{g}^{(n)}$ but for all the computations we recall \cite{CCFGG2}. 

\begin{align*}
\bar{g}^{(n)}=&-\trace\left(\nabla \tilde{w}^{(n)}(\tilde{\zeta}^{(n)}-\mathcal{I}) J^P(\tilde{X}^{(n)})\right)-\trace\left(\nabla\phi(\tilde{\zeta}^{(n)}-\tilde{\zeta}_{\phi})J^P(\tilde{X}^{(n)})\right)\\[2mm]
&+\trace\left(\nabla \tilde{w}^{(n)}(J^P-J^P(\tilde{X}^{(n)}))\right)+\trace\left(\nabla\phi\tilde{\zeta}_{\phi}(J^P_{\phi}-J^P(\tilde{X}^{(n)}))\right).
\end{align*}

\noindent The final estimate is 
\begin{equation}\label{ball-g}
\|\bar{g}^{(n)}\|_{\mathcal{\bar{K}}^s_{(0)}}\leq C(\tilde{v}_0,\tilde{G}_0) T^{\theta}.
\end{equation}
\bigskip

\noindent\underline{\textbf{Estimate for $\tilde{h}^{(n)}$}}\\

We rewrite $\tilde{h}^{(n)}$, as follows

\begin{equation}\label{split-h}
\begin{split}
&\tilde{h}^{(n)}_w=\nabla\tilde{w}^{(n)}\tilde{n}_0-\nabla\tilde{w}^{(n)}\tilde{\zeta}^{(n)}\nabla_{\Lambda}\tilde{X}^{(n)}\tilde{n}_0,\\[2mm]
&\tilde{h}^{(n)}_{w^T}=(\nabla\tilde{w}J^P)^T(J^P)^{-1}\tilde{n}_0-(\nabla\tilde{w}^{(n)}\tilde{\zeta}^{(n)}J^P(\tilde{X}^{(n)}))^T(J^P)^{-1}\nabla_{\Lambda}\tilde{X}^{(n)}\tilde{n}_0,\\[2mm]
&\tilde{h}^{(n)}_{\phi}=\nabla\phi\tilde{n}_0-\nabla\phi\tilde{\zeta}^{(n)}\nabla_{\Lambda}\tilde{X}^{(n)}\tilde{n}_0,\\[2mm]
&\tilde{h}^{(n)}_{\phi^T}=(\nabla\phi J^P)^T(J^P)^{-1}\tilde{n}_0-(\nabla\phi\tilde{\zeta}^{(n)}J^P(\tilde{X}^{(n)}))^T(J^P)^{-1}\nabla_{\Lambda}\tilde{X}^{(n)}\tilde{n}_0,\\[2mm]
&\tilde{h}^{(n)}_{q}=\tilde{q}^{(n)}(J^P(\tilde{X}^{(n)}))^{-1}\nabla_{\Lambda}\tilde{X}^{(n)}\tilde{n}_0-\tilde{q}^{(n)}(J^P)^{-1}\tilde{n}_0,\\[2mm]
&\tilde{h}^{(n)}_G=-(\tilde{G}^{(n)}\tilde{G}^{T(n)}-\mathcal{I})(J^P(\tilde{X}^{(n)}))^{-1}\nabla_{\Lambda}\tilde{X}^{(n)}\tilde{n}_0.
\end{split}
\end{equation}

\noindent For this term we have to estimate separately
$\tilde{h}^{(n)}_v, \tilde{h}^{(n)}_{v^T}, \tilde{h}^{(n)}_q$and $\tilde{h}^{(n)}_{G}-\tilde{h}_{G_0}$. As we show for $\tilde{f}^{(n)}$, the estimates are the same given in \cite{CCFGG2} except for $\tilde{h}^{(n)}_{G}-\tilde{h}_{G_0}$.\\

\begin{align*}
\|\tilde{h}^{(n)}_{G}-\tilde{h}_{G_0}\|_{\mathcal{K}^{s-\frac{1}{2}}_{(0)}}\leq\|\tilde{h}^{(n)}_{G}-\tilde{h}_{G_0}\|_{L^2H^{s-\frac{1}{2}}}+\|\tilde{h}^{(n)}_{G}-\tilde{h}_{G_0}\|_{H^{\frac{s}{2}-\frac{1}{4}}_{(0)}L^2}
\end{align*}
\medskip

\noindent The estimate in $L^2H^{s-\frac{1}{2}}$ is obtained through these terms

\begin{align*}
\tilde{h}^{(n)}_{G}-\tilde{h}_{G_0}&=(\tilde{G}_0-\tilde{G}^{(n)})(\tilde{G}^{T(n)}-\tilde{G}_0^T)(J^P(\tilde{X}^{(n)}))^{-1}\nabla_{\Lambda}\tilde{X}^{(n)}\tilde{n}_0\\[2mm]
&-\tilde{G}_0(\tilde{G}^{T(n)}-\tilde{G}^T_0)(J^P(\tilde{X}^{(n)}))^{-1}\nabla_{\Lambda}\tilde{X}^{(n)}\tilde{n}_0+(\tilde{G}_0-\tilde{G}^{(n)})\tilde{G}_0^T(J^P(\tilde{X}^{(n)}))^{-1}\nabla_{\Lambda}\tilde{X}^{(n)}\tilde{n}_0\\[2mm]
&-\tilde{G}_0\tilde{G}_0^T((J^P(\tilde{X}^{(n)}))^{-1}-(J^P)^{-1})\nabla_{\Lambda}\tilde{X}^{(n)}\tilde{n}_0-\tilde{G}_0\tilde{G}_0^T(J^P)^{-1}(\nabla_{\Lambda}\tilde{X}^{(n)}-\mathcal{I})\tilde{n}_0\\[2mm]
&+((J^P(\tilde{X}^{(n)}))^{-1}-(J^P)^{-1})\nabla_{\Lambda}\tilde{X}^{(n)}\tilde{n}_0+(J^P)^{-1}(\nabla_{\Lambda}\tilde{X}^{(n)}-\mathcal{I})\tilde{n}_0=\sum_{i=1}^7 d_i^h.
\end{align*}
\medskip

\noindent We show only the estimate of the most difficult term $d_1^h$, by using the trace theorem, lemma \ref{Jp-est}, lemma \ref{zeta-est} and \eqref{flux-estim}, \eqref{defgrad-estim}.

\begin{align*}
&\|d_1^h\|_{L^2H^{s-\frac{1}{2}}}\leq \|\tilde{G}_0-\tilde{G}^{(n)}\|_{L^2H^{s-\frac{1}{2}}}\|\tilde{G}^{T(n)}-\tilde{G}_0^T\|_{L^{\infty}H^{s-\frac{1}{2}}}\|(J^P(\tilde{X}^{(n)}))^{-1}\|_{L^{\infty}H^{s-\frac{1}{2}}}\|\nabla_{\Lambda}\tilde{X}^{(n)}\|_{L^{\infty}H^{s-\frac{1}{2}}}\\[2mm]
&\hspace{1.9cm}\leq T^{\frac{3}{4}}\|\tilde{G}_0-\tilde{G}^{(n)}\|_{L^{\infty}H^s}\|\tilde{G}^{T(n)}\|_{L^{\infty}H^s}\|(J^P(\tilde{X}^{(n)}))^{-1}\|_{L^{\infty}H^s}\|\nabla_{\Lambda}\tilde{X}^{(n)}\|_{L^{\infty}H^s}\\[2mm]
&\hspace{1.9cm}\leq C(\tilde{v}_0, \tilde{G}_0) T^{\frac{3}{4}}.
\end{align*}
\medskip

\noindent So we deduce $\sum_{i=2}^7\|d_i^h\|_{L^2H^{s-\frac{1}{2}}}\leq C(\tilde{v}_0, \tilde{G}_0) T^{\frac{3}{4}}.$
\bigskip

\noindent For the estimates in $H^{\frac{s}{2}-\frac{1}{4}}_{(0)}([0,T])L^2(\partial\tilde{\Omega}_0)$, we use lemma \ref{lem3}, with $\eta>\frac{1}{2}$, lemma \ref{lem5}, lemma \ref{lem4}, with $\frac{1}{p}=\frac{1}{2}-\eta'$ and $\frac{1}{q}=\frac{1}{2}+\eta'$, Lemma \ref{lem6} and trace theorem \ref{parabolic-trace}. Furthermore we need also lemma \ref{Jp-est}, lemma \ref{zeta-est} and \eqref{flux-estim}, \eqref{defgrad-estim}. Thus in order to satisfy the hypothesis of these lemmas we need an appropriate splitting.

\begin{align*}
&\tilde{h}^{(n)}_{G}-\tilde{h}_{G_0}=(\tilde{G}_0-\tilde{G}^{(n)})(\tilde{G}^{T(n)}-\tilde{G}_0^T)((J^P(\tilde{X}^{(n)}))^{-1}-(J^P)^{-1})(\nabla_{\Lambda}\tilde{X}^{(n)}-\mathcal{I})\tilde{n}_0\\[2mm]
&\hspace{0.5cm}+(\tilde{G}_0-\tilde{G}^{(n)})(\tilde{G}^{T(n)}-\tilde{G}_0^T)(J^P)^{-1}(\nabla_{\Lambda}\tilde{X}^{(n)}-\mathcal{I})\tilde{n}_0\\[2mm]
&\hspace{0.5cm}-\tilde{G}_0(\tilde{G}^{T(n)}-\tilde{G}_0^T)((J^P(\tilde{X}^{(n)}))^{-1}-(J^P)^{-1})(\nabla_{\Lambda}\tilde{X}^{(n)}-\mathcal{I})\tilde{n}_0\\[2mm]
&\hspace{0.5cm}-\tilde{G}_0(\tilde{G}^{T(n)}-\tilde{G}_0^T)(J^P)^{-1}(\nabla_{\Lambda}\tilde{X}^{(n)}-\mathcal{I})\tilde{n}_0\\[2mm]
&\hspace{0.5cm}+(\tilde{G}_0-\tilde{G}^{(n)})\tilde{G}_0^T((J^P(\tilde{X}^{(n)}))^{-1}-(J^P)^{-1})(\nabla_{\Lambda}\tilde{X}^{(n)}-\mathcal{I})\tilde{n}_0\\[2mm]
&\hspace{0.5cm}+(\tilde{G}_0-\tilde{G}^{(n)})\tilde{G}_0^T(J^P)^{-1}(\nabla_{\Lambda}\tilde{X}^{(n)}-\mathcal{I})\tilde{n}_0-\tilde{G}_0\tilde{G}_0^T((J^P(\tilde{X}^{(n)}))^{-1}-(J^P)^{-1})(\nabla_{\Lambda}\tilde{X}^{(n)}-\mathcal{I})\tilde{n}_0,\\[2mm]
&\hspace{0.5cm}-\tilde{G}_0\tilde{G}_0^T)(J^P)^{-1}(\nabla_{\Lambda}\tilde{X}^{(n)}-\mathcal{I})\tilde{n}_0+(\tilde{G}_0-\tilde{G}^{(n)})\tilde{G}_0^T((J^P(\tilde{X}^{(n)}))^{-1}-(J^P)^{-1})\tilde{n}_0\\[2mm]
&\hspace{0.5cm}+(\tilde{G}_0-\tilde{G}^{(n)})\tilde{G}_0^T(J^P)^{-1}\tilde{n}_0+(\tilde{G}_0-\tilde{G}^{(n)})(\tilde{G}^{T(n)}-\tilde{G}_0^T)((J^P(\tilde{X}^{(n)}))^{-1}-(J^P)^{-1})\tilde{n}_0\\[2mm]
&\hspace{0.5cm}+(\tilde{G}_0-\tilde{G}^{(n)})(\tilde{G}^{T(n)}-\tilde{G}_0^T)(J^P)^{-1}\tilde{n}_0+\tilde{G}_0\tilde{G}_0^T((J^P(\tilde{X}^{(n)}))^{-1}-(J^P)^{-1})\tilde{n}_0\\[2mm]
&\hspace{0.5cm}-\tilde{G}_0(\tilde{G}^{T(n)}-\tilde{G}_0^T)((J^P(\tilde{X}^{(n)}))^{-1}-(J^P)^{-1})\tilde{n}_0-\tilde{G}_0(\tilde{G}^{T(n)}-\tilde{G}_0^T)(J^P)^{-1}\tilde{n}_0\\[2mm]
&\hspace{0.5cm}+((J^P(\tilde{X}^{(n)}))^{-1}-(J^P)^{-1})(\nabla_{\Lambda}\tilde{X}^{(n)}-\mathcal{I})\tilde{n}_0+(J^P)^{-1}(\nabla_{\Lambda}\tilde{X}^{(n)}-\mathcal{I})\tilde{n}_0\\[2mm]
&\hspace{0.5cm}+((J^P(\tilde{X}^{(n)}))^{-1}-(J^P)^{-1})\tilde{n}_0=\sum_{i=1}^{18} d_i^h.
\end{align*}
\medskip

\noindent We focus on the estimate of $d_1^h$, where we have more terms

\begin{align*}
&\|d_1^h\|_{H^{\frac{s}{2}-\frac{1}{4}}_{(0)}L^2}\leq\|(\tilde{G}_0-\tilde{G}^{(n)})(\tilde{G}^{T(n)}-\tilde{G}_0^T)\|_{H^{\frac{s}{2}-\frac{1}{4}}_{(0)}L^2}\|((J^P(\tilde{X}^{(n)}))^{-1}-(J^P)^{-1})(\nabla_{\Lambda}\tilde{X}^{(n)}-\mathcal{I})\|_{H^{\frac{s}{2}-\frac{1}{4}}_{(0)}H^{\frac{1}{2}+\eta}}\\[2mm]
&\leq\|\tilde{G}_0-\tilde{G}^{(n)}\|_{H^{\frac{s}{2}-\frac{1}{4}}_{(0)}H^{\frac{1}{2}+\eta'}}\|\tilde{G}^{T(n)}-\tilde{G}_0^T\|_{H^{\frac{s}{2}-\frac{1}{4}}_{(0)}H^{\frac{1}{2}-\eta'}}\|J^P(\tilde{X}^{(n)}))^{-1}-(J^P)^{-1}\|_{H^{\frac{s}{2}-\frac{1}{4}}_{(0)}H^{\frac{1}{2}+\eta}}
\end{align*}
\begin{align*}
&\hspace{0.5cm}\cdot\|\nabla_{\Lambda}\tilde{X}^{(n)}-\mathcal{I}\|_{H^{\frac{s}{2}-\frac{1}{4}}_{(0)}H^{\frac{1}{2}+\eta}}\leq \|\tilde{G}_0-\tilde{G}^{(n)}\|_{H^{\frac{s}{2}-\frac{1}{4}}_{(0)}H^{1+\eta'}}\|\tilde{G}^{T(n)}-\tilde{G}_0^T\|_{H^{\frac{s}{2}-\frac{1}{4}}_{(0)}H^{1-\eta'}}\\[2mm]
&\hspace{0.5cm}\cdot\|J^P(\tilde{X}^{(n)}))^{-1}-(J^P)^{-1}\|_{H^{\frac{s}{2}-\frac{1}{4}}_{(0)}H^{1+\eta}}\|\nabla_{\Lambda}\tilde{X}^{(n)}-\mathcal{I}\|_{H^{\frac{s}{2}-\frac{1}{4}}_{(0)}H^{1+\eta}}\leq C(\tilde{v}_0,\tilde{G}_0) T^{\beta_1}.
\end{align*}
\medskip

\noindent However $\sum_{i=2}^{18}\|d_i^h\|_{H^{\frac{s}{2}-\frac{1}{4}}_{(0)}L^2}\leq \sum_{i=2}^{18} C(\tilde{v}_0,\tilde{G}_0) T^{\beta_i}$. Thus for a suitable $\beta>0$

\begin{equation}\label{ball-h}
\begin{split}
&\|\tilde{h}^{(n)}_w\|_{\mathcal{K}^{s-\frac{1}{2}}_{(0)}}+\|\tilde{h}^{(n)}_{w^T}\|_{\mathcal{K}^{s-\frac{1}{2}}_{(0)}}+\|\tilde{h}^{(n)}_{\phi}\|_{\mathcal{K}^{s-\frac{1}{2}}_{(0)}}+\|\tilde{h}^{(n)}_{\phi^T}\|_{\mathcal{K}^{s-\frac{1}{2}}_{(0)}}+
\|\tilde{h}^{(n)}_q\|_{\mathcal{K}^{s-\frac{1}{2}}_{(0)}}\\[2mm]
&\hspace{5.5cm}+\|\tilde{h}^{(n)}_G-\tilde{h}_{G_0}\|_{\mathcal{K}^{s-\frac{1}{2}}_{(0)}}\leq C(\tilde{v}_0,\tilde{G}_0) T^{\beta}
\end{split}
\end{equation}
\medskip

\noindent We can put  together the estimates \eqref{ball-f}, \eqref{ball-g}, \eqref{ball-h} and choose, as in Proposition \ref{Gn}, $\varrho=\min\{\delta,\theta,\beta\}$ in order to get the thesis of Part 1.
\medskip

\noindent\textbf{Part 2.}\\

\noindent For this part we have to take the differences so the terms $\tilde{f}^L_{\phi}, \bar{g}^L_{\phi}, \tilde{h}^L_{\phi}$ desappear. Then it is enough to show

\begin{align*}
&\|\tilde{f}^{(n)}-\tilde{f}^{(n-1)}\|_{\mathcal{K}^{s-1}_{(0)}}\leq C(\tilde{v}_0,\tilde{G}_0)T^{\delta}\left( \|\tilde{w}^{(n)}-\tilde{w}^{(n-1)}\|_{\mathcal{K}^{s+1}_{(0)}}+\|\tilde{q}_w^{(n)}-\tilde{q}_w^{(n-1)}\|_{\mathcal{K}^{s}_{pr(0)}}\right.\\[2mm]
&\hspace{6.1cm}\left.+\|\tilde{G}^{(n)}-\tilde{G}^{(n-1)}\|_{\mathcal{F}^{s,\gamma-1}}+\|\tilde{X}^{(n)}-\tilde{X}^{(n-1)}\|_{\mathcal{F}^{s+1,\gamma}}\right),\\[5mm]
&\|\tilde{g}^{(n)}-\tilde{g}^{(n-1)}\|_{\mathcal{\bar{K}}^{s}_{(0)}}\leq C(\tilde{v}_0,\tilde{G}_0)T^{\theta}\left(\|\tilde{w}^{(n)}-\tilde{w}^{(n-1)}\|_{\mathcal{K}^{s+1}_{(0)}}+\|\tilde{X}^{(n)}-\tilde{X}^{(n-1)}\|_{\mathcal{F}^{s+1,\gamma}}\right),\\[5mm]
&\|\tilde{h}^{(n)}-\tilde{h}^{(n-1)}\|_{\mathcal{K}^{s-\frac{1}{2}}_{(0)}}\leq C(\tilde{v}_0,\tilde{G}_0) T^{\beta}\left( \|\tilde{w}^{(n)}-\tilde{w}^{(n-1)}\|_{\mathcal{K}^{s+1}_{(0)}}+\|\tilde{q}_w^{(n)}-\tilde{q}_w^{(n-1)}\|_{\mathcal{K}^{s}_{pr(0)}}\right.\\[2mm]
&\hspace{6.1cm}\left.+\|\tilde{G}^{(n)}-\tilde{G}^{(n-1)}\|_{\mathcal{F}^{s,\gamma-1}}+\|\tilde{X}^{(n)}-\tilde{X}^{(n-1)}\|_{\mathcal{F}^{s+1,\gamma}}\right),
\end{align*}

\medskip

\noindent \underline{ \textbf{Estimate for $\tilde{f}^{(n)}-\tilde{f}^{(n-1)}$}}\\ \\

\noindent The term $\tilde{f}^{(n)}$ could be split in four terms $\tilde{f}^{(n)}=\tilde{f}^{(n)}_w+\tilde{f}^{(n)}_{\phi}+\tilde{f}^{(n)}_q+\tilde{f}^{(n)}_G$, as we already show in \eqref{split-f}. The estimates for $\tilde{f}^{(n)}_w-\tilde{f}^{(n-1)}_w$, $\tilde{f}^{(n)}_{\phi}-\tilde{f}^{(n-1)}_{\phi}$, $\tilde{f}^{(n)}_q-\tilde{f}^{(n-1)}_q$ can be found in \cite[Proposition 5.4]{CCFGG2}, so we only consider $\tilde{f}^{(n)}_G-\tilde{f}^{(n-1)}_G=J^P(\tilde{X}^{(n)})\tilde{G}^{(n)}\tilde{\zeta}^{(n)}\nabla\tilde{G}^{(n)}-J^P(\tilde{X}^{(n-1)})\tilde{G}^{(n-1)}\tilde{\zeta}^{(n-1)}\nabla\tilde{G}^{(n-1)}$. 

\begin{align*}
\|\tilde{f}^{(n)}_G-\tilde{f}^{(n-1)}_G\|_{\mathcal{K}^{s-1}_{(0)}}\leq \|\tilde{f}^{(n)}_G-\tilde{f}^{(n-1)}_G\|_{L^2H^{s-1}}+\|\tilde{f}^{(n)}_G-\tilde{f}^{(n-1)}_G\|_{H^{\frac{s-1}{2}}_{(0)}L^2}.
\end{align*}

\noindent For the estimate in $L^2H^{s-1}$, we have to split $\tilde{f}^{(n)}_G-\tilde{f}^{(n-1)}_G$ in such a way that we can use lemma \ref{Jp-est},lemma \ref{Jp-dif-est}, lemma \ref{zeta-est}  and lemma \ref{zeta-dif-est} and estimate \eqref{defgrad-estim}.

\begin{align*}
&\tilde{f}^{(n)}_G-\tilde{f}^{(n-1)}_G=(J^P(\tilde{X}^{(n)})-J^P(\tilde{X}^{(n-1)}))(\tilde{G}^{(n)}-\tilde{G}_0)\tilde{\zeta}^{(n)}(\nabla \tilde{G}^{(n)}-\nabla\tilde{G}_0)\\[1mm]
&\hspace{0.5cm}+(J^P(\tilde{X}^{(n)})-J^P(\tilde{X}^{(n-1)}))\tilde{G}_0\tilde{\zeta}^{(n)}(\nabla \tilde{G}^{(n)}-\nabla\tilde{G}_0)\\[1mm]
&\hspace{0.5cm}+(J^P(\tilde{X}^{(n)})-J^P(\tilde{X}^{(n-1)}))(\tilde{G}^{(n)}-\tilde{G}_0)\tilde{\zeta}^{(n)}\nabla\tilde{G}_0\\[1mm]
&\hspace{0.5cm}+(J^P(\tilde{X}^{(n)})-J^P(\tilde{X}^{(n-1)}))\tilde{G}_0\tilde{\zeta}^{(n)}\nabla\tilde{G}_0+J^P(\tilde{X}^{(n-1)})(\tilde{G}^{(n)}-\tilde{G}^{(n-1)})\tilde{\zeta}^{(n)}(\nabla \tilde{G}^{(n)}-\nabla\tilde{G}_0)\\[1mm]
&\hspace{0.5cm}+J^P(\tilde{X}^{(n-1)})(\tilde{G}^{(n)}-\tilde{G}^{(n-1)})\tilde{\zeta}^{(n)}\nabla\tilde{G}_0\\[1mm]
&\hspace{0.5cm}+J^P(\tilde{X}^{(n-1)})(\tilde{G}^{(n-1)}-\tilde{G}_0)(\tilde{\zeta}^{(n)}-\tilde{\zeta}^{(n-1)})(\nabla \tilde{G}^{(n)}-\nabla\tilde{G}_0)\\[1mm]
&\hspace{0.5cm}+J^P(\tilde{X}^{(n-1)})(\tilde{G}^{(n-1)}-\tilde{G}_0)(\tilde{\zeta}^{(n)}-\tilde{\zeta}^{(n-1)})\nabla\tilde{G}_0\\[1mm]
&\hspace{0.5cm}+J^P(\tilde{X}^{(n-1)})\tilde{G}_0(\tilde{\zeta}^{(n)}-\tilde{\zeta}^{(n-1)})(\nabla \tilde{G}^{(n)}-\nabla\tilde{G}_0)\\[1mm]
&\hspace{0.5cm}+J^P(\tilde{X}^{(n-1)})\tilde{G}_0(\tilde{\zeta}^{(n)}-\tilde{\zeta}^{(n-1)})\nabla\tilde{G}_0+J^P(\tilde{X}^{(n-1)})(\tilde{G}^{(n-1)}-\tilde{G}_0)\tilde{\zeta}^{(n-1)}(\nabla \tilde{G}^{(n)}-\nabla\tilde{G}^{(n-1)})\\[1mm]
&\hspace{0.5cm}+J^P(\tilde{X}^{(n-1)})\tilde{G}_0\tilde{\zeta}^{(n-1)}(\nabla \tilde{G}^{(n)}-\nabla\tilde{G}^{(n-1)})=\sum_{i=1}^{12} d_i^f.
\end{align*}
\medskip

\noindent We reserve the right to estimate just the most difficult terms, such as  $d_1^f, d_5^f, d_7^f$ and $d_{11}^f$.

\begin{align*}
&\|d_{1}^f\|_{L^2H^{s-1}}\\[1mm]
&\hspace{0.5cm}\leq \|J^P(\tilde{X}^{(n)})-J^P(\tilde{X}^{(n-1)})\|_{L^{\infty}H^{s-1}}\|\tilde{G}^{(n)}-\tilde{G}_0\|_{L^2H^{s-1}}\|\tilde{\zeta}^{(n)}\|_{L^{\infty}H^{s-1}}\|\nabla \tilde{G}^{(n)}-\nabla\tilde{G}_0\|_{L^{\infty}H^{s-1}}\\[2mm]
&\hspace{0.5cm}\leq C(\tilde{v}_0)\|\tilde{X}^{(n)}-\tilde{X}^{(n-1)}\|_{L^{\infty}H^{s-1}}T^{\frac{1}{2}}\|\tilde{G}^{(n)}-\tilde{G}_0\|_{L^{\infty}H^{s-1}}\|\tilde{G}^{(n)}-\tilde{G}_0\|_{L^{\infty}H^{s}}\\[2mm]
&\hspace{0.5cm}\leq C(\tilde{v}_0,\tilde{G}_0) T^{\frac{3}{4}}\|\tilde{X}^{(n)}-\tilde{X}^{(n-1)}\|_{L^{\infty}_{\frac{1}{4}}H^{s+1}}\leq C(\tilde{v}_0,\tilde{G}_0) T^{\frac{3}{4}}\|\tilde{X}^{(n)}-\tilde{X}^{(n-1)}\|_{\mathcal{F}^{s+1,\gamma}},\\[5mm]
&\|d_{5}^f\|_{L^2H^{s-1}}\leq\|J^P(\tilde{X}^{(n-1)})\|_{L^{\infty}H^{s-1}}\|\tilde{G}^{(n)}-\tilde{G}^{(n-1)}\|_{L^{\infty}H^{s-1}}\|\tilde{\zeta}^{(n)}\|_{L^{\infty}H^{s-1}}\|\nabla \tilde{G}^{(n)}-\nabla\tilde{G}_0\|_{L^2H^{s-1}}\\[2mm]
&\hspace{0.5cm}\leq C(\tilde{v}_0,\tilde{G}_0) T^{\frac{3}{4}}\|\tilde{G}^{(n)}-\tilde{G}^{(n-1)}\|_{L^{\infty}_{\frac{1}{4}}H^{s}}\leq C(\tilde{v}_0,\tilde{G}_0) T^{\frac{3}{4}}\|\tilde{G}^{(n)}-\tilde{G}^{(n-1)}\|_{\mathcal{F}^{s,\gamma-1}},\\[5mm]
&\|d_{7}^f\|_{L^2H^{s-1}}\\[1mm]
&\hspace{0.5cm}\leq \|J^P(\tilde{X}^{(n-1)})\|_{L^{\infty}H^{s-1}}\|\tilde{G}^{(n-1)}-\tilde{G}_0\|_{L^2H^{s-1}}\|\tilde{\zeta}^{(n)}-\tilde{\zeta}^{(n-1)}\|_{L^{\infty}H^{s-1}}\|\nabla \tilde{G}^{(n)}-\nabla\tilde{G}_0\|_{L^{\infty}H^{s-1}}\\[2mm]
&\hspace{0.5cm}\leq C(\tilde{v}_0,\tilde{G}_0) T^{\frac{3}{4}}\|\tilde{X}^{(n)}-\tilde{X}^{(n-1)}\|_{L^{\infty}_{\frac{1}{4}}H^{s}}\leq C(\tilde{v}_0,\tilde{G}_0) T^{\frac{3}{4}}\|\tilde{X}^{(n)}-\tilde{X}^{(n-1)}\|_{\mathcal{F}^{s+1,\gamma}},\\[5mm]
&\|d_{11}^f\|_{L^2H^{s-1}}\\[1mm]
&\hspace{0.5cm}\leq\|J^P(\tilde{X}^{(n-1)})\|_{L^{\infty}H^{s-1}}\|\tilde{G}^{(n-1)}-\tilde{G}_0\|_{L^2H^{s-1}}\|\tilde{\zeta}^{(n-1)}\|_{L^{\infty}H^{s-1}}\|\nabla \tilde{G}^{(n)}-\nabla\tilde{G}^{(n-1)}\|_{L^{\infty}H^{s-1}}\\[2mm]
&\hspace{0.5cm}\leq C(\tilde{v}_0,\tilde{G}_0)T^{\frac{1}{2}}\|\tilde{G}^{(n)}-\tilde{G}^{(n-1)}\|_{L^{\infty}H^{s}}\leq C(\tilde{v}_0,\tilde{G}_0)T^{\frac{3}{4}}\|\tilde{G}^{(n)}-\tilde{G}^{(n-1)}\|_{\mathcal{F}^{s,\gamma-1}}.
\end{align*}
\medskip

\noindent It remains to show the estimates in $H^{\frac{s-1}{2}}_{(0)}L^2$, by appropriate splittings. First of all we separate $\tilde{f}^{(n)}_G-\tilde{f}^{(n-1)}_G$ as below

\begin{align*}
&I_1=(J^P(\tilde{X}^{(n)})-J^P(\tilde{X}^{(n-1)}))\tilde{G}^{(n)}\tilde{\zeta}^{(n)}\nabla \tilde{G}^{(n)},\\[2mm]
&I_2=J^P(\tilde{X}^{(n-1)})(\tilde{G}^{(n)}-\tilde{G}^{(n-1)})\tilde{\zeta}^{(n)}\nabla \tilde{G}^{(n)},\\[2mm]
&I_3=J^P(\tilde{X}^{(n-1)})\tilde{G}^{(n-1)}(\tilde{\zeta}^{(n)}-\tilde{\zeta}^{(n-1)})\nabla \tilde{G}^{(n)},\\[2mm]
&I_4=J^P(\tilde{X}^{(n-1)})\tilde{G}^{(n-1)}\tilde{\zeta}^{(n-1)})(\nabla \tilde{G}^{(n)}-\nabla \tilde{G}^{(n-1)}).
\end{align*}
\medskip
\noindent For these four terms we have to do estimates by using preiminary lemmas which require some zero conditions at $t=0$, for that reason we split again, in such a way that all the hypothesis of lemmas are satisfied. For $I_1$, we have

\begin{align*}
I_1&=(J^P(\tilde{X}^{(n)})-J^P(\tilde{X}^{(n-1)}))(\tilde{G}^{(n)}-\tilde{G}_0)(\tilde{\zeta}^{(n)}-\mathcal{I})(\nabla \tilde{G}^{(n)}-\nabla\tilde{G}_0)\\[1mm]
&+(J^P(\tilde{X}^{(n)})-J^P(\tilde{X}^{(n-1)}))\tilde{G}_0(\tilde{\zeta}^{(n)}-\mathcal{I})(\nabla \tilde{G}^{(n)}-\nabla\tilde{G}_0)\\[1mm]
&+(J^P(\tilde{X}^{(n)})-J^P(\tilde{X}^{(n-1)}))(\tilde{G}^{(n)}-\tilde{G}_0)(\tilde{\zeta}^{(n)}-\mathcal{I})\nabla\tilde{G}_0\\[1mm]
&+(J^P(\tilde{X}^{(n)})-J^P(\tilde{X}^{(n-1)}))\tilde{G}_0(\tilde{\zeta}^{(n)}-\mathcal{I})\nabla\tilde{G}_0\\[1mm]
&+(J^P(\tilde{X}^{(n)})-J^P(\tilde{X}^{(n-1)}))(\tilde{G}^{(n)}-\tilde{G}_0)(\nabla \tilde{G}^{(n)}-\nabla\tilde{G}_0)\\[1mm]
&+(J^P(\tilde{X}^{(n)})-J^P(\tilde{X}^{(n-1)}))\tilde{G}_0(\nabla \tilde{G}^{(n)}-\nabla\tilde{G}_0)\\[1mm]
&+(J^P(\tilde{X}^{(n)})-J^P(\tilde{X}^{(n-1)}))(\tilde{G}^{(n)}-\tilde{G}_0)\nabla\tilde{G}_0+(J^P(\tilde{X}^{(n)})-J^P(\tilde{X}^{(n-1)}))\tilde{G}_0\nabla\tilde{G}_0=\sum_{i=1}^{8}d_{i}^f
\end{align*}
\medskip

\noindent We want to show the most significant estimate, which is $d_1^f$, by using lemma \ref{lem3}, lemma \ref{lem5}, lemma \ref{lem4}, with $\frac{1}{q}=\eta'$. Moreover we need also lemma \ref{Jp-dif-est}, lemma \ref{zeta-est}, lemma \ref{lem2} and lemma \ref{lem6} but also \eqref{flux-estim} and \eqref{defgrad-estim}.

\begin{align*}
&\|d_1^f\|_{H^{\frac{s-1}{2}}_{(0)}L^2}\leq\|(J^P(\tilde{X}^{(n)})-J^P(\tilde{X}^{(n-1)}))(\tilde{\zeta}^{(n)}-\mathcal{I})\|_{H^{\frac{s-1}{2}}_{(0)}H^{1+\eta}}\|(\tilde{G}^{(n)}-\tilde{G}_0)(\nabla \tilde{G}^{(n)}-\nabla\tilde{G}_0)\|_{H^{\frac{s-1}{2}}_{(0)}L^2}\\[2mm]
&\leq \|J^P(\tilde{X}^{(n)})-J^P(\tilde{X}^{(n-1)})\|_{H^{\frac{s-1}{2}}_{(0)}H^{1+\eta}}\|\tilde{\zeta}^{(n)}-\mathcal{I}\|_{H^{\frac{s-1}{2}}_{(0)}H^{1+\eta}}\|\tilde{G}^{(n)}-\tilde{G}_0\|_{H^{\frac{s-1}{2}}_{(0)}H^{1-\eta'}}\|\nabla \tilde{G}^{(n)}-\nabla\tilde{G}_0\|_{H^{\frac{s-1}{2}}_{(0)}H^{\eta'}}\\[2mm]
&\leq C(\tilde{v}_0,\tilde{G}_0) \|\tilde{X}^{(n)}-\tilde{X}^{(n)}\|_{H^{\frac{s-1}{2}}_{(0)}H^{1+\eta}}\leq C(\tilde{v}_0,\tilde{G}_0) \left\|\partial_t\int_0^t\tilde{X}^{(n)}-\tilde{X}^{(n)}\right\|_{H^{\frac{s-1}{2}+\delta_1-\delta_1}_{(0)}H^{1+\eta}}\\[2mm]
&\leq C(\tilde{v}_0,\tilde{G}_0) T^{\delta_1}\|\tilde{X}^{(n)}-\tilde{X}^{(n-1)}\|_{H^{\frac{s-1}{2}+\delta_1}_{(0)}H^{1+\eta}}\leq C(\tilde{v}_0,\tilde{G}_0) T^{\delta_1}\|\tilde{X}^{(n)}-\tilde{X}^{(n-1)}\|_{\mathcal{F}^{s+1,\gamma}}.
\end{align*}
\medskip

\noindent For $i=2,\ldots, 8$, we have $\|d_i^f\|_{H^{\frac{s-1}{2}}_{(0)}L^2}\leq C(\tilde{v}_0,\tilde{G}_0) T^{\delta_i}\|\tilde{X}^{(n)}-\tilde{X}^{(n)}\|_{\mathcal{F}^{s+1,\gamma}}.$ Then we rewrite $I_2$, as follows

\begin{align*}
I_2&=(J^P(\tilde{X}^{(n-1)})-J^P)(\tilde{G}^{(n)}-\tilde{G}^{(n-1)})(\tilde{\zeta}^{(n)}-\mathcal{I})(\nabla \tilde{G}^{(n)}-\nabla\tilde{G}_0)\\[1mm]
&+J^P(\tilde{G}^{(n)}-\tilde{G}^{(n-1)})(\tilde{\zeta}^{(n)}-\mathcal{I})(\nabla \tilde{G}^{(n)}-\nabla\tilde{G}_0)\\[1mm]
&+(J^P(\tilde{X}^{(n-1)})-J^P)(\tilde{G}^{(n)}-\tilde{G}^{(n-1)})(\tilde{\zeta}^{(n)}-\mathcal{I})\nabla\tilde{G}_0\\[1mm]
&+J^P(\tilde{G}^{(n)}-\tilde{G}^{(n-1)})(\tilde{\zeta}^{(n)}-\mathcal{I})\nabla\tilde{G}_0+(J^P(\tilde{X}^{(n-1)})-J^P)(\tilde{G}^{(n)}-\tilde{G}^{(n-1)})(\nabla \tilde{G}^{(n)}-\nabla\tilde{G}_0)\\[1mm]
&+J^P(\tilde{G}^{(n)}-\tilde{G}^{(n-1)})(\nabla \tilde{G}^{(n)}-\nabla\tilde{G}_0)+(J^P(\tilde{X}^{(n-1)})-J^P)(\tilde{G}^{(n)}-\tilde{G}^{(n-1)})\nabla\tilde{G}_0\\[1mm]
&+J^P(\tilde{G}^{(n)}-\tilde{G}^{(n-1)})\nabla\tilde{G}_0=\sum_{i=9}^{16} d_i^f.
\end{align*}
\medskip

\noindent By using lemma \ref{lem3}, lemma \ref{lem5}, lemma \ref{lem4}, with $\frac{1}{q}=\eta'$, in addition lemma \ref{Jp-est}, lemma \ref{zeta-est} and \eqref{flux-estim}, \eqref{defgrad-estim} and finally lemma \ref{lem2} and lemma \ref{lem6}, we are able to estimate $d_9^f$, for all the other terms we need to use only some of these lemmas.

\begin{align*}
&\|d_9^f\|_{H^{\frac{s-1}{2}}_{(0)}L^2}\leq\|(J^P(\tilde{X}^{(n-1)})-J^P)(\tilde{\zeta}^{(n)}-\mathcal{I})\|_{H^{\frac{s-1}{2}}_{(0)}H^{1+\eta}}\|(\tilde{G}^{(n)}-\tilde{G}^{(n-1)})(\nabla \tilde{G}^{(n)}-\nabla\tilde{G}_0)\|_{H^{\frac{s-1}{2}}_{(0)}L^2}\\[2mm]
&\leq \|J^P(\tilde{X}^{(n-1)})-J^P\|_{H^{\frac{s-1}{2}}_{(0)}H^{1+\eta}}\|\tilde{\zeta}^{(n)}-\mathcal{I}\|_{H^{\frac{s-1}{2}}_{(0)}H^{1+\eta}}\|\tilde{G}^{(n)}-\tilde{G}^{(n-1)}\|_{H^{\frac{s-1}{2}}_{(0)}H^{1-\eta'}}\|\nabla \tilde{G}^{(n)}-\nabla\tilde{G}_0\|_{H^{\frac{s-1}{2}}_{(0)}H^{\eta'}}\\[2mm]
&\leq C(\tilde{v}_0,\tilde{G}_0) \left\|\partial_t\int_0^t\tilde{G}^{(n)}-\tilde{G}^{(n-1)}\right\|_{H^{\frac{s-1}{2}+\delta_9-\delta_9}_{(0)}H^{1-\eta'}}\leq C(\tilde{v}_0,\tilde{G}_0) T^{\delta_9}\|\tilde{G}^{(n)}-\tilde{G}^{(n-1)}\|_{H^{\frac{s-1}{2}+\delta_9}_{(0)}H^{1+\eta'}}\\[2mm]
&\leq C(\tilde{v}_0,\tilde{G}_0) T^{\delta_9}\|\tilde{G}^{(n)}-\tilde{G}^{(n-1)}\|_{\mathcal{F}^{s+1,\gamma}}.
\end{align*}
\medskip

\noindent Similar results occur for $d_i^f$, with $i=10,\ldots,16$. In particular $\|d_i^f\|_{H^{\frac{s-1}{2}}_{(0)}L^2}\leq C(\tilde{v}_0,\tilde{G}_0) T^{\delta_i}\|\tilde{G}^{(n)}-\tilde{G}^{(n-1)}\|_{\mathcal{F}^{s+1,\gamma}}.$ For $I_3$ the splitting is given by these terms

\begin{align*}
I_3&=(J^P(\tilde{X}^{(n-1)})-J^P)(\tilde{G}^{(n)}-\tilde{G}_0)(\tilde{\zeta}^{(n)}-\tilde{\zeta}^{(n-1)})(\nabla \tilde{G}^{(n)}-\nabla\tilde{G}_0)\\[2mm]
&+J^P(\tilde{G}^{(n-1)}-\tilde{G}_0)(\tilde{\zeta}^{(n)}-\tilde{\zeta}^{(n-1)})(\nabla \tilde{G}^{(n)}-\nabla\tilde{G}_0)\\[2mm]
&+(J^P(\tilde{X}^{(n-1)})-J^P)(\tilde{G}^{(n-1)}-\tilde{G}_0)(\tilde{\zeta}^{(n)}-\tilde{\zeta}^{(n-1)})\nabla\tilde{G}_0+J^P(\tilde{G}^{(n)}-\tilde{G}_0)(\tilde{\zeta}^{(n)}-\tilde{\zeta}^{(n-1)})\nabla\tilde{G}_0\\[2mm]
&+(J^P(\tilde{X}^{(n-1)})-J^P)\tilde{G}_0(\tilde{\zeta}^{(n)}-\tilde{\zeta}^{(n-1)})(\nabla \tilde{G}^{(n)}-\nabla\tilde{G}_0)+J^P\tilde{G}_0(\tilde{\zeta}^{(n)}-\tilde{\zeta}^{(n-1)})(\nabla \tilde{G}^{(n)}-\nabla\tilde{G}_0)\\[2mm]
&+(J^P(\tilde{X}^{(n-1)})-J^P)\tilde{G}_0(\tilde{\zeta}^{(n)}-\tilde{\zeta}^{(n-1)})\nabla\tilde{G}_0+J^P\tilde{G}_0(\tilde{\zeta}^{(n)}-\tilde{\zeta}^{(n-1)})\nabla\tilde{G}_0=\sum_{i=17}^{24} d_i^f.
\end{align*}
\medskip

\noindent For the estimate of these terms we can proceed exactly in the same way as we did for $I_1$, the only difference is that here instead of using lemma \ref{Jp-dif-est} we use lemma \ref{Jp-est} and instead of lemma \ref{zeta-est} we use lemma \ref{zeta-dif-est}. In the end we can conclude that the result is the same, hence for $i=17,\ldots,24$

$$\|d_i^f\|_{H^{\frac{s-1}{2}}_{(0)}L^2}\leq C(\tilde{v}_0,\tilde{G}_0) T^{\delta_i}\|\tilde{X}^{(n)}-\tilde{X}^{(n)}\|_{\mathcal{F}^{s+1,\gamma}}.$$
\medskip

\noindent For the last term $I_4$ we can easily conclude that it can be splitted in the same way as $I_2$, then the splitting will give eight terms, whose estimates are exactly the same as $I_2$, thus 

\begin{align*}
I_4&=(J^P(\tilde{X}^{(n-1)})-J^P)(\tilde{G}^{(n-1)}-\tilde{G}_0)(\tilde{\zeta}^{(n-1)}-\mathcal{I})(\nabla \tilde{G}^{(n)}-\nabla\tilde{G}^{(n-1)})\\[1mm]
&+J^P(\tilde{G}^{(n-1)}-\tilde{G}_0)(\tilde{\zeta}^{(n-1)}-\mathcal{I})(\nabla \tilde{G}^{(n)}-\nabla\tilde{G}^{(n-1)})\\[1mm]
&+(J^P(\tilde{X}^{(n-1)})-J^P)\tilde{G}_0(\tilde{\zeta}^{(n-1)}-\mathcal{I})(\nabla \tilde{G}^{(n)}-\nabla\tilde{G}^{(n-1)})+J^P\tilde{G}_0(\tilde{\zeta}^{(n-1)}-\mathcal{I})(\nabla \tilde{G}^{(n)}-\nabla\tilde{G}^{(n-1)})\\[1mm]
&+(J^P(\tilde{X}^{(n-1)})-J^P)(\tilde{G}^{(n-1)}-\tilde{G}_0)(\nabla \tilde{G}^{(n)}-\nabla\tilde{G}^{(n-1)})+J^P(\tilde{G}^{(n-1)}-\tilde{G}_0)(\nabla \tilde{G}^{(n)}-\nabla\tilde{G}^{(n-1)})\\[1mm]
&+(J^P(\tilde{X}^{(n-1)})-J^P)\tilde{G}_0(\nabla \tilde{G}^{(n)}-\nabla\tilde{G}^{(n-1)})+J^P\tilde{G}_0(\nabla \tilde{G}^{(n)}-\nabla\tilde{G}^{(n-1)})=\sum_{i=25}^{32} d_i^f.
\end{align*}
\medskip

\noindent Finally, for $i=25,\ldots,32$ we get

$$\|d_i^f\|_{H^{\frac{s-1}{2}}_{(0)}L^2}\leq C(\tilde{v}_0,\tilde{G}_0) T^{\delta_i}\|\tilde{G}^{(n)}-\tilde{G}^{(n-1)}\|_{\mathcal{F}^{s+1,\gamma}}.$$
\bigskip

\noindent \underline{ \textbf{Estimate for $\tilde{h}^{(n)}-\tilde{h}^{(n-1)}$}}.\\ \\
\noindent We can split $\tilde{h}^{(n)}$ in four terms $\tilde{h}^{(n)}=\tilde{h}^{(n)}_v+\tilde{h}^{(n)}_{v^T}+\tilde{h}^{(n)}_q+\tilde{h}^{(n)}_G$, see \eqref{split-h}. We estimate $\tilde{h}^{(n)}_G-\tilde{h}^{(n-1)}_G$. In \cite[Proposition 5.4]{CCFGG2}, one can find the other terms.

\begin{align*}
\|\tilde{h}^{(n)}_G-\tilde{h}^{(n-1)}_G\|_{\mathcal{K}^{s-\frac{1}{2}}_{(0)}}\leq \|\tilde{h}^{(n)}_G-\tilde{h}^{(n-1)}_G\|_{L^2H^{s-\frac{1}{2}}}+\|\tilde{h}^{(n)}_G-\tilde{h}^{(n-1)}_G\|_{H^{\frac{s}{2}-\frac{1}{4}}_{(0)}L^2}.
\end{align*}

\noindent At this point it is clear how to split this difference in both spaces in order to end up in the right estimate. We start in $L^2H^{s-\frac{1}{2}}$ and we have

\begin{equation}\label{h-split-L2}
\begin{split}
&\tilde{h}^{(n)}_G-\tilde{h}^{(n-1)}_G=(\tilde{G}^{(n-1)}-\tilde{G}_0)(\tilde{G}^{T(n-1)}-\tilde{G}_0^T)J^P(\tilde{X}^{(n-1)})^{-1}( \nabla_{\Lambda} \tilde{X}^{(n-1)}-\nabla_{\Lambda} \tilde{X}^{(n)}) \tilde{n}_0\\[1mm]
&\hspace{0.5cm}+\tilde{G}_0(\tilde{G}^{T(n-1)}-\tilde{G}_0^T)J^P(\tilde{X}^{(n-1)})^{-1}( \nabla_{\Lambda} \tilde{X}^{(n-1)}-\nabla_{\Lambda} \tilde{X}^{(n)}) \tilde{n}_0\\[1mm]
&\hspace{0.5cm}+(\tilde{G}^{(n-1)}-\tilde{G}_0)\tilde{G}_0^TJ^P(\tilde{X}^{(n-1)})^{-1}( \nabla_{\Lambda} \tilde{X}^{(n-1)}-\nabla_{\Lambda} \tilde{X}^{(n)}) \tilde{n}_0\\[1mm]
&\hspace{0.5cm}+\tilde{G}_0\tilde{G}_0^TJ^P(\tilde{X}^{(n-1)})^{-1}( \nabla_{\Lambda} \tilde{X}^{(n-1)}-\nabla_{\Lambda} \tilde{X}^{(n)}) \tilde{n}_0\\[1mm]
&\hspace{0.5cm}+(\tilde{G}^{(n-1)}-\tilde{G}_0)(\tilde{G}^{T(n-1)}-\tilde{G}_0^T)(J^P(\tilde{X}^{(n-1)})^{-1}-J^P(\tilde{X}^{(n)})^{-1})\nabla_{\Lambda} \tilde{X}^{(n)} \tilde{n}_0,\\[1mm]
&\hspace{0.5cm}+\tilde{G}_0(\tilde{G}^{T(n-1)}-\tilde{G}_0^T)(J^P(\tilde{X}^{(n-1)})^{-1}-J^P(\tilde{X}^{(n)})^{-1})\nabla_{\Lambda} \tilde{X}^{(n)} \tilde{n}_0\\[1mm]
&\hspace{0.5cm}+(\tilde{G}^{(n-1)}-\tilde{G}_0)\tilde{G}_0^T (J^P(\tilde{X}^{(n-1)})^{-1}-J^P(\tilde{X}^{(n)})^{-1})\nabla_{\Lambda} \tilde{X}^{(n)} \tilde{n}_0\\[1mm]
&\hspace{0.5cm}+\tilde{G}_0\tilde{G}_0^T (J^P(\tilde{X}^{(n-1)})^{-1}-J^P(\tilde{X}^{(n)})^{-1})\nabla_{\Lambda} \tilde{X}^{(n)} \tilde{n}_0\\[1mm]
&\hspace{0.5cm}+(\tilde{G}^{(n-1)}-\tilde{G}_0)(\tilde{G}^{T(n-1)}-\tilde{G}^{T(n)})J^P(\tilde{X}^{(n)})^{-1} \nabla_{\Lambda} \tilde{X}^{(n)} \tilde{n}_0\\[1mm]
&\hspace{0.5cm}+\tilde{G}_0(\tilde{G}^{T(n-1)}-\tilde{G}^{T(n)})J^P(\tilde{X}^{(n)})^{-1} \nabla_{\Lambda} \tilde{X}^{(n)} \tilde{n}_0
\end{split}
\end{equation}
\begin{align*}
&\hspace{0.5cm}+( \tilde{G}^{(n-1)}-\tilde{G}^{(n)})( \tilde{G}^{T(n)}-\tilde{G}_0^T)J^P(\tilde{X}^{(n)})^{-1} \nabla_{\Lambda} \tilde{X}^{(n)} \tilde{n}_0\\[1mm]
&\hspace{0.5cm}+( \tilde{G}^{(n-1)}-\tilde{G}^{(n)})\tilde{G}_0^TJ^P(\tilde{X}^{(n)})^{-1} \nabla_{\Lambda} \tilde{X}^{(n)} \tilde{n}_0+J^P(\tilde{X}^{(n)})^{-1}( \nabla_{\Lambda} \tilde{X}^{(n)} -\nabla_{\Lambda} \tilde{X}^{(n-1)}) \tilde{n}_0\\[1mm]
&\hspace{0.5cm}+(J^P(\tilde{X}^{(n)})^{-1} -J^P(\tilde{X}^{(n-1)})^{-1}) \nabla_{\Lambda} \tilde{X}^{(n-1)}\tilde{n}_0=\sum_{i=1}^{14} d_i^h.
\end{align*}
\medskip

\noindent We show only the most relevant estimates since the others can be deduced from them.

\begin{align*}
&\|d_{1}^h\|_{L^2H^{s-\frac{1}{2}}}\leq \|\tilde{G}^{(n-1)}-\tilde{G}_0\|_{L^2H^{s-\frac{1}{2}}}\|\tilde{G}^{T(n-1)}-\tilde{G}_0^T\|_{L^{\infty}H^{s-\frac{1}{2}}} \|J^P(\tilde{X}^{(n-1)})^{-1}\|_{L^{\infty}H^{s-\frac{1}{2}}}\\[2mm]
&\hspace{2cm}\cdot\| \nabla_{\Lambda} \tilde{X}^{(n-1)}-\nabla_{\Lambda} \tilde{X}^{(n)}\|_{L^{\infty}H^{s-\frac{1}{2}}}\\[2mm]
&\leq T^{\frac{1}{2}} \|\tilde{G}^{(n-1)}-\tilde{G}_0\|_{L^{\infty}H^{s}}\|\tilde{G}^{T(n-1)}-\tilde{G}_0^T\|_{L^{\infty}H^s} \|J^P(\tilde{X}^{(n-1)})^{-1}\|_{L^{\infty}H^{s}}\| \nabla_{\Lambda} \tilde{X}^{(n-1)}-\nabla_{\Lambda} \tilde{X}^{(n)}\|_{L^{\infty}H^s}\\[2mm]
&\leq C(\tilde{v}_0,\tilde{G}_0)T^{\frac{3}{4}}\|\tilde{X}^{(n)}- \tilde{X}^{(n-1)}\|_{L^{\infty}_{\frac{1}{4}}H^{s+1}}\leq C(\tilde{v}_0,\tilde{G}_0)T^{\frac{3}{4}}\|\tilde{X}^{(n)}- \tilde{X}^{(n-1)}\|_{\mathcal{F}^{s+1,\gamma}},\\[5mm]
&\|d_{11}^h\|_{L^2H^{s-\frac{1}{2}}}\leq\|\tilde{G}^{(n-1)}-\tilde{G}^{(n)}\|_{L^{\infty}H^{s-\frac{1}{2}}}\| \tilde{G}^{T(n)}-\tilde{G}_0^T\|_{L^2H^{s-\frac{1}{2}}}\|J^P(\tilde{X}^{(n)})^{-1} \|_{L^{\infty}H^{s-\frac{1}{2}}}\|\nabla_{\Lambda} \tilde{X}^{(n)}\|_{L^{\infty}H^{s-\frac{1}{2}}}\\[2mm]
&\leq T^{\frac{1}{2}} \|\tilde{G}^{(n-1)}-\tilde{G}^{(n)}\|_{L^{\infty}H^{s}}\| \tilde{G}^{T(n)}-\tilde{G}_0^T\|_{L^{\infty}H^{s}}\|J^P(\tilde{X}^{(n)})^{-1} \|_{L^{\infty}H^{s}}\|\nabla_{\Lambda} \tilde{X}^{(n)}\|_{L^{\infty}H^{s}}\\[2mm]
&\leq C(\tilde{v}_0,\tilde{G}_0) T^{\frac{3}{4}}\|\tilde{G}^{(n)}-\tilde{G}^{(n-1)}\|_{L^{\infty}_{\frac{1}{4}}H^{s}}\leq C(\tilde{v}_0,\tilde{G}_0) T^{\frac{3}{4}}\|\tilde{G}^{(n)}-\tilde{G}^{(n-1)}\|_{\mathcal{F}^{s,\gamma-1}},
\end{align*}

\noindent where we used trace theorem \ref{parabolic-trace}, lemma \ref{Jp-est}, lemma \ref{zeta-dif-est} and \eqref{flux-estim}, \eqref{defgrad-estim} for $d_1^h$. For $d_{11}^h$, we used the same lemmas but instead of lemma \ref{zeta-dif-est} we applied lemma \ref{zeta-est}. We also state that for $i=2,\ldots,8$, the estimates is the same as $d_1^h$ so

$$\|d_i^h\|_{L^2H^{s-\frac{1}{2}}}\leq C(\tilde{v}_0,\tilde{G}_0)T^{\frac{3}{4}}\|\tilde{X}^{(n)}- \tilde{X}^{(n-1)}\|_{\mathcal{F}^{s+1,\gamma}},$$

\noindent and for $j=9,10$ and $j=12,\ldots,18$ the estimates are similar to $d_{11}^h$, thus

$$\|d_j^h\|_{L^2H^{s-\frac{1}{2}}}\leq C(\tilde{v}_0,\tilde{G}_0)T^{\frac{3}{4}}\|\tilde{G}^{(n)}-\tilde{G}^{(n-1)}\|_{\mathcal{F}^{s,\gamma-1}}.$$
\medskip

\noindent Now we can proceed with the estimates in $H^{\frac{s}{2}-\frac{1}{4}}_{(0)}L^2$. The main difficulty with respect to the result in $L^2 H^{s-\frac{1}{2}}$ is that in this case the lemmas we have to use require all terms to be zero at $t=0$, in order to have constants independent of time. For this reason we need accurate splitting, but we want to avoid to write it because it is made of 32 terms. We remark that unlike the splitting \eqref{h-split-L2}, where it is enough to estimate $J^P(\tilde{X})^{-1}$ or $\nabla_{\Lambda}\tilde{X}$.  In $H^{\frac{s}{2}-\frac{1}{4}}_{(0)}L^2$, we analyze $J^P(\tilde{X})^{-1}-(J^P)^{-1}$ and $\nabla_{\Lambda}\tilde{X}-\mathcal{I}$. We show, as before, the estimates of the key terms, which give us the desired differences. 

\begin{align*}
&\|(\tilde{G}^{(n-1)}-\tilde{G}_0)(\tilde{G}^{T(n-1)}-\tilde{G}_0^T)(J^P(\tilde{X}^{(n-1)})^{-1}-(J^P)^{-1})( \nabla_{\Lambda} \tilde{X}^{(n-1)}-\nabla_{\Lambda} \tilde{X}^{(n)})\|_{H^{\frac{s}{2}-\frac{1}{4}}_{(0)}L^2}\\[2mm]
&\leq\|(\tilde{G}^{(n-1)}-\tilde{G}_0)(\tilde{G}^{T(n-1)}-\tilde{G}_0^T)\|_{H^{\frac{s}{2}-\frac{1}{4}}_{(0)}L^2}\|(J^P(\tilde{X}^{(n-1)})^{-1}-(J^P)^{-1})( \nabla_{\Lambda} \tilde{X}^{(n-1)}-\nabla_{\Lambda} \tilde{X}^{(n)})\|_{H^{\frac{s}{2}-\frac{1}{4}}_{(0)}H^{\frac{1}{2}+\eta}}
\end{align*}
\begin{align*}
&\leq \|\tilde{G}^{(n-1)}-\tilde{G}_0\|_{H^{\frac{s}{2}-\frac{1}{4}}_{(0)}H^{\frac{1}{2}-\eta'}}\|\tilde{G}^{T(n-1)}-\tilde{G}_0^T\|_{H^{\frac{s}{2}-\frac{1}{4}}_{(0)}H^{\frac{1}{2}+\eta'}}\|J^P(\tilde{X}^{(n-1)})^{-1}-(J^P)^{-1}\|_{H^{\frac{s}{2}-\frac{1}{4}}_{(0)}H^{\frac{1}{2}+\eta}}\\[2mm]
&\hspace{1cm}\cdot\|\nabla_{\Lambda} \tilde{X}^{(n-1)}-\nabla_{\Lambda} \tilde{X}^{(n)}\|_{H^{\frac{s}{2}-\frac{1}{4}}_{(0)}H^{\frac{1}{2}+\eta}}\\[2mm]
&\leq \|\tilde{G}^{(n-1)}-\tilde{G}_0\|_{H^{\frac{s}{2}-\frac{1}{4}}_{(0)}H^{1-\eta'}}\|\tilde{G}^{T(n-1)}-\tilde{G}_0^T\|_{H^{\frac{s}{2}-\frac{1}{4}}_{(0)}H^{1+\eta'}}\|J^P(\tilde{X}^{(n-1)})^{-1}-(J^P)^{-1}\|_{H^{\frac{s}{2}-\frac{1}{4}}_{(0)}H^{1+\eta}}\\[2mm]
&\hspace{1cm}\cdot\|\nabla_{\Lambda} \tilde{X}^{(n-1)}-\nabla_{\Lambda} \tilde{X}^{(n)}\|_{H^{\frac{s}{2}-\frac{1}{4}}_{(0)}H^{1+\eta}}\\[2mm]
&\leq C(\tilde{v}_0,\tilde{G}_0) \|\tilde{X}^{(n)}- \tilde{X}^{(n-1)}\|_{H^{\frac{s}{2}-\frac{1}{4}}_{(0)}H^{2+\eta}}\leq C(\tilde{v}_0,\tilde{G}_0) \left\|\partial_t\int_0^t\tilde{X}^{(n)}- \tilde{X}^{(n-1)}\right\|_{H^{\frac{s}{2}-\frac{1}{4}+\beta_1-\beta_1}_{(0)}H^{2+\eta}}\\[2mm]
&\leq C(\tilde{v}_0,\tilde{G}_0) T^{\beta_1}\|\tilde{X}^{(n)}- \tilde{X}^{(n-1)}\|_{H^{\frac{s}{2}-\frac{1}{4}+\beta_1}_{(0)}H^{2+\eta}}\leq C(\tilde{v}_0,\tilde{G}_0) T^{\beta_1}\|\tilde{X}^{(n)}- \tilde{X}^{(n-1)}\|_{\mathcal{F}^{s+1,\gamma}}.
\end{align*}
\medskip

\noindent This result is obtained by applying lemma \ref{lem3}, lemma \ref{lem4}, with $\frac{1}{q}=\frac{1}{2}-\eta'$ and lemma \ref{lem5}. Moreover, we use trace theorem \ref{parabolic-trace}, lemma \ref{Jp-est} and estimate \eqref{flux-estim} and \eqref{defgrad-estim}. Finally lemma \ref{zeta-dif-est}, lemma \ref{lem2} and lemma \ref{lem6} give the final result. Furthermore, we have also the differences of the deformation gradient which are given, for instance, by the following term

\begin{align*}
&\|( \tilde{G}^{(n-1)}-\tilde{G}^{(n)})( \tilde{G}^{T(n)}-\tilde{G}_0^T)(J^P(\tilde{X}^{(n)})^{-1} -(J^P)^{-1})(\nabla_{\Lambda} \tilde{X}^{(n)}-\mathcal{I})\|_{H^{\frac{s}{2}-\frac{1}{4}}_{(0)}L^2}\\[2mm]\
&\leq \|\tilde{G}^{(n-1)}-\tilde{G}^{(n)}\|_{H^{\frac{s}{2}-\frac{1}{4}}_{(0)}H^{\frac{1}{2}-\eta'}}\|\tilde{G}^{T(n-1)}-\tilde{G}_0^T\|_{H^{\frac{s}{2}-\frac{1}{4}}_{(0)}H^{\frac{1}{2}+\eta'}}\|J^P(\tilde{X}^{(n)})^{-1}-(J^P)^{-1}\|_{H^{\frac{s}{2}-\frac{1}{4}}_{(0)}H^{\frac{1}{2}+\eta}}\\[2mm]
&\hspace{1cm}\cdot\|\nabla_{\Lambda} \tilde{X}^{(n)}-\mathcal{I}\|_{H^{\frac{s}{2}-\frac{1}{4}}_{(0)}H^{\frac{1}{2}+\eta}}\\[2mm]
&\leq \|\tilde{G}^{(n-1)}-\tilde{G}^{(n)}\|_{H^{\frac{s}{2}-\frac{1}{4}}_{(0)}H^{1-\eta'}}\|\tilde{G}^{T(n-1)}-\tilde{G}_0^T\|_{H^{\frac{s}{2}-\frac{1}{4}}_{(0)}H^{1+\eta'}}\|J^P(\tilde{X}^{(n)})^{-1}-(J^P)^{-1}\|_{H^{\frac{s}{2}-\frac{1}{4}}_{(0)}H^{1+\eta}}\\[2mm]
&\hspace{1cm}\cdot\|\nabla_{\Lambda} \tilde{X}^{(n)}-\mathcal{I}\|_{H^{\frac{s}{2}-\frac{1}{4}}_{(0)}H^{1+\eta}}\\[2mm]
&\leq C(\tilde{v}_0,\tilde{G}_0) \|\tilde{G}^{(n-1)}-\tilde{G}^{(n)}\|_{H^{\frac{s}{2}-\frac{1}{4}}_{(0)}H^{1-\eta'}}\leq C(\tilde{v}_0,\tilde{G}_0) \left\|\partial_t\int_0^t\tilde{G}^{(n-1)}-\tilde{G}^{(n)}\right\|_{H^{\frac{s}{2}-\frac{1}{4}+\beta_2-\beta_2}_{(0)}H^{1-\eta'}}\\[2mm]
&\leq C(\tilde{v}_0,\tilde{G}_0) T^{\beta_2}\|\tilde{G}^{(n-1)}-\tilde{G}^{(n)}\|_{\mathcal{F}^{s,\gamma-1}}.
\end{align*}
\medskip

\noindent We remark that we use the same lemmas as above but here we use also lemma \ref{zeta-est}.

\noindent  Finally, as in \cite[Proposition 5.4]{CCFGG2} 

\begin{align*}
\|\bar{g}^{(n)}-\bar{g}^{(n-1)}\|_{\mathcal{\bar{K}}^{s}_{(0)}}&\leq C(\tilde{v}_0,\tilde{G}_0) T^{\theta} \left(\|\tilde{X}^{(n)}-\tilde{X}^{(n-1)}\|_{\mathcal{F}^{s+1,\gamma}}+\|\tilde{w}^{(n)}-\tilde{w}^{(n-1)}\|_{\mathcal{K}^{s+1}_{(0)}}\right).
\end{align*}

\noindent In the end, if we put all the estimates together and if we choose $\varrho=\min\{\frac{3}{4},\delta_i,\beta_j,\theta\},$ for $i=1,\ldots, 32$ and $j=1,\ldots,32$ then we prove  Lemma \ref{fixed point}  and consequently also Theorem \ref{localex}.
\end{proof}

\section{Stability Estimates}\label{sec5}
\noindent In order to prove stability we choose a one-parameter  family of initial conditions $\tilde{\Omega}_{\varepsilon}(0)$ and $\tilde{v}_{\varepsilon}(0)$, such that
$$\tilde{\Omega}_{\varepsilon}(0)=\tilde{\Omega}_0+\varepsilon b,$$
\noindent where $b$ is a constant vector, $|b|=1$, such that $P^{-1}(\tilde{\Omega}_{\varepsilon}(0))$ is not a self-intersecting domain, as in this figures below.

\begin{figure}[htbp]
\centering
\includegraphics[scale=0.48] {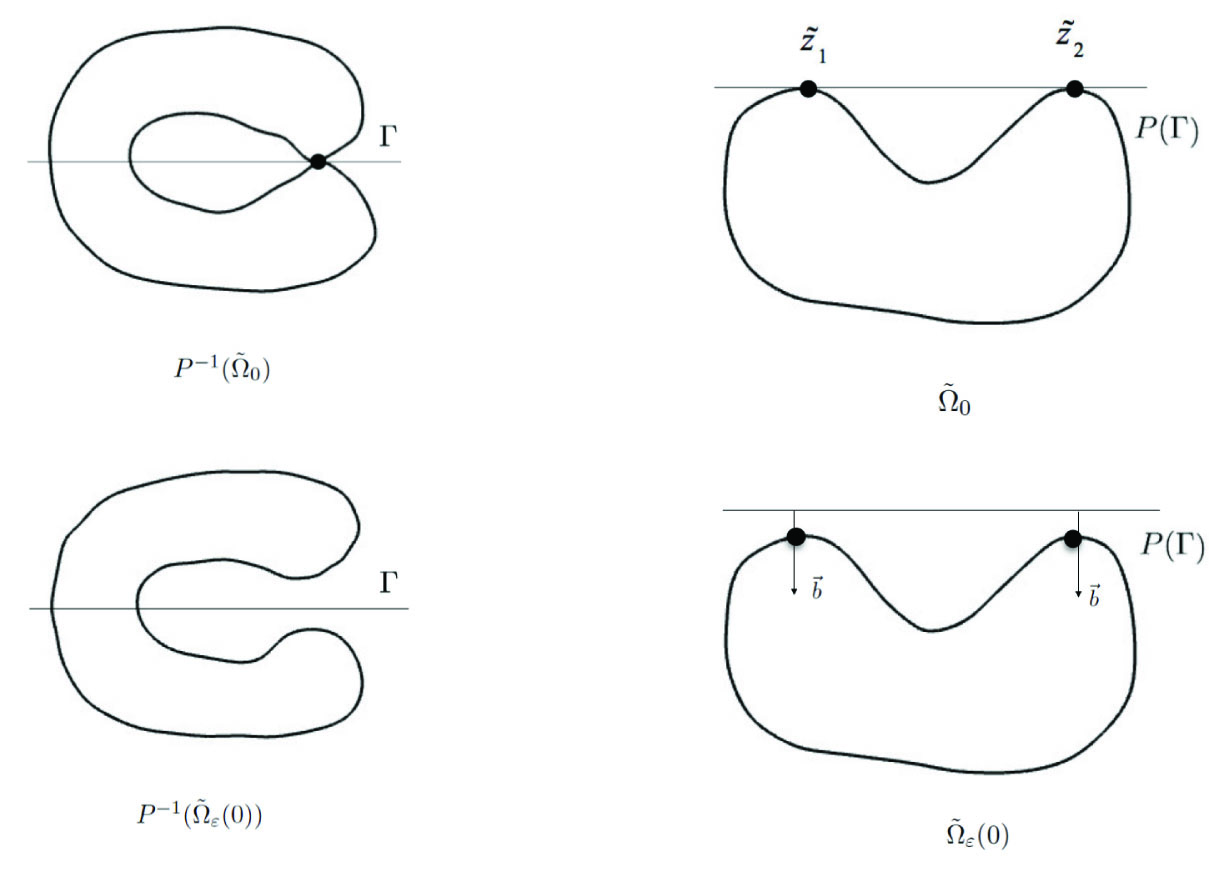}
\caption{}
\end{figure}

\noindent We compare the solution $(\tilde{w}, \tilde{q}, \tilde{X}, \tilde{G})$ and the solution $(\tilde{w}_{\varepsilon}, \tilde{q}_{\varepsilon}, \tilde{X}_{\varepsilon}, \tilde{G}_{\varepsilon})$. Let us consider the following system

\begin{equation}\label{w-stabality}
\left\{\begin{array}{lll}
\displaystyle\partial_t(\tilde{w}-\tilde{w}_{\varepsilon})-Q^2\Delta(\tilde{w}-\tilde{w}_{\varepsilon})+(J^P)^T\nabla(\tilde{q}_w-\tilde{q}_{w,\varepsilon})=\tilde{F}_{\varepsilon}\\ [3mm]
\displaystyle \trace(\nabla(\tilde{w}-\tilde{w}_{\varepsilon})J^P)=\tilde{K}_{\varepsilon}\\ [3mm]
[-(\tilde{q}_w-\tilde{q}_{w,\varepsilon})\mathcal{I}+\nabla(\tilde{w}-\tilde{w}_{\varepsilon})J^P+(\nabla(\tilde{w}-\tilde{w}_{\varepsilon})J^P)^T](J^P)^{-1}\tilde{n}_0=\tilde{H}_{\varepsilon}\\[3mm]
\displaystyle \tilde{w}_0-\tilde{w}_{\varepsilon,0}=0,
\end{array}\right.
\end{equation}
\medskip

\noindent where
\begin{align*}
&\tilde{F}_{\varepsilon}=\tilde{f}-\tilde{f}_{\varepsilon}+\tilde{f}_{\phi}^L-\tilde{f}_{\phi,\varepsilon}^L+(Q^2-Q^2_{\varepsilon})\Delta \tilde{w}_{\varepsilon}-((J^P)^T-(J^P)^T_{\varepsilon})\nabla \tilde{q}_{w,\varepsilon},\\
&\tilde{K}_{\varepsilon}=\tilde{g}-\tilde{g}_{\varepsilon}+\tilde{g}_{\phi}^L-\tilde{g}_{\phi,\varepsilon}^L-\trace(\nabla\tilde{w}_{\varepsilon}(J^P-J^P_{\varepsilon})),\\
&\tilde{H_{\varepsilon}}=\tilde{h}-\tilde{h}_{\varepsilon}+\tilde{h}_{\phi}^L-\tilde{h}_{\phi,\varepsilon}^L+\tilde{q}_{w,\varepsilon}((J^P)^{-1}-(J^P_{\varepsilon})^{-1})\tilde{n}_0-(\nabla\tilde{w}_{\varepsilon}J^P)(J^P)^{-1}\tilde{n}_0\\
&\hspace{0.5cm} -(\nabla\tilde{w}_{\varepsilon}J^P)^T (J^P)^{-1}\tilde{n}_0+(\nabla\tilde{w}_{\varepsilon}J^P_{\varepsilon})(J^P_{\varepsilon})^{-1}\tilde{n}_0+(\nabla\tilde{w}_{\varepsilon}J^P_{\varepsilon})^T(J^P_{\varepsilon})^{-1}\tilde{n}_0,
\end{align*}

\medskip
\noindent because of the definition of the velocity field, we have
\begin{align*}
&\tilde{f}_{\phi}^L-\tilde{f}_{\phi,\varepsilon}^L=-\frac{d}{dt}(\phi-\phi_{\varepsilon})+Q^2\Delta\phi-Q^2_{\varepsilon}\Delta\phi_{\varepsilon}-(J^P)^T\nabla \tilde{q}_{\phi}+(J^P_{\varepsilon})^T\nabla \tilde{q}_{\phi,\varepsilon},\\
&\tilde{g}_{\phi}^L-\tilde{g}_{\phi,\varepsilon}^L=-\trace(\nabla\phi J^P)+\trace(\nabla\phi_{\varepsilon}J^P_{\varepsilon}),\\
&\tilde{h}_{\phi}^L-\tilde{h}_{\phi,\varepsilon}^L=\tilde{q}_{\phi}(J^P)^{-1}n_0-\tilde{q}_{\phi,\varepsilon}(J^P_{\varepsilon})^{-1}n_0-[(\nabla\phi J^P)+(\nabla\phi J^P)^T](J^P)^{-1}n_0\\
&\hspace{1.6cm}+[(\nabla\phi_{\varepsilon}J^P_{\varepsilon})
+(\nabla\phi_{\varepsilon} J^P_{\varepsilon})^T](J^P_{\varepsilon})^{-1}n_0,
\end{align*}
\medskip

\noindent where $J^P_{\varepsilon}(\tilde{\alpha})=J^P(\tilde{\alpha}+\varepsilon b)$ and $Q_{\varepsilon}^2(\tilde{\alpha})=Q^2(\tilde{\alpha}+\varepsilon b)$. Furthermore the function $\phi_{\varepsilon}$ is contructed exactly as we did in the previous Section, in order to satisfy all the hypothesis of Lemma \ref{invL}.

$$\phi_{\varepsilon}=\tilde{v}_0+t(Q^2_{\varepsilon}\Delta\tilde{v}_0-(J^P)^T_{\varepsilon}\nabla q_{\phi,\varepsilon} + J^P_{\varepsilon}\tilde{G}_{0}\nabla \tilde{G}_{0})=\tilde{v}_0+t\hat{\phi}_{\varepsilon}.$$
\medskip

\noindent For the flux we have that $\tilde{X}_{\varepsilon}$ satisfies

\begin{equation}
\left\{\begin{array}{lll}
\displaystyle \frac{d}{dt} \tilde{X}_{\varepsilon}(t,\tilde{\alpha})=J^P(\tilde{X}_{\varepsilon}(t,\tilde{\alpha}))\tilde{ v}_{\varepsilon}(t,\tilde{\alpha})\\[3mm]
\displaystyle\tilde{ X}_{\varepsilon}(0,\tilde{\alpha})=\tilde{\alpha}+\varepsilon b,
\end{array}\right.
\end{equation}
\medskip

\noindent and so 

$$\tilde{X}-\tilde{X}_{\varepsilon}=-\varepsilon b+\int_0^t \left(J^P( \tilde{X}) \tilde{v}-J^P(\tilde{X}_{\varepsilon}) \tilde{v}_{\varepsilon}\right)\,d\tau.$$

\medskip

\noindent Similar the perturbed deformation gradient $\tilde{G}_{\varepsilon}$ satisfies
\medskip

\begin{equation}
\left\{ \begin{array}{lll}
\displaystyle \partial_t \tilde{G}_{\varepsilon}=J^P( \tilde{X}_{\varepsilon} )\tilde{\zeta}_{\varepsilon}\nabla\tilde{ v}_{\varepsilon} \tilde{G}_{\varepsilon}\\[3mm]
\displaystyle \tilde{G}_{\varepsilon}(0,\tilde{\alpha})=\tilde{G}_0,
\end{array}\right.
\end{equation}
\medskip

\noindent hence
\begin{equation}\label{G-Geps}
\tilde{G}-\tilde{G}_{\varepsilon}=\int_0^t (J^P( \tilde{X})\tilde{\zeta}\nabla \tilde{v} \tilde{G}-J^P( \tilde{X}_{\varepsilon} )\tilde{\zeta}_{\varepsilon}\nabla \tilde{v}_{\varepsilon} \tilde{G}_{\varepsilon})\,d\tau.
\end{equation}

\bigskip
 
\noindent The main stability result we will prove is the following

\begin{theorem}\label{main-stab}
Let $2<s<\frac{5}{2}$ and a suitable $\delta>0$. If $0<T<\frac{1}{(3C)^{\frac{1}{\delta}}}$ then
\begin{equation*}
\|\tilde{X}-\tilde{X}_{\varepsilon}\|_{L^{\infty}H^{s+1}}\leq 3C\varepsilon.
\end{equation*}
\end{theorem}
\medskip

\noindent The proof is an outcome of the following Lemma. 

\begin{lemma}\label{stab-bound}
For $2<s<\frac{5}{2}$, a suitable $\delta>0$ and suppose that
\medskip

\begin{enumerate}

\item $\|J^P-J^P_{\varepsilon}\|_{H^r}\leq C\varepsilon,\hspace{0.3cm} \|Q^2- Q_{\varepsilon}^2\|_{H^r}\leq C\varepsilon \hspace{0.3cm}\textrm{for all}\hspace{0.2 cm} r,$ since $Q^2$ and $J^P$ are $C^{\infty}$ functions. \\[1mm]

\item $\|\phi-\phi_{\varepsilon}\|_{L^{\infty}H^{s+1}}\leq C\varepsilon,\hspace{0.3cm} \|\phi-\phi_{\varepsilon}\|_{H^1_{(0)}H^{\gamma}}\leq C\varepsilon$, for smooth $\tilde{v}_0, \tilde{G}_0$.\\[1mm]

\item $\|\tilde{q}_{\phi}-\tilde{q}_{\phi,\varepsilon}\|_{H^{r+1}}\leq C\varepsilon\hspace{0.3cm}\forall r\geq 0.$\\[1mm]

\item $\|\tilde{X}-\tilde{X}_{\varepsilon}+\varepsilon b-t(J^P-J^P_{\varepsilon})\tilde{v}_0\|_{\mathcal{F}^{s+1,\gamma}}\leq C\varepsilon + C T^{\delta}\left(\|\tilde{X}-\tilde{X}_{\varepsilon}+\varepsilon b-t(J^P-J^P_{\varepsilon})\tilde{v}_0\|_{\mathcal{F}^{s+1,\gamma}}\right.\\[2mm]
\left.  \hspace{11.5cm}+ \|\tilde{w}-\tilde{w}_{\varepsilon}\|_{\mathcal{K}^{s+1}_{(0)}}\right).$
\end{enumerate}
\medskip

\noindent Then

\begin{align*}
&\|\tilde{w}-\tilde{w}_{\varepsilon}\|_{\mathcal{K}^{s+1}_{(0)}} +\|\tilde{q}_w-\tilde{q}_{w,\varepsilon}\|_{\mathcal{K}^{s}_{pr}(0)}+\|\tilde{X}-\tilde{X}_{\varepsilon}+\varepsilon b-t(J^P-J^P_{\varepsilon})\tilde{v}_0\|_{\mathcal{F}^{s+1,\gamma}}\\[2mm]
&+\|\tilde{G}-\tilde{G}_{\varepsilon}-t(J^P-J^P_{\varepsilon})\nabla\tilde{v}_0\tilde{G}_0\|_{\mathcal{F}^{s,\gamma-1}}\leq 3C\varepsilon+3CT^{\delta}\left(\|\tilde{w}-\tilde{w}_{\varepsilon}\|_{\mathcal{K}^{s+1}_{(0)}}+\|\tilde{q}_w-\tilde{q}_{w,\varepsilon}\|_{\mathcal{K}^{s}_{pr(0)}}\right.\\[2mm]
&+\left.\|\tilde{X}-\tilde{X}_{\varepsilon}+\varepsilon b-t(J^P-J^P_{\varepsilon})\tilde{v}_0\|_{\mathcal{F}^{s+1,\gamma}}+\|\tilde{G}-\tilde{G}_{\varepsilon}-t(J^P-J^P_{\varepsilon})\nabla\tilde{v}_0\tilde{G}_0\|_{\mathcal{F}^{s,\gamma-1}} \right),
\end{align*}

where the constant $C$ depends only on the initial data. 
\end{lemma}
\medskip

\noindent The  points $(1)-(4)$ of this Lemma are results obtained in \cite[Lemma $6.1$]{CCFGG2}, with small modifications because of the new definition of $\phi, \phi_{\varepsilon}$, which now depend on $\tilde{v}_0$ but also on $\tilde{G}_0$. Then Theorem \ref{main-stab} follows easily and so 

\begin{equation}\label{Xcloseness}
\|\tilde{X}-\tilde{X}_{\varepsilon}\|_{L^{\infty}H^{s+1}}\leq C\varepsilon\hspace{0.5cm}\Rightarrow\hspace{0.5cm}\textrm{dist}(\partial\tilde{\Omega}(t),\partial\tilde{\Omega}_{\varepsilon}(t))\lesssim \varepsilon.
\end{equation}
\medskip

\subsection{Proof of Lemma \ref{stab-bound}}
\noindent In order to prove the Lemma above, we need to estimate both $(\tilde{w}-\tilde{w}_{\varepsilon},\tilde{q}_w-\tilde{q}_{w,\varepsilon})$ and $\tilde{G}-\tilde{G}_{\varepsilon}-t(J^P-J^P_{\varepsilon})\nabla\tilde{v}_0\tilde{G}_0$. For the deformation gradient we prove the following result.
 
\begin{proposition}\label{G-eps}
For a suitable $\beta>0$  and $2<s<\frac{5}{2}$, we have
\begin{align*}
&\|\tilde{G}-\tilde{G}_{\varepsilon}-t(J^P-J^P_{\varepsilon})\nabla\tilde{v}_0\tilde{G}_0\|_{\mathcal{F}^{s,\gamma-1}}\leq C(\tilde{v}_0,\tilde{G}_0)\varepsilon+C(\tilde{v}_0,\tilde{G}_0) T^{\beta} \left(\|\tilde{w}-\tilde{w}_{\varepsilon}\|_{\mathcal{K}^{s+1}_{(0)}}\right.\\[2mm]
&\hspace{1cm}\left.+\|\tilde{G}-\tilde{G}_{\varepsilon}-t(J^P-J^P_{\varepsilon})\nabla\tilde{v}_0\tilde{G}_0\|_{\mathcal{F}^{s,\gamma-1}}+\|\tilde{X}-\tilde{X}_{\varepsilon}+\varepsilon b-t(J^P-J^P_{\varepsilon})\tilde{v}_0\|_{\mathcal{F}^{s+1,\gamma}}\right).
\end{align*}
\end{proposition}

\begin{proof}
\noindent The difference $\tilde{G}-\tilde{G}_{\varepsilon}-t(J^P-J^P_{\varepsilon})\nabla\tilde{v}_0\tilde{G}_0$ has to be estimated in $L^{\infty}_{\frac{1}{4}}H^s$ and $H^{2}_{(0)}H^{\gamma-1}$. As we have already shown in the iterative bounds, the second estimate will be much more complicated. For this reason, we separate the two results. We first deal with $L^{\infty}_{\frac{1}{4}}H^s$ and by using the definition of the norm we have

\begin{align*}
\|\tilde{G}-\tilde{G}_{\varepsilon}-t(J^P-J^P_{\varepsilon})\nabla\tilde{v}_0\tilde{G}_0\|_{L^{\infty}_{\frac{1}{4}}H^s}\leq T^{\frac{1}{4}}\|J^P( \tilde{X})\tilde{\zeta}\nabla \tilde{v} \tilde{G}-J^P( \tilde{X}_{\varepsilon} )\tilde{\zeta}_{\varepsilon}\nabla \tilde{v}_{\varepsilon} \tilde{G}_{\varepsilon}-(J^P-J^P_{\varepsilon})\nabla\tilde{v}_0\tilde{G}_0\|_{L^2H^{s}}.
\end{align*}
\medskip

\noindent Now we take into account the definition of the velocity and we split the difference as follows.

\begin{align*}
&I_1=(J^P( \tilde{X})-J^P( \tilde{X}_{\varepsilon})-J^P+J^P_{\varepsilon})\tilde{\zeta}\nabla \tilde{w} (\tilde{G}-\tilde{G}_0)+(J^P( \tilde{X})-J^P( \tilde{X}_{\varepsilon})-J^P+J^P_{\varepsilon})\tilde{\zeta}\nabla \tilde{w}\tilde{G}_0\\[1mm]
&\hspace{0.5cm}+(J^P-J^P_{\varepsilon})\tilde{\zeta}\nabla \tilde{w} (\tilde{G}-\tilde{G}_0)+(J^P-J^P_{\varepsilon})\tilde{\zeta}\nabla \tilde{w}\tilde{G}_0=I_{1,1}+I_{1,2}+I_{1,3}+I_{1,4},\\[2mm]
&I_2=J^P( \tilde{X}_{\varepsilon})(\tilde{\zeta}-\tilde{\zeta}_{\varepsilon})\nabla\tilde{w}(\tilde{G}-\tilde{G}_0)+J^P( \tilde{X}_{\varepsilon})(\tilde{\zeta}-\tilde{\zeta}_{\varepsilon})\nabla\tilde{w}\tilde{G}_0=I_{2,1}+I_{2,2},\\[2mm]
&I_3=J^P( \tilde{X}_{\varepsilon})\tilde{\zeta}_{\varepsilon}(\nabla \tilde{w}
-\nabla \tilde{w}_{\varepsilon})(\tilde{G}-\tilde{G}_0)+J^P( \tilde{X}_{\varepsilon})\tilde{\zeta}_{\varepsilon}(\nabla \tilde{w}
-\nabla \tilde{w}_{\varepsilon})\tilde{G}_0=I_{3,1}+I_{3,2},
\end{align*}
\begin{align*}
&I_4=J^P(\tilde{X}_{\varepsilon})\tilde{\zeta}_{\varepsilon}\nabla \tilde{w}_{\varepsilon}(\tilde{G}-\tilde{G}_{\varepsilon}),\\[2mm]
&I_5=(J^P( \tilde{X})-J^P( \tilde{X}_{\varepsilon})-J^P+J^P_{\varepsilon})\tilde{\zeta}\nabla \tilde{v}_0 (\tilde{G}-\tilde{G}_0)+(J^P( \tilde{X})-J^P( \tilde{X}_{\varepsilon})-J^P+J^P_{\varepsilon})\tilde{\zeta}\nabla \tilde{v}_0\tilde{G}_0\\[1mm]
&\hspace{0.5cm}+(J^P-J^P_{\varepsilon})\tilde{\zeta}\nabla \tilde{v}_0 (\tilde{G}-\tilde{G}_0)+(J^P-J^P_{\varepsilon})(\tilde{\zeta}-\mathcal{I})\nabla \tilde{v}_0\tilde{G}_0=I_{5,1}+I_{5,2}+I_{5,3}+I_{5,4},\\[2mm]
&I_6=J^P( \tilde{X}_{\varepsilon})(\tilde{\zeta}-\tilde{\zeta}_{\varepsilon})\nabla\tilde{v}_0(\tilde{G}-\tilde{G}_0)+J^P( \tilde{X}_{\varepsilon})(\tilde{\zeta}-\tilde{\zeta}_{\varepsilon})\nabla\tilde{v}_0\tilde{G}_0=I_{6,1}+I_{6,2},\\[2mm]
&I_7=J^P(\tilde{X}_{\varepsilon})\tilde{\zeta}_{\varepsilon}\nabla \tilde{v}_{0}(\tilde{G}-\tilde{G}_{\varepsilon}),\\[2mm]
&I_8=(J^P( \tilde{X})-J^P( \tilde{X}_{\varepsilon})-J^P+J^P_{\varepsilon})\tilde{\zeta}t\nabla \hat{\phi} (\tilde{G}-\tilde{G}_0)+(J^P( \tilde{X})-J^P( \tilde{X}_{\varepsilon})-J^P+J^P_{\varepsilon})\tilde{\zeta}t\nabla \hat{\phi} \tilde{G}_0\\[1mm]
&\hspace{0.5cm}+(J^P-J^P_{\varepsilon})\tilde{\zeta}t\nabla \hat{\phi}(\tilde{G}-\tilde{G}_0)+(J^P-J^P_{\varepsilon})\tilde{\zeta}t\nabla \hat{\phi} \tilde{G}_0=I_{8,1}+I_{8,2}+I_{8,3}+I_{8,4},\\[2mm]
&I_9=J^P( \tilde{X}_{\varepsilon})(\tilde{\zeta}-\tilde{\zeta}_{\varepsilon})t\nabla \hat{\phi} (\tilde{G}-\tilde{G}_0)+J^P( \tilde{X}_{\varepsilon})(\tilde{\zeta}-\tilde{\zeta}_{\varepsilon})t\nabla \hat{\phi} \tilde{G}_0=I_{9,1}+I_{9,2},\\[2mm]
&I_{10}=J^P( \tilde{X}_{\varepsilon})\tilde{\zeta}_{\varepsilon}t(\nabla \hat{\phi}
-\nabla \hat{\phi}_{\varepsilon})(\tilde{G}-\tilde{G}_0)+J^P( \tilde{X}_{\varepsilon})\tilde{\zeta}_{\varepsilon}t(\nabla \hat{\phi}
-\nabla \hat{\phi}_{\varepsilon})\tilde{G}_0=I_{10,1}+I_{10,2},\\[2mm]
&I_{11}=J^P(\tilde{X}_{\varepsilon})\tilde{\zeta}_{\varepsilon}t\nabla \hat{\phi}_{\varepsilon}(\tilde{G}-\tilde{G}_{\varepsilon}).
\end{align*}
\medskip

\noindent We show the estimate of the terms which provide the required differences, so we deal with $I_{1,1}, I_{3,1}$ and $I_{11}$. Moreover, we remark that in $I_j$, for $j=5,6,7$  we have the initial velocity and their estimates are similar to the estimates of $I_j$, for $j=8,9,11$ since the term $\hat{\phi}$ depends only on the initial data too. For $I_{1,1}$ we use lemma \ref{zeta-est}, \eqref{flux-estim},\eqref{defgrad-estim} and the most difficult part is to use lemma \ref{Jp-dif-est}.

\begin{align*}
&T^{\frac{1}{4}}\|I_{1,1}\|_{L^2H^s}\leq T^{\frac{1}{4}}\|J^P( \tilde{X})-J^P( \tilde{X}_{\varepsilon})-J^P+J^P_{\varepsilon}\|_{L^{\infty}H^s}\|\tilde{\zeta}\|_{L^{\infty}H^s}\|\nabla \tilde{w}\|_{L^2H^s}\|\tilde{G}-\tilde{G}_0\|_{L^{\infty}H^s}\\[2mm]
&\leq T^{\frac{1}{4}}\left(\|J^P(\tilde{X}+\varepsilon b)-J^P(\tilde{X}_{\varepsilon})\|_{L^{\infty}H^s}+\|J^P(\tilde{X})-J^P-(J^P(\tilde{X}+\varepsilon b)-J^P_{\varepsilon})\|_{L^{\infty}H^s}\right)\\[2mm]
&\hspace{1cm}\cdot\|\tilde{X}-\tilde{\alpha}\|_{L^{\infty}H^{s+1}}\|\tilde{w}\|_{L^2H^{s+1}}\|\tilde{G}-\tilde{G}_0\|_{L^{\infty}H^s}\\[2mm]
&\leq C(\tilde{v}_0,\tilde{G}_0)T^{\frac{1}{4}}\|\tilde{X}-\tilde{X}_{\varepsilon}+\varepsilon b\|_{L^{\infty}H^{s+1}}+C(\tilde{v}_0,\tilde{G}_0)\varepsilon\\[2mm]
&\leq C(\tilde{v}_0,\tilde{G}_0)\varepsilon+ C(\tilde{v}_0,\tilde{G}_0)T^{\frac{1}{4}}\left(\|\tilde{X}-\tilde{X}_{\varepsilon}+\varepsilon b-t(J^P-J^P_{\varepsilon})\nabla\tilde{v}_0\tilde{G}_0\|_{L^{\infty}H^{s+1}}+\|t(J^P-J^P_{\varepsilon})\nabla\tilde{v}_0\tilde{G}_0\|_{L^{\infty}H^{s+1}}\right)\\[2mm]
&\leq C(\tilde{v}_0,\tilde{G}_0)\varepsilon+ C(\tilde{v}_0,\tilde{G}_0)T^{\frac{1}{4}}\|\tilde{X}-\tilde{X}_{\varepsilon}+\varepsilon b-t(J^P-J^P_{\varepsilon})\nabla\tilde{v}_0\tilde{G}_0\|_{\mathcal{F}^{s+1,\gamma}}.
\end{align*}
\medskip

\noindent For $I_{3,1}$ and $I_{11}$, we use lemma \ref{Jp-est}, lemma \ref{zeta-est} and \eqref{flux-estim}, \eqref{defgrad-estim}.

\begin{align*}
&T^{\frac{1}{4}}\|I_{3,1}\|_{L^2H^{s}}\leq T^{\frac{1}{4}}\|J^P( \tilde{X}_{\varepsilon})\|_{L^{\infty}H^s}\|\tilde{\zeta}_{\varepsilon}\|_{L^{\infty}H^s}\|\nabla \tilde{w}
-\nabla \tilde{w}_{\varepsilon}\|_{L^2H^s}\|\tilde{G}-\tilde{G}_0\|_{L^{\infty}H^s}\\[2mm]
&\hspace{2.2cm}\leq C(\tilde{v}_0,\tilde{G}_0)T^{\frac{1}{4}}\|\tilde{w}-\tilde{w}_{\varepsilon}\|_{\mathcal{K}^{s+1}_{(0)}}.\\[4mm]
\end{align*}
\begin{align*}
&T^{\frac{1}{4}} \|I_{11}\|_{L^2H^s}\leq T^{\frac{1}{4}} \|J^P(\tilde{X}_{\varepsilon})\|_{L^{\infty}H^s}\|\tilde{\zeta}_{\varepsilon}\|_{L^{\infty}H^s} \|t\nabla \hat{\phi}_{\varepsilon}\|_{L^2H^s}\left(\|\tilde{G}-\tilde{G}_{\varepsilon}-t(J^P-J^P_{\varepsilon})\nabla\tilde{v}_0\tilde{G}_0\|_{L^{\infty}H^s}\right.\\
&\left.\hspace{10.5cm}+\|t(J^P-J^P_{\varepsilon})\nabla\tilde{v}_0\tilde{G}_0\|_{L^{\infty}H^s}\right)\\[2mm]
&\hspace{2.2cm}\leq C(\tilde{v}_0,\tilde{G}_0) T^2 \|\tilde{G}-\tilde{G}_{\varepsilon}-t(J^P-J^P_{\varepsilon})\nabla\tilde{v}_0\tilde{G}_0\|_{\mathcal{F}^{s,\gamma-1}}.
\end{align*}
\medskip

\noindent  To get the estimate in $H^2_{(0)}H^{\gamma-1}$, we observe the following result by using lemma \ref{lem2} 

\begin{align*}
&\|\tilde{G}-\tilde{G}_{\varepsilon}-t(J^P-J^P_{\varepsilon})\nabla\tilde{v}_0\tilde{G}_0\|_{H^2_{(0)}H^{\gamma-1}}\\
&\hspace{1cm}\leq \left\|\int_0^t J^P( \tilde{X})\tilde{\zeta}\nabla \tilde{v} \tilde{G}-J^P( \tilde{X}_{\varepsilon} )\tilde{\zeta}_{\varepsilon}\nabla \tilde{v}_{\varepsilon} \tilde{G}_{\varepsilon}-(J^P-J^P_{\varepsilon})\nabla\tilde{v}_0\tilde{G}_0\right\|_{H^2_{(0)}H^{\gamma-1}}\\[2mm]
&\hspace{1cm}\leq\left\|J^P( \tilde{X})\tilde{\zeta}\nabla \tilde{v} \tilde{G}-J^P( \tilde{X}_{\varepsilon} )\tilde{\zeta}_{\varepsilon}\nabla \tilde{v}_{\varepsilon} \tilde{G}_{\varepsilon}-(J^P-J^P_{\varepsilon})\nabla\tilde{v}_0\tilde{G}_0\right\|_{H^1_{(0)}H^{\gamma-1}}
\end{align*}
\medskip

\noindent We rewrite the difference in a convenient way, that allow us to apply all the necessary lemmas in order to obtain the final estimate. Thus

\begin{align*}
&I_1=(J^P( \tilde{X})-J^P( \tilde{X}_{\varepsilon})-J^P+J^P_{\varepsilon})(\tilde{\zeta}-\mathcal{I})\nabla \tilde{w} (\tilde{G}-\tilde{G}_0)\\[1mm]
&\hspace{0.5cm}+(J^P( \tilde{X})-J^P( \tilde{X}_{\varepsilon})-J^P+J^P_{\varepsilon})(\tilde{\zeta}-\mathcal{I})\nabla \tilde{w}\tilde{G}_0\\[1mm]
&\hspace{0.5cm}+(J^P-J^P_{\varepsilon})(\tilde{\zeta}-\mathcal{I})\nabla \tilde{w} (\tilde{G}-\tilde{G}_0)+(J^P-J^P_{\varepsilon})(\tilde{\zeta}-\mathcal{I})\nabla \tilde{w}\tilde{G}_0\\[1mm]
&\hspace{0.5cm}+(J^P( \tilde{X})-J^P( \tilde{X}_{\varepsilon})-J^P+J^P_{\varepsilon})\nabla \tilde{w} (\tilde{G}-\tilde{G}_0)+(J^P( \tilde{X})-J^P( \tilde{X}_{\varepsilon})-J^P+J^P_{\varepsilon})\nabla \tilde{w}\tilde{G}_0\\[1mm]
&\hspace{0.5cm}+(J^P-J^P_{\varepsilon})\nabla \tilde{w} (\tilde{G}-\tilde{G}_0)+(J^P-J^P_{\varepsilon})\nabla \tilde{w}\tilde{G}_0\\[1mm]
&\hspace{0.5cm}=I_{1,1}+I_{1,2}+I_{1,3}+I_{1,4}+I_{1,5}+I_{1,6}+I_{1,7}+I_{1,8},\\[3mm]
&I_2=(J^P( \tilde{X}_{\varepsilon})-J^P_{\varepsilon})
(\tilde{\zeta}-\tilde{\zeta}_{\varepsilon})\nabla\tilde{w}(\tilde{G}-\tilde{G}_0)+(J^P( \tilde{X}_{\varepsilon})-J^P_{\varepsilon})(\tilde{\zeta}-\tilde{\zeta}_{\varepsilon})\nabla\tilde{w}\tilde{G}_0\\[1mm]
&\hspace{0.5cm}+J^P_{\varepsilon}(\tilde{\zeta}-\tilde{\zeta}_{\varepsilon})\nabla\tilde{w}(\tilde{G}-\tilde{G}_0)+J^P_{\varepsilon}(\tilde{\zeta}-\tilde{\zeta}_{\varepsilon})\nabla\tilde{w}\tilde{G}_0=I_{2,1}+I_{2,2}+I_{2,3}+I_{2,4},\\[3mm]
&I_3=(J^P( \tilde{X}_{\varepsilon})-J^P_{\varepsilon})(\tilde{\zeta}_{\varepsilon}-\mathcal{I})(\nabla \tilde{w}
-\nabla \tilde{w}_{\varepsilon})(\tilde{G}-\tilde{G}_0)+(J^P( \tilde{X}_{\varepsilon})-J^P_{\varepsilon})(\tilde{\zeta}_{\varepsilon}-\mathcal{I})(\nabla \tilde{w}
-\nabla \tilde{w}_{\varepsilon})\tilde{G}_0\\[1mm]
&\hspace{0.5cm}+J^P_{\varepsilon}(\tilde{\zeta}_{\varepsilon}-\mathcal{I})(\nabla \tilde{w}
-\nabla \tilde{w}_{\varepsilon})(\tilde{G}-\tilde{G}_0)+J^P_{\varepsilon}(\tilde{\zeta}_{\varepsilon}-\mathcal{I})(\nabla \tilde{w}
-\nabla \tilde{w}_{\varepsilon})(\tilde{G}-\tilde{G}_0)\\[1mm]
&\hspace{0.5cm}+(J^P( \tilde{X}_{\varepsilon})-J^P_{\varepsilon})(\nabla \tilde{w}-\nabla \tilde{w}_{\varepsilon})(\tilde{G}-\tilde{G}_0)+(J^P( \tilde{X}_{\varepsilon})-J^P_{\varepsilon})(\nabla \tilde{w}-\nabla \tilde{w}_{\varepsilon})\tilde{G}_0\\[1mm]
&\hspace{0.5cm}+J^P_{\varepsilon}(\nabla \tilde{w}-\nabla \tilde{w}_{\varepsilon})(\tilde{G}-\tilde{G}_0)+J^P_{\varepsilon}(\nabla \tilde{w}-\nabla \tilde{w}_{\varepsilon})\tilde{G}_0\\[1mm]
&\hspace{0.5cm}=I_{3,1}+I_{3,2}+I_{3,3}+I_{3,4}+I_{3,5}+I_{3,6}+I_{3,7}+I_{3,8},\\[3mm]
&I_4=(J^P(\tilde{X}_{\varepsilon})-J^P_{\varepsilon})(\tilde{\zeta}_{\varepsilon}-\mathcal{I})\nabla \tilde{w}_{\varepsilon}(\tilde{G}-\tilde{G}_{\varepsilon})+J^P_{\varepsilon}(\tilde{\zeta}_{\varepsilon}-\mathcal{I})\nabla \tilde{w}_{\varepsilon}(\tilde{G}-\tilde{G}_{\varepsilon})\\[1mm]
&\hspace{0.5cm}+(J^P(\tilde{X}_{\varepsilon})-J^P_{\varepsilon})\nabla \tilde{w}_{\varepsilon}(\tilde{G}-\tilde{G}_{\varepsilon})+J^P_{\varepsilon}\nabla \tilde{w}_{\varepsilon}(\tilde{G}-\tilde{G}_{\varepsilon})=I_{4,1}+I_{4,2}+I_{4,3}+I_{4,4}.
\end{align*}
\medskip

\noindent  In these terms we wrote $\nabla\tilde{w}$ and $\nabla\tilde{w}_{\varepsilon}$, instead of $\nabla\tilde{v}$ and $\nabla\tilde{v}_{\varepsilon}$. But we have other 4 terms  that come by substituting $t\nabla\hat{\phi}$ and $t\nabla\hat{\phi}_{\varepsilon}$ and other 3 by substituting $\nabla\tilde{v}_0$. We focus on the estimates of the terms above, because all the terms with $\nabla\tilde{v}_0$ or  $t\nabla\hat{\phi}$ can be easily estimate in the same way as those we will show and by using the facts that $\|t\|_{H^{1}_{(0)}}\leq T^{\frac{1}{2}}$ and $\hat{\phi}, \hat{\phi}_{\varepsilon}$ depend only on $\tilde{v}_0, \tilde{G}_0$. In particular we study $I_{1,1}, I_{3,1}$ and $I_{4,1}$. For $I_{1,1}$ we use lemma \ref{lem3} with $\gamma>1$, lemma \ref{lem5}, lemma \ref{zeta-est}, lemma \ref{lem1} and \eqref{flux-estim}, \eqref{defgrad-estim}, for $t$ small enough. Finally, by using lemma \ref{Jp-dif-est}, lemma \ref{lem2}, with $0<\eta_1<\delta_1$ and lemma \ref{lem6}. 

\begin{align*}
&\|I_{1,1}\|_{H^{1}_{(0)}H^{\gamma-1}}\leq \|J^P( \tilde{X})-J^P( \tilde{X}_{\varepsilon})-J^P+J^P_{\varepsilon}\|_{H^{1}_{(0)}H^{\gamma}}\|(\tilde{\zeta}-\mathcal{I})\nabla \tilde{w} (\tilde{G}-\tilde{G}_0)\|_{H^{1}_{(0)}H^{\gamma-1}}\\[2mm]
&\leq \|J^P( \tilde{X})-J^P( \tilde{X}_{\varepsilon})-J^P+J^P_{\varepsilon}\|_{H^{1}_{(0)}H^{\gamma}}\|\tilde{\zeta}-\mathcal{I}\|_{H^{1}_{(0)}H^{\gamma}}\|\nabla \tilde{w}\|_{H^{1}_{(0)}H^{\gamma}}\|\tilde{G}-\tilde{G}_0\|_{H^{1}_{(0)}H^{\gamma-1}}\\[2mm]
&\leq C(\tilde{v}_0,\tilde{G}_0)\left(\|J^P(\tilde{X}+\varepsilon b)-J^P(\tilde{X}_{\varepsilon})\|_{H^{1}_{(0)}H^{\gamma}}+\|J^P(\tilde{X})-J^P-(J^P(\tilde{X}+\varepsilon b)-J^P_{\varepsilon})\|_{H^{1}_{(0)}H^{\gamma}}\right)\\[2mm]
&\leq C(\tilde{v}_0,\tilde{G}_0)\varepsilon+ C(\tilde{v}_0,\tilde{G}_0)\left(\|\tilde{X}-\tilde{X}_{\varepsilon}+\varepsilon b-t(J^P-J^P_{\varepsilon})\tilde{v}_0\|_{H^{1}_{(0)}H^{\gamma}}+\|t(J^P-J^P_{\varepsilon})\tilde{v}_0\|_{H^{1}_{(0)}H^{\gamma}}\right)\\[2mm]
&\leq C(\tilde{v}_0,\tilde{G}_0)\varepsilon+ C(\tilde{v}_0,\tilde{G}_0)\left\|\int_0^t \partial_t(\tilde{X}-\tilde{X}_{\varepsilon}+\varepsilon b-t(J^P-J^P_{\varepsilon})\tilde{v}_0)\right\|_{H^{1+\eta_1-\delta_1}_{(0)}H^{\gamma}}\\[2mm]
&\leq C(\tilde{v}_0,\tilde{G}_0)\varepsilon+ C(\tilde{v}_0,\tilde{G}_0)T^{\delta_1}\|\tilde{X}-\tilde{X}_{\varepsilon}+\varepsilon b-t(J^P-J^P_{\varepsilon})\tilde{v}_0\|_{\mathcal{F}^{s+1,\gamma}}.
\end{align*}
\medskip

\noindent For $I_{3,1}$ and $I_{4,1}$, lemma \ref{lem3}, with $\gamma>1$, lemma \ref{lem5} and lemma \ref{Jp-est}, lemma \ref{zeta-est} and \eqref{flux-estim}, \eqref{defgrad-estim}. For the difference of the velocity we use lemma \ref{lem1}, with $\delta_3<\eta_3<\frac{s-1-\gamma}{2}$ and lemma \ref{lem2}. For the difference of the deformation gradient we use only lemma \ref{lem2}, with $0<\delta_4<\eta_4$.

\begin{align*}
&\|I_{3,1}\|_{H^{1}_{(0)}H^{\gamma-1}}\leq \|J^P( \tilde{X}_{\varepsilon})-J^P_{\varepsilon}\|_{H^{1}_{(0)}H^{\gamma}}\|(\tilde{\zeta}_{\varepsilon}-\mathcal{I})(\nabla \tilde{w}
-\nabla \tilde{w}_{\varepsilon})(\tilde{G}-\tilde{G}_0)\|_{H^{1}_{(0)}H^{\gamma-1}}\\[2mm]
&\leq \|\tilde{X}_{\varepsilon}-\tilde{\alpha}-\varepsilon b\|_{H^{1}_{(0)}H^{\gamma}}\|\tilde{\zeta}_{\varepsilon}-\mathcal{I}\|_{H^{1}_{(0)}H^{\gamma-1}}\|\tilde{w}-\tilde{w}_{\varepsilon}\|_{H^{1}_{(0)}H^{\gamma}}\|\tilde{G}-\tilde{G}_0\|_{H^{1}_{(0)}H^{\gamma-1}}\\[2mm]
&\leq C(\tilde{v}_0,\tilde{G}_0)\left\|\int_0^t\partial_t(\tilde{w}-\tilde{w}_{\varepsilon})\right\|_{H^{1+\eta_3-\delta_3}_{(0)}H^{\gamma}}\leq C(\tilde{v}_0,\tilde{G}_0) T^{\delta_3} \|\tilde{w}-\tilde{w}_{\varepsilon}\|_{\mathcal{K}^{s+1}_{(0)}},\\[7mm]
&\|I_{4,1}\|_{H^{1}_{(0)}H^{\gamma-1}}\leq\|J^P(\tilde{X}_{\varepsilon})-J^P_{\varepsilon}\|_{H^1_{(0)}H^{\gamma}}\|(\tilde{\zeta}_{\varepsilon}-\mathcal{I})\nabla \tilde{w}_{\varepsilon}(\tilde{G}-\tilde{G}_{\varepsilon})\|_{H^1_{(0)}H^{\gamma-1}}\\[2mm]
&\leq \|\tilde{X}_{\varepsilon}-\tilde{\alpha}-\varepsilon b\|_{H^{1}_{(0)}H^{\gamma}}\|\tilde{\zeta}_{\varepsilon}-\mathcal{I}\|_{H^{1}_{(0)}H^{\gamma-1}}\|\tilde{w}_{\varepsilon}\|_{H^{1}_{(0)}H^{\gamma}}\left(\|\tilde{G}-\tilde{G}_{\varepsilon}-t(J^P-J^P_{\varepsilon})\nabla\tilde{v}_0\tilde{G}_0\|_{H^{1}_{(0)}H^{\gamma-1}}\right.\\[2mm]
&\hspace{10.5cm}\left.+\|t(J^P-J^P_{\varepsilon})\nabla\tilde{v}_0\tilde{G}_0\|_{H^{1}_{(0)}H^{\gamma-1}}\right)\\[2mm]
&\leq C(\tilde{v}_0,\tilde{G}_0) \left\|\int_0^t\partial_t(\tilde{G}-\tilde{G}_{\varepsilon}-t(J^P-J^P_{\varepsilon})\nabla\tilde{v}_0\tilde{G}_0)\right\|_{H^{1+\eta_4-\delta_4}_{(0)}H^{\gamma}}+C(\tilde{v}_0,\tilde{G}_0)\varepsilon\\[2mm]
&\leq C(\tilde{v}_0,\tilde{G}_0)\varepsilon+C(\tilde{v}_0,\tilde{G}_0) T^{\delta_4} \|\tilde{G}-\tilde{G}_{\varepsilon}-t(J^P-J^P_{\varepsilon})\nabla\tilde{v}_0\tilde{G}_0\|_{\mathcal{F}^{s,\gamma-1}}.
\end{align*}
\medskip

\noindent Thus, by taking $\beta=\min\{\frac{1}{4}, \delta_i\}$, where the $\delta_i>0$ are the exponents coming from the estimates of the terms in $H^2_{(0)}H^{\gamma-1}$, the proposition holds.
\end{proof}
\medskip

\noindent For the velocity and the pressure, the following result holds.

\begin{proposition}\label{stability}
For a suitable $\varrho>0$ and $2<s<\frac{5}{2}$, we have

\begin{align*}
& \|\tilde{w}-\tilde{w}_{\varepsilon}\|_{\mathcal{K}^{s+1}_{(0)}} +\|\tilde{q}_w-\tilde{q}_{w,\varepsilon}\|_{\mathcal{K}^{s}_{pr(0)}}\leq C(\tilde{v}_0,\tilde{G}_0)\varepsilon+C(\tilde{v}_0,\tilde{G}_0)T^{\varrho}\left( \|\tilde{w}-\tilde{w}_{\varepsilon}\|_{\mathcal{K}^{s+1}_{(0)}} +\|\tilde{q}_w-\tilde{q}_{w,\varepsilon}\|_{\mathcal{K}^{s}_{pr(0)}}\right.\\[2mm]
&\hspace{3cm}\left.+\|\tilde{X}-\tilde{X}_{\varepsilon}+\varepsilon b-t(J^P-J^P_{\varepsilon})\tilde{v}_0\|_{\mathcal{F}^{s+1,\gamma}}+\|\tilde{G}-\tilde{G}_{\varepsilon}-t(J^P-J^P_{\varepsilon})\nabla\tilde{v}_0\tilde{G}_0\|_{\mathcal{F}^{s,\gamma-1}}\right).
\end{align*}

\end{proposition}
\medskip

\begin{proof}
\noindent As for the proof of Proposition \ref{estim-(v,q)}, we use the result of Lemma \ref{invL}. Therefore we have

$$(\tilde{w}-\tilde{w}_{\varepsilon},\tilde{q}_w-\tilde{q}_{\varepsilon})=L^{-1}(\tilde{F}_{\varepsilon},\tilde{K}_{\varepsilon},\tilde{H}_{\varepsilon}),$$

\noindent and so we get

\begin{align*}
&\|\tilde{w}-\tilde{w}_{\varepsilon}\|_{\mathcal{K}^{s+1}_{(0)}} +\|\tilde{q}_w-\tilde{q}_{w,\varepsilon}\|_{\mathcal{K}^{s}_{pr(0)}}
\leq C \left(\|\tilde{F}_{\varepsilon}\|_{\mathcal{K}^{s-1}_{(0)}}+\|\tilde{K}_{\varepsilon}\|_{\mathcal{\bar{K}}^{s}_{(0)}}+\|\tilde{H}_{\varepsilon}\|_{\mathcal{K}^{s-\frac{1}{2}}_{(0)}}\right).
\end{align*}

\noindent For this reason it is sufficient to prove 

\begin{align*}
&\|\tilde{F}_{\varepsilon}\|_{\mathcal{K}^{s-1}_{(0)}}+\|\tilde{K}_{\varepsilon}\|_{\mathcal{\bar{K}}^{s}_{(0)}}+\|\tilde{H}_{\varepsilon}\|_{\mathcal{K}^{s-\frac{1}{2}}_{(0)}} \leq C\varepsilon+C T^{\delta}\left( \|\tilde{w}-\tilde{w}_{\varepsilon}\|_{\mathcal{K}^{s+1}_{(0)}} +\|\tilde{q}_w-\tilde{q}_{w,\varepsilon}\|_{\mathcal{K}^{s}_{pr(0)}}\right.\\
&\left.+\|\tilde{X}-\tilde{X}_{\varepsilon}+\varepsilon b-t(J^P-J^P_{\varepsilon})\tilde{v}_0\|_{\mathcal{F}^{s+1,\gamma}}+\|\tilde{G}-\tilde{G}_{\varepsilon}-t(J^P-J^P_{\varepsilon})\nabla\tilde{v}_0\tilde{G}_0\|_{\mathcal{F}^{s,\gamma-1}}\right).
\end{align*}
\medskip

\noindent In the following we are not doing all the splittings we did in the proof of proposition \ref{estim-(v,q)}, since they are similar, but we show the most important one.
\medskip 

\noindent\underline{\textbf{Estimate for $\tilde{F}_{\varepsilon}$}}.
\medskip

\noindent We have to analyze the following term 

$$\tilde{F}_{\varepsilon}=\tilde{f}-\tilde{f}_{\varepsilon}+\tilde{f}_{\phi}^L-\tilde{f}_{\phi,\varepsilon}^L+(Q^2-Q^2_{\varepsilon})\Delta \tilde{w}_{\varepsilon}-((J^P)^T-(J^P_{\varepsilon})^T)\nabla \tilde{q}_{w,\varepsilon}.$$
\bigskip

\noindent We start with the estimate of $\tilde{F}_{\varepsilon}$ in $L^2H^{s-1}$. Using (1)-(4) of lemma \ref{stab-bound} we get

\begin{align*}
&\|(Q^2-Q^2_{\varepsilon})\Delta \tilde{w}_{\varepsilon}\|_{L^2H^{s-1}}\leq \|Q^2-Q^2_{\varepsilon}\|_{L^{\infty}H^{s-1}}\|\tilde{w}_{\varepsilon}\|_{L^2H^{s+1}}\leq C\varepsilon,\\ \\
&\left\|((J^P)^T-(J^P_{\varepsilon})^T)\nabla \tilde{q}_{w,\varepsilon}\right\|_{L^2H^{s-1}}\leq \|(J^P)^T-(J^P_{\varepsilon})^T\|_{L^{\infty}H^{s-1}}\| \tilde{q}_{w,\varepsilon}\|_{L^2H^s}\leq C\varepsilon
\end{align*}
\medskip

\noindent For the $H^{\frac{s-1}{2}}_{(0)}L^2$-norm, we use lemma \ref{lem3} and lemma \ref{lem1}, then

\begin{align*}
&\|(Q^2-Q^2_{\varepsilon})\Delta \tilde{w}_{\varepsilon}\|_{H^{\frac{s-1}{2}}_{(0)}L^2}\leq\|Q^2-Q^2_{\varepsilon}\|_{H^{1+\eta}}\|\Delta \tilde{w}_{\varepsilon}\|_{H^{\frac{s-1}{2}}_{(0)}L^2}\leq C\varepsilon \|\tilde{w}_{\varepsilon}\|_{H^{\frac{s-1}{2}}_{(0)}H^2}\leq C(\tilde{v}_0)\varepsilon,\\ \\
&\|((J^P)^T-(J^P_{\varepsilon})^T)\nabla \tilde{q}_{w,\varepsilon}\|_{H^{\frac{s-1}{2}_{(0)}}L^2}\leq \|(J^P)^T-(J^P_{\varepsilon})^T\|_{H^{1+\eta}}\| \tilde{q}_{w,\varepsilon}\|_{H^{\frac{s-1}{2}}_{(0)}H^1}\leq C(\tilde{v}_0)\varepsilon.
\end{align*}
\medskip

\noindent We have to pay attention in the estimate of  $\tilde{f}_{\phi}^L-\tilde{f}_{\phi,\varepsilon}^L=t(Q^2-Q^2_{\varepsilon})\Delta\hat{\phi}+t Q^2_{\varepsilon}(\Delta\hat{\phi}-\Delta\hat{\phi}_{\varepsilon})+(J^P_{\varepsilon}-J^P)\tilde{G}_0\nabla\tilde{G}_0=I_1+I_2+I_3.$ We observe that in $I_1$ and $ I_2$, we have all the terms depending on the initial data, but the fact that we have $t$ in front of these terms allow us to get

\begin{align*}
\|I_1\|_{\mathcal{K}^{s-1}_{(0)}}+\|I_2\|_{\mathcal{K}^{s-1}_{(0)}}\leq C(\tilde{v}_0,\tilde{G}_0)\varepsilon.
\end{align*}
\medskip

\noindent We cannot state the same thing for $I_3$ because this term depends only on the initial data and we are not able to estimate it in $H^{\frac{s-1}{2}}_{(0)}([0,T])$. Thus we have to consider this term with $\tilde{f}-\tilde{f}_{\varepsilon}$.  Concerning this difference, we rewrite it as $\tilde{f}-\tilde{f}_{\varepsilon}=\tilde{f}_w-\tilde{f}_{w,\varepsilon}+\tilde{f}_{\phi}-\tilde{f}_{\phi,\varepsilon}+\tilde{f}_q-\tilde{f}_{q,\varepsilon}+\tilde{f}_G-\tilde{f}_{G,\varepsilon}$. We have to estimate only the term $\tilde{f}_G-\tilde{f}_{G,\varepsilon}$, because the others have already been estimated in \cite[Lemma 6.2]{CCFGG2}. We take account of 

\begin{align*}
&\|\tilde{f}_G-\tilde{f}_{G,\varepsilon}+I_3\|_{\mathcal{K}^{s-1}_{(0)}}=\|J^P( \tilde{X} )\tilde{G} \tilde{\zeta} \nabla \tilde{G}- J^P( \tilde{X}_{\varepsilon} )\tilde{G}_{\varepsilon}\tilde{\zeta}_{\varepsilon}\nabla\tilde{ G}_{\varepsilon}+(J^P_{\varepsilon}-J^P)\tilde{G}_0\nabla\tilde{G}_0\|_{\mathcal{K}^{s-1}_{(0)}}\\[1mm]
&\leq \|\tilde{f}_G-\tilde{f}_{G,\varepsilon}+(J^P_{\varepsilon}-J^P)\tilde{G}_0\nabla\tilde{G}_0\|_{L^2H^{s-1}}
+\|\tilde{f}_G-\tilde{f}_{G,\varepsilon}+(J^P_{\varepsilon}-J^P)\tilde{G}_0\nabla\tilde{G}_0\|_{H^{\frac{s-1}{2}}_{(0)}L^2}=I_1+I_2.
\end{align*}
\medskip

\noindent At first we deal with the $L^2H^{s-1}$-norm and we use the following splitting

\begin{align*}
I_1&=(J^P(\tilde{X})-J^P(\tilde{X}_{\varepsilon})-J^P+J^P_{\varepsilon})(\tilde{G}-\tilde{G}_0)\tilde{\zeta}(\nabla \tilde{G}-\nabla\tilde{G}_0)+(J^P-J^P_{\varepsilon})(\tilde{G}-\tilde{G}_0)\tilde{\zeta}(\nabla \tilde{G}-\nabla\tilde{G}_0)\\[2mm]
&+(J^P(\tilde{X})-J^P(\tilde{X}_{\varepsilon})-J^P+J^P_{\varepsilon})\tilde{G}_0\tilde{\zeta}(\nabla \tilde{G}-\nabla\tilde{G}_0)+(J^P-J^P_{\varepsilon})(\tilde{G}-\tilde{G}_0)\tilde{\zeta}\nabla\tilde{G}_0\\[2mm]
&+(J^P(\tilde{X})-J^P(\tilde{X}_{\varepsilon})-J^P+J^P_{\varepsilon})(\tilde{G}-\tilde{G}_0)\tilde{\zeta}\nabla\tilde{G}_0+(J^P-J^P_{\varepsilon})(\tilde{G}-\tilde{G}_0)\tilde{\zeta}\nabla\tilde{G}_0\\[2mm]
&+(J^P(\tilde{X})-J^P(\tilde{X}_{\varepsilon})-J^P+J^P_{\varepsilon})\tilde{G}_0\tilde{\zeta}\nabla\tilde{G}_0+(J^P-J^P_{\varepsilon})\tilde{G}_0(\tilde{\zeta}-\mathcal{I})\nabla\tilde{G}_0\\[2mm]
&+J^P(\tilde{X}_{\varepsilon})(\tilde{G}-\tilde{G}_{\varepsilon})\tilde{\zeta}(\nabla \tilde{G}-\nabla\tilde{G}_0)+J^P(\tilde{X}_{\varepsilon})(\tilde{G}-\tilde{G}_{\varepsilon})\tilde{\zeta}\nabla\tilde{G}_0\\[2mm]
&+J^P(\tilde{X}_{\varepsilon})(\tilde{G}_{\varepsilon}-\tilde{G}_0)(\tilde{\zeta}-\tilde{\zeta}_{\varepsilon})(\nabla \tilde{G} -\nabla\tilde{G}_0)+J^P(\tilde{X}_{\varepsilon})(\tilde{G}_{\varepsilon}-\tilde{G}_0)(\tilde{\zeta}-\tilde{\zeta}_{\varepsilon})\nabla\tilde{G}_0\\[2mm]
&+J^P(\tilde{X}_{\varepsilon})\tilde{G}_0(\tilde{\zeta}-\tilde{\zeta}_{\varepsilon})(\nabla \tilde{G} -\nabla\tilde{G}_0)+J^P(\tilde{X}_{\varepsilon})\tilde{G}_0(\tilde{\zeta}-\tilde{\zeta}_{\varepsilon})\nabla\tilde{G}_0\\[2mm]
&+J^P( \tilde{X}_{\varepsilon})(\tilde{G}_{\varepsilon}-\tilde{G}_0)\tilde{\zeta}_{\varepsilon}(\nabla\tilde{ G} -\nabla \tilde{G}_{\varepsilon})+J^P( \tilde{X}_{\varepsilon})\tilde{G}_0\tilde{\zeta}_{\varepsilon}(\nabla\tilde{ G} -\nabla \tilde{G}_{\varepsilon})=\sum_{i=1}^{16} d_i^{F_{\varepsilon}}.
\end{align*}
\medskip

\noindent We analyze $d_{1}^{F_{\varepsilon}},  d_{9}^{F_{\varepsilon}}$, the terms which require multiple estimates. For both of them we use lemma \ref{zeta-est} and \eqref{flux-estim}, \eqref{defgrad-estim}. Finally lemma \ref{Jp-est} and lemma \ref{Jp-dif-est}.

\begin{align*}
&\| d_{1}^{F_{\varepsilon}}\|_{L^2H^{s-1}}\leq \|J^P(\tilde{X})-J^P(\tilde{X}_{\varepsilon})-J^P+J^P_{\varepsilon}\|_{L^{\infty}H^{s-1}}\|\tilde{G}-\tilde{G}_0\|_{L^2H^{s-1}}\|\tilde{\zeta}\|_{L^{\infty}H^{s-1}}\|\nabla \tilde{G}-\nabla\tilde{G}_0\|_{L^{\infty}H^{s-1}}\\[2mm]
&\leq C(\tilde{v}_0,\tilde{G}_0)T^{\frac{1}{2}}\left(\|J^P(\tilde{X}+\varepsilon b)-J^P(\tilde{X}_{\varepsilon})\|_{L^{\infty}H^{s-1}}+\|J^P(\tilde{X})-J^P-(J^P(\tilde{X}+\varepsilon b)-J^P_{\varepsilon})
\|_{L^{\infty}H^{s-1}}\right)\\[2mm]
&\leq C(\tilde{v}_0,\tilde{G}_0)T^{\frac{1}{2}}\left(\|\tilde{X}-\tilde{X}_{\varepsilon}+\varepsilon b\|_{L^{\infty}H^{s+1}}+\varepsilon\right)\\[2mm]
&\leq C(\tilde{v}_0,\tilde{G}_0)\varepsilon+ C(\tilde{v}_0,\tilde{G}_0)T^{\frac{1}{2}}\left(\|\tilde{X}-\tilde{X}_{\varepsilon}+\varepsilon b-t(J^P-J^P_{\varepsilon})\tilde{v}_0\|_{L^{\infty}H^{s+1}}+\|t(J^P-J^P_{\varepsilon})\tilde{v}_0\|_{L^{\infty}H^{s+1}}\right)
\end{align*}
\begin{align*}
&\leq C(\tilde{v}_0,\tilde{G}_0)\varepsilon+ C(\tilde{v}_0,\tilde{G}_0)T^{\frac{3}{4}}\|\tilde{X}-\tilde{X}_{\varepsilon}+\varepsilon b-t(J^P-J^P_{\varepsilon})\tilde{v}_0\|_{\mathcal{F}^{s+1,\gamma}}.\\[7mm]
&\|d_{9}^{F_{\varepsilon}}\|_{L^2H^{s-1}}\leq \|J^P(\tilde{X}_{\varepsilon})\|_{L^{\infty}H^{s-1}}\|\tilde{G}-\tilde{G}_{\varepsilon}\|_{L^{\infty}H^{s-1}}\|\tilde{\zeta}\|_{L^{\infty}H^{s-1}}\|\nabla \tilde{G}-\nabla\tilde{G}_0\|_{L^2H^{s-1}}\\[2mm]
&\leq C(\tilde{v}_0,\tilde{G}_0)T^{\frac{1}{2}}\left(\|\tilde{G}-\tilde{G}_{\varepsilon}-t(J^P-J^P_{\varepsilon})\nabla\tilde{v}_0\tilde{G}_0\|_{L^{\infty}H^s}+\|t(J^P-J^P_{\varepsilon})\nabla\tilde{v}_0\tilde{G}_0\|_{L^{\infty}H^s}\right)\\[2mm]
&\leq C(\tilde{v}_0,\tilde{G}_0) T^{\frac{3}{4}} \|\tilde{G}-\tilde{G}_{\varepsilon}-t(J^P-J^P_{\varepsilon})\nabla\tilde{v}_0\tilde{G}_0\|_{\mathcal{F}^{s,\gamma-1}}.
\end{align*}
\medskip

\noindent The study of $H^{\frac{s-1}{2}}_{(0)}L^2$-norm is longer, we show the main idea. We rewrite the term  $I_2$ as follows

\begin{align*}
&I_{2,1}=(J^P(\tilde{X})-J^P(\tilde{X}_{\varepsilon}))\tilde{G}\tilde{\zeta} \nabla \tilde{G}-(J^P-J^P_{\varepsilon})\tilde{G}_0\nabla\tilde{G}_0,\\[2mm]
&I_{2,2}= J^P(\tilde{X}_{\varepsilon})(\tilde{G}-\tilde{G}_{\varepsilon})\tilde{\zeta}\nabla \tilde{G},\\[2mm]
&I_{2,3}=J^P(\tilde{X}_{\varepsilon})\tilde{G}_{\varepsilon}(\tilde{\zeta}-\tilde{\zeta}_{\varepsilon})\nabla \tilde{G},\\[2mm]
&I_{2,4}=J^P( \tilde{X}_{\varepsilon})\tilde{G}_{\varepsilon}\tilde{\zeta}_{\varepsilon}(\nabla\tilde{ G} -\nabla \tilde{G}_{\varepsilon}).
\end{align*}
\medskip

\noindent Because of the requirement for all functions to be zero at $t=0$, we show how to split $I_{2,1}$

\begin{align*}
I_{2,1}&=(J^P(\tilde{X})-J^P(\tilde{X}_{\varepsilon})-J^P+J^P_{\varepsilon})(\tilde{G}-\tilde{G}_0)(\tilde{\zeta}-\mathcal{I})( \nabla \tilde{G}-\nabla\tilde{G}_0)\\[1mm]
&+(J^P-J^P_{\varepsilon})(\tilde{G}-\tilde{G}_0)(\tilde{\zeta}-\mathcal{I})( \nabla \tilde{G}-\nabla\tilde{G}_0)\\[1mm]
&+(J^P(\tilde{X})-J^P(\tilde{X}_{\varepsilon})-J^P+J^P_{\varepsilon})\tilde{G}_0(\tilde{\zeta}-\mathcal{I})( \nabla \tilde{G}-\nabla\tilde{G}_0)\\[1mm]
&+(J^P-J^P_{\varepsilon})\tilde{G}_0(\tilde{\zeta}-\mathcal{I})( \nabla \tilde{G}-\nabla\tilde{G}_0)+(J^P(\tilde{X})-J^P(\tilde{X}_{\varepsilon})-J^P+J^P_{\varepsilon})(\tilde{G}-\tilde{G}_0)(\tilde{\zeta}-\mathcal{I})\nabla\tilde{G}_0\\[1mm]
&+(J^P-J^P_{\varepsilon})(\tilde{G}-\tilde{G}_0)(\tilde{\zeta}-\mathcal{I})\nabla\tilde{G}_0+(J^P(\tilde{X})-J^P(\tilde{X}_{\varepsilon})-J^P+J^P_{\varepsilon})\tilde{G}_0(\tilde{\zeta}-\mathcal{I})\nabla\tilde{G}_0\\[1mm]
&+(J^P-J^P_{\varepsilon})\tilde{G}_0(\tilde{\zeta}-\mathcal{I})\nabla\tilde{G}_0+(J^P(\tilde{X})-J^P(\tilde{X}_{\varepsilon})-J^P+J^P_{\varepsilon})(\tilde{G}-\tilde{G}_0)( \nabla \tilde{G}-\nabla\tilde{G}_0)\\[1mm]
&+(J^P-J^P_{\varepsilon})(\tilde{G}-\tilde{G}_0)(\nabla \tilde{G}-\nabla\tilde{G}_0)+(J^P(\tilde{X})-J^P(\tilde{X}_{\varepsilon})-J^P+J^P_{\varepsilon})\tilde{G}_0( \nabla \tilde{G}-\nabla\tilde{G}_0)\\[1mm]
&+(J^P-J^P_{\varepsilon})\tilde{G}_0(\nabla \tilde{G}-\nabla\tilde{G}_0)+(J^P(\tilde{X})-J^P(\tilde{X}_{\varepsilon})-J^P+J^P_{\varepsilon})(\tilde{G}-\tilde{G}_0)\nabla\tilde{G}_0\\[1mm]
&+(J^P-J^P_{\varepsilon})(\tilde{G}-\tilde{G}_0) \nabla\tilde{G}_0+(J^P(\tilde{X})-J^P(\tilde{X}_{\varepsilon})-J^P+J^P_{\varepsilon})\tilde{G}_0\nabla\tilde{G}_0=\sum_{i=1}^{15} d_i^{F_{\varepsilon}}.
\end{align*}
\medskip

\noindent For $I_{2,i}$ with $ i=2,3,4$ the splitting is similar to $I_{2,1}$, for details see the split of $\tilde{f}^{(n)}_G-\tilde{f}^{(n-1)}_G$. Thus we show some of the estimates above, where we use lemma \ref{lem3}, lemma \ref{lem5} and lemma \ref{lem4}, with $\frac{1}{q}=\eta'$. Then lemma \ref{zeta-est} and  \eqref{flux-estim}, \eqref{defgrad-estim}. Finally for the flux's difference we use lemma \ref{Jp-est}, lemma \ref{Jp-dif-est} and lemma \ref{lem2}.

\begin{align*}
&\|d_1^{F_{\varepsilon}}\|_{H^{\frac{s-1}{2}}_{(0)}L^2}=\|(J^P(\tilde{X})-J^P(\tilde{X}_{\varepsilon})-J^P+J^P_{\varepsilon})(\tilde{\zeta}-\mathcal{I})\|_{H^{\frac{s-1}{2}}_{(0)}H^{1+\eta}}\|(\tilde{G}-\tilde{G}_0)( \nabla \tilde{G}-\nabla\tilde{G}_0)\|_{H^{\frac{s-1}{2}}_{(0)}L^2}
\end{align*}
\begin{align*}
&\leq\|J^P(\tilde{X})-J^P(\tilde{X}_{\varepsilon})-J^P+J^P_{\varepsilon}\|_{H^{\frac{s-1}{2}}_{(0)}H^{1+\eta}}\|\tilde{\zeta}-\mathcal{I}\|_{H^{\frac{s-1}{2}}_{(0)}H^{1+\eta}}\|\tilde{G}-\tilde{G}_0\|_{H^{\frac{s-1}{2}}_{(0)}H^{1-\eta'}}\|\nabla \tilde{G}-\nabla\tilde{G}_0\|_{H^{\frac{s-1}{2}}_{(0)}H^{1+\eta'}}\\[2mm]
&\leq C(\tilde{v}_0,\tilde{G}_0) \left(\|J^P(\tilde{X}+\varepsilon b)-J^P(\tilde{X}_{\varepsilon})\|_{H^{\frac{s-1}{2}}_{(0)}H^{1+\eta}}+\|J^P(\tilde{X})-J^P-(J^P(\tilde{X}+\varepsilon b)-J^P_{\varepsilon})\|_{H^{\frac{s-1}{2}}_{(0)}H^{1+\eta}}\right)\\[2mm]
&\leq C(\tilde{v}_0,\tilde{G}_0) \left(\|\tilde{X}-\tilde{X}_{\varepsilon}+\varepsilon b\|_{H^{\frac{s-1}{2}}_{(0)}H^{1+\eta}}+\varepsilon\right)\\[2mm]
&\leq C(\tilde{v}_0,\tilde{G}_0)\varepsilon+ C(\tilde{v}_0,\tilde{G}_0)\left(\|\tilde{X}-\tilde{X}_{\varepsilon}+\varepsilon b-t(J^P-J^P_{\varepsilon})\tilde{v}_0\|_{H^{\frac{s-1}{2}}_{(0)}H^{1+\eta}}+\|t(J^P-J^P_{\varepsilon})\tilde{v}_0\|_{H^{\frac{s-1}{2}}_{(0)}H^{1+\eta}}\right)\\[2mm]
&\leq C(\tilde{v}_0,\tilde{G}_0)\varepsilon+ C(\tilde{v}_0,\tilde{G}_0)\left\|\int_0^t\partial_t(\tilde{X}-\tilde{X}_{\varepsilon}+\varepsilon b-t(J^P-J^P_{\varepsilon})\tilde{v}_0)\right\|_{H^{\frac{s-1}{2}+\delta_1-\delta_1}_{(0)}H^{1+\eta}}\\[2mm]
&\leq C(\tilde{v}_0,\tilde{G}_0)\varepsilon+ C(\tilde{v}_0,\tilde{G}_0)T^{\delta_1}\|\tilde{X}-\tilde{X}_{\varepsilon}+\varepsilon b-t(J^P-J^P_{\varepsilon})\tilde{v}_0\|_{H^{\frac{s-1}{2}+\delta_1}_{(0)}H^{1+\eta}}\\[2mm]
&\leq C(\tilde{v}_0,\tilde{G}_0)\varepsilon+ C(\tilde{v}_0,\tilde{G}_0)T^{\delta_1}\|\tilde{X}-\tilde{X}_{\varepsilon}+\varepsilon b-t(J^P-J^P_{\varepsilon})\tilde{v}_0\|_{\mathcal{F}^{s+1,\gamma}}.
\end{align*}
\medskip

\noindent Hence we can conclude, for $i=2,\ldots,15$

$$\|d_i^{F_{\varepsilon}}\|_{H^{\frac{s-1}{2}}_{(0)}L^2}\leq C(\tilde{v}_0,\tilde{G}_0)\varepsilon+ C(\tilde{v}_0,\tilde{G}_0)T^{\delta_i}\|\tilde{X}-\tilde{X}_{\varepsilon}+\varepsilon b-t(J^P-J^P_{\varepsilon})\tilde{v}_0\|_{\mathcal{F}^{s+1,\gamma}}.$$
\medskip

\noindent In addition for $I_{2,i}$, with $i=2,3,4$, by making the right splitting we end up in the following results.

\begin{align*}
\sum_{i=2}^4\|I_{2,i}\|_{H^{\frac{s-1}{2}}_{(0)}L^2}\leq C(\tilde{v}_0,\tilde{G}_0)\varepsilon+C(\tilde{v}_0,\tilde{G}_0)T^{\delta_{16}}&\left(\|\tilde{X}-\tilde{X}_{\varepsilon}+\varepsilon b-t(J^P-J^P_{\varepsilon})\tilde{v}_0\|_{\mathcal{F}^{s+1,\gamma}}\right.\\[2mm]
&\left.+\|\tilde{G}-\tilde{G}_{\varepsilon}-t(J^P-J^P_{\varepsilon})\nabla\tilde{v}_0\tilde{G}_0\|_{\mathcal{F}^{s,\gamma-1}}\right).
\end{align*}
\bigskip

\noindent \underline{\textbf{Estimate for $\tilde{H}_{\varepsilon}$}}\\

\noindent We rewrite the definition of $\tilde{H}_{\varepsilon}$
\begin{align*}
\tilde{H}_{\varepsilon}&=\tilde{h}-\tilde{h}_{\varepsilon}+\tilde{h}_{\phi}^L-\tilde{h}_{\phi,\varepsilon}^L+\tilde{q}_{w,\varepsilon}((J^P)^{-1}-(J^P_{\varepsilon})^{-1})\tilde{n}_0-(\nabla \tilde{w}_{\varepsilon}J^P)(J^P)^{-1}\tilde{n}_0\\
&-(\nabla \tilde{w}_{\varepsilon}J^P)^T (J^P)^{-1}\tilde{n}_0+(\nabla\tilde{w}_{\varepsilon}J^P_{\varepsilon})(J^P_{\varepsilon})^{-1}\tilde{n}_0+(\nabla \tilde{w}_{\varepsilon}J^P_{\varepsilon})^T(J^P_{\varepsilon})^{-1}\tilde{n}_0\\
&=\tilde{h}-\tilde{h}_{\varepsilon}+\tilde{h}_{\phi}^L-\tilde{h}_{\phi,\varepsilon}^L+\bar{H}_{\varepsilon}.
\end{align*}

\noindent We start by estimating $\bar{H_{\varepsilon}}$, in particular we rewrite this term as follow

\begin{align*}
&I_1=\tilde{q}_{w,\varepsilon}((J^P)^{-1}-(J^P_{\varepsilon})^{-1})\tilde{n}_0,\\[2mm]
&I_2=[\nabla \tilde{w}_{\varepsilon}(J^P_{\varepsilon}-J^P)+(\nabla \tilde{w}_{\varepsilon}(J^P_{\varepsilon}-J^P))^T](J^P)^{-1} \tilde{n}_0,\\[2mm]
&I_3=[(\nabla \tilde{w}_{\varepsilon}J^P_{\varepsilon})+(\nabla \tilde{w}_{\varepsilon} J^P_{\varepsilon}))^T]((J^P_{\varepsilon})^{-1}- (J^P)^{-1})\tilde{n}_0,
\end{align*}
\medskip

\noindent we have to estimate these quantities in $\mathcal{K}^{s-\frac{1}{2}}_{(0)}$, we can notice that 

$$\|I_1\|_{\mathcal{K}^{s-\frac{1}{2}}_{(0)}}+ \|I_2\|_{\mathcal{K}^{s-\frac{1}{2}}_{(0)}}+\|I_3\|_{\mathcal{K}^{s-\frac{1}{2}}_{(0)}}\leq C(\tilde{v}_0,\tilde{G}_0)\varepsilon,$$

\noindent thanks to Lemma \ref{stab-bound}. Now we pass to estimate 

\begin{align*}
\tilde{h}_{\phi}^L-\tilde{h}_{\phi,\varepsilon}^L&=t(\nabla\hat{\phi}-\nabla\hat{\phi}_{\varepsilon}\tilde{n}_0+\left((t\nabla\hat{\phi}_{\varepsilon}J^P_{\varepsilon})^T(J^P_{\varepsilon})^{-1}-(t\nabla\hat{\phi}J^P)^T(J^P)^{-1}\right)\tilde{n}_0\\[2mm]
&+(\tilde{G}_0\tilde{G}_0^T-\mathcal{I})\left((J^P)^{-1}-(J^P_{\varepsilon})^{-1})\right)\tilde{n}_0=J_1+J_2+J_3.
\end{align*}

\noindent We observe that for $J_1$ and $J_2$ we have just to apply lemma \ref{stab-bound} to get

$$\|J_1\|_{\mathcal{K}^{s-\frac{1}{2}}_{(0)}}+\|J_2\|_{\mathcal{K}^{s-\frac{1}{2}}_{(0)}}\leq C(\tilde{v}_0,\tilde{G}_0)\varepsilon.$$

\noindent Concerning $J_3$, we cannot estimate it directly since this term does not depend on time. For this reason we put $J_3$ with $\tilde{h}-\tilde{h}_{\varepsilon}$. As for $\tilde{f}-\tilde{f}_{\varepsilon}$ also for $\tilde{h}-\tilde{h}_{\varepsilon}$ , we have that $\tilde{h}_q-\tilde{h}_{q,\varepsilon}, \tilde{h}_{w}-\tilde{h}_{w,\varepsilon}$ have already been estimated in \cite{CCFGG2}, so we focus only on $\tilde{h}_G-\tilde{h}_{G,\varepsilon}$. Hence

\begin{align*}
\|\tilde{h}_G-\tilde{h}_{G,\varepsilon}+J_3\|_{\mathcal{K}^{s-\frac{1}{2}}_{(0)}}&=\|(\tilde{G}_{\varepsilon}\tilde{G}_{\varepsilon}^T-\mathcal{I})J^P(\tilde{X}_{\varepsilon})^{-1}\nabla_{\Lambda}\tilde{X}_{\varepsilon}\tilde{n}_0-(\tilde{G}\tilde{G}^T-\mathcal{I})J^P(\tilde{X})^{-1}\nabla_{\Lambda}\tilde{X}\tilde{n}_0 \\
&\hspace{4cm}+(\tilde{G}_0\tilde{G}_0^T-\mathcal{I})\left((J^P)^{-1}-(J^P_{\varepsilon})^{-1})\right)\tilde{n}_0\|_{\mathcal{K}^{s-\frac{1}{2}}_{(0)}}\\
&\leq \|\tilde{h}_G-\tilde{h}_{G,\varepsilon}+J_3\|_{L^2H^{s-\frac{1}{2}}}+\|\tilde{h}_G-\tilde{h}_{G,\varepsilon}+J_3\|_{H^{\frac{s}{2}-\frac{1}{4}}_{(0)}L^2}.
\end{align*}
\medskip

\noindent Before separating the estimates in the two functional spaces, we rewrite the difference as follows

\begin{align*}
&I_1=\tilde{G}_{\varepsilon} \tilde{G}^T_{\varepsilon} (J^P(\tilde{X}_{\varepsilon}))^{-1} 
(\nabla_{\Lambda} \tilde{X}_{\varepsilon}-\nabla_{\Lambda} \tilde{X}) \tilde{n}_0,\\[2mm]
&I_2=\tilde{G}_{\varepsilon}
\tilde{G}^T_{\varepsilon}((J^P(\tilde{X}_{\varepsilon}))^{-1}-(J^P(\tilde{X}))^{-1} )\nabla_{\Lambda} \tilde{X} \tilde{n}_0+\tilde{G}_0\tilde{G}_0^T((J^P)^{-1}-(J^P_{\varepsilon})^{-1}))\tilde{n}_0,\\[2mm]
&I_3=\tilde{G}_{\varepsilon}(\tilde{G}^T_{\varepsilon}-\tilde{G}^T)(J^P(\tilde{X}))^{-1} \nabla_{\Lambda} \tilde{X} \tilde{n}_0,\\[2mm]
&I_4=(\tilde{G}_{\varepsilon}-\tilde{G})\tilde{G}^T (J^P(\tilde{X}))^{-1}\nabla_{\Lambda} \tilde{X} \tilde{n}_0,\\[2mm]
&I_5=(J^P(\tilde{X}))^{-1} (\nabla_{\Lambda} \tilde{X}-\nabla_{\Lambda}\tilde{X}_{\varepsilon}) \tilde{n}_0,\\[2mm]
&I_6=((J^P(\tilde{X}))^{-1}-(J^P(\tilde{X}_{\varepsilon}))^{-1} )\nabla_{\Lambda} \tilde{X}_{\varepsilon}\tilde{n}_0-((J^P)^{-1}-(J^P_{\varepsilon})^{-1}))\tilde{n}_0.
\end{align*}
\medskip

\noindent At first we deal with $L^2H^{s-\frac{1}{2}}$, thus we need to split the terms above in order to apply lemma \ref{Jp-est}, lemma \ref{zeta-est} and \eqref{flux-estim}, \eqref{defgrad-estim}. We show how to split $I_2$ and $I_4$, which give the two differences we need for the estimate. For $I_2$ we have

\begin{align*}
I_2&=(\tilde{G}_{\varepsilon}-\tilde{G}_0)
(\tilde{G}^T_{\varepsilon}-\tilde{G}_0^T)((J^P(\tilde{X}_{\varepsilon}))^{-1}-(J^P(\tilde{X}))^{-1}-(J^P_{\varepsilon})^{-1}+(J^P)^{-1})\nabla_{\Lambda} \tilde{X} \tilde{n}_0\\[2mm]
&+(\tilde{G}_{\varepsilon}-\tilde{G}_0)
(\tilde{G}^T_{\varepsilon}-\tilde{G}_0^T)((J^P_{\varepsilon})^{-1}-(J^P)^{-1})\nabla_{\Lambda} \tilde{X} \tilde{n}_0\\[2mm]
&+\tilde{G}_0(\tilde{G}^T_{\varepsilon}-\tilde{G}_0^T)((J^P(\tilde{X}_{\varepsilon}))^{-1}-(J^P(\tilde{X}))^{-1}-(J^P_{\varepsilon})^{-1}+(J^P)^{-1})\nabla_{\Lambda} \tilde{X} \tilde{n}_0\\[2mm]
&+\tilde{G}_0
(\tilde{G}^T_{\varepsilon}-\tilde{G}_0^T)((J^P_{\varepsilon})^{-1}-(J^P)^{-1})\nabla_{\Lambda} \tilde{X} \tilde{n}_0\\[2mm]
&+(\tilde{G}_{\varepsilon}-\tilde{G}_0)
\tilde{G}_0^T((J^P(\tilde{X}_{\varepsilon}))^{-1}-(J^P(\tilde{X}))^{-1}-(J^P_{\varepsilon})^{-1}+(J^P)^{-1})\nabla_{\Lambda} \tilde{X} \tilde{n}_0\\[2mm]
&+(\tilde{G}_{\varepsilon}-\tilde{G}_0)\tilde{G}_0^T((J^P_{\varepsilon})^{-1}-(J^P)^{-1})\nabla_{\Lambda} \tilde{X} \tilde{n}_0\\[2mm]
&+\tilde{G}_0\tilde{G}_0^T((J^P(\tilde{X}_{\varepsilon}))^{-1}-(J^P(\tilde{X}))^{-1}-(J^P_{\varepsilon})^{-1}+(J^P)^{-1})\nabla_{\Lambda} \tilde{X} \tilde{n}_0\\[2mm]
&+\tilde{G}_0\tilde{G}_0^T((J^P_{\varepsilon})^{-1}-(J^P)^{-1})(\nabla_{\Lambda} \tilde{X}-\mathcal{I}) \tilde{n}_0=\sum_{i=1}^8d_{i}^{H_{\varepsilon}}.
\end{align*}
\medskip

\noindent By using trace theorem \ref{parabolic-trace}, lemma \ref{zeta-est} and \eqref{flux-estim}, \eqref{defgrad-estim}, we get the estimate of $d_1^{H_{\varepsilon}}$

\begin{align*}
&\|d_1^{H_{\varepsilon}}\|_{L^2H^{s-\frac{1}{2}}}\leq \|\tilde{G}_{\varepsilon}-\tilde{G}_0\|_{L^2H^{s-\frac{1}{2}}}\|\tilde{G}^T_{\varepsilon}-\tilde{G}_0^T\|_{L^{\infty}H^{s-\frac{1}{2}}}\|\nabla_{\Lambda} \tilde{X}\|_{L^{\infty}H^{s-\frac{1}{2}}}\\
&\hspace{2.6cm}\cdot\|(J^P(\tilde{X}_{\varepsilon}))^{-1}-(J^P(\tilde{X}))^{-1}-(J^P_{\varepsilon})^{-1}+(J^P)^{-1}\|_{L^{\infty}H^{s-\frac{1}{2}}}\\[2mm]
&\leq T^{\frac{1}{2}}\|\tilde{G}_{\varepsilon}-\tilde{G}_0\|_{L^{\infty}H^{s}}\|\tilde{G}^T_{\varepsilon}-\tilde{G}_0^T\|_{L^{\infty}H^{s}}\|\nabla_{\Lambda} \tilde{X}\|_{L^{\infty}H^{s}}\|(J^P(\tilde{X}_{\varepsilon}))^{-1}-(J^P(\tilde{X}))^{-1}-(J^P_{\varepsilon})^{-1}+(J^P)^{-1}\|_{L^{\infty}H^{s}}\\[2mm]
&\leq C(\tilde{v}_0, \tilde{G}_0) T^{\frac{1}{2}}\left(\|(J^P(\tilde{X}_{\varepsilon}))^{-1}-(J^P(\tilde{X}+\varepsilon b))^{-1}\|_{L^{\infty}H^{s}}\right.\\
&\left.\hspace{2.6cm}+\|(J^P(\tilde{X}+\varepsilon b))^{-1}-(J^P(\tilde{X}))^{-1}-(J^P_{\varepsilon})^{-1}+(J^P)^{-1}\|_{L^{\infty}H^{s}}\right)\\[2mm]
&\leq C(\tilde{v}_0, \tilde{G}_0) T^{\frac{1}{2}}\left(\|\tilde{X}_{\varepsilon}-\tilde{X}-\varepsilon b\|_{L^{\infty}H^{s}}+\varepsilon\right)\\[2mm]
&\leq C(\tilde{v}_0, \tilde{G}_0) \varepsilon+ C(\tilde{v}_0, \tilde{G}_0) T^{\frac{3}{4}}\|\tilde{X}-\tilde{X}_{\varepsilon}+\varepsilon b -t(J^P-J^P_{\varepsilon})\tilde{v}_0\|_{\mathcal{F}^{s+1,\gamma}}.
\end{align*}
\medskip

\noindent We can deduce 

$$\sum_{i=2}^8\|d_i^{H_{\varepsilon}}\|_{L^2H^{s-\frac{1}{2}}}\leq C(\tilde{v}_0, \tilde{G}_0) \varepsilon+ C(\tilde{v}_0, \tilde{G}_0) T^{\frac{3}{4}}\|\tilde{X}-\tilde{X}_{\varepsilon}+\varepsilon b -t(J^P-J^P_{\varepsilon})\tilde{v}_0\|_{\mathcal{F}^{s+1,\gamma}}.$$
\medskip

\noindent The splitting for $I_4$ is the following

\begin{align*}
I_4&=(\tilde{G}_{\varepsilon}-\tilde{G})(\tilde{G}^T-\tilde{G}_0^T) (J^P(\tilde{X}))^{-1}\nabla_{\Lambda} \tilde{X} \tilde{n}_0+(\tilde{G}_{\varepsilon}-\tilde{G})\tilde{G}_0^T (J^P(\tilde{X}))^{-1}\nabla_{\Lambda} \tilde{X} \tilde{n}_0=d_{9}^{H_{\varepsilon}}+d_{10}^{H_{\varepsilon}}
\end{align*}
\medskip

\noindent We study $d_{9}^{H_{\varepsilon}}$ by using trace theorem \ref{parabolic-trace}, lemma \ref{Jp-est}, lemma \ref{zeta-est} and \eqref{flux-estim}, \eqref{defgrad-estim}.

\begin{align*}
&\|d_{9}^{H_{\varepsilon}}\|_{L^2H^{s-\frac{1}{2}}}\leq \|\tilde{G}_{\varepsilon}-\tilde{G}\|_{L^{\infty}H^{s-\frac{1}{2}}}\|\tilde{G}^T-\tilde{G}_0^T\|_{L^2H^{s-\frac{1}{2}}}\|J^P(\tilde{X}))^{-1}\|_{L^{\infty}H^{s-\frac{1}{2}}}\|\nabla_{\Lambda} \tilde{X}\|_{L^{\infty}H^{s-\frac{1}{2}}}\\[2mm]
&\leq \|\tilde{G}_{\varepsilon}-\tilde{G}\|_{L^{\infty}H^{s}}\|\tilde{G}^T-\tilde{G}_0^T\|_{L^2H^{s}}\|J^P(\tilde{X}))^{-1}\|_{L^{\infty}H^{s}}\|\nabla_{\Lambda} \tilde{X}\|_{L^{\infty}H^{s}}\\[2mm]
&\leq C(\tilde{v}_0,\tilde{G}_0)T^{\frac{1}{2}}\left(\|\tilde{G}_{\varepsilon}-\tilde{G}+t(J^P-J^P_{\varepsilon})\nabla\tilde{v}_0\tilde{G}_0\|_{L^{\infty}H^{s}}+\|t(J^P_{\varepsilon}-J^P)\nabla\tilde{v}_0\tilde{G}_0\|_{L^{\infty}H^{s}}\right)\\[2mm]
&\leq C(\tilde{v}_0,\tilde{G}_0)\varepsilon+ C(\tilde{v}_0,\tilde{G}_0)T^{\frac{3}{4}}\|\tilde{G}-\tilde{G}_{\varepsilon}-t(J^P-J^P_{\varepsilon})\nabla\tilde{v}_0\tilde{G}_0\|_{\mathcal{F}^{s,\gamma-1}}.
\end{align*}
\medskip

\noindent We summarize the final result
\begin{align*}
\sum_{i=1}^6\|I_i\|_{L^2H^{s-\frac{1}{2}}}\leq C(\tilde{v}_0,\tilde{G}_0)\varepsilon+ C(\tilde{v}_0,\tilde{G}_0)T^{\frac{3}{4}}&\left(\|\tilde{X}-\tilde{X}_{\varepsilon}+\varepsilon b -t(J^P-J^P_{\varepsilon})\tilde{v}_0\|_{\mathcal{F}^{s+1,\gamma}}\right.\\
&\left.+\|\tilde{G}-\tilde{G}_{\varepsilon}-t(J^P-J^P_{\varepsilon})\nabla\tilde{v}_0\tilde{G}_0\|_{\mathcal{F}^{s,\gamma-1}}\right)
\end{align*}
\medskip

\noindent In the end we have to split $I_i$, for $i=1,\ldots,6$ in order to do estimates  in $H^{\frac{s}{2}-\frac{1}{4}}_{(0)}L^2$. In this case we show how to separate $I_4$ and $I_6$ to give a general idea on the splitting of all the terms, for details see the term $\tilde{h}^{(n)}_G-\tilde{h}^{(n-1)}_G$ in proposition \ref{estim-(v,q)}. For $I_4$ we have

\begin{align*}
I_4&=(\tilde{G}_{\varepsilon}-\tilde{G})(\tilde{G}^T-\tilde{G}_0^T) ((J^P(\tilde{X}))^{-1}-(J^P)^{-1})(\nabla_{\Lambda} \tilde{X} -\mathcal{I})\tilde{n}_0\\[2mm]
&+(\tilde{G}_{\varepsilon}-\tilde{G})(\tilde{G}^T-\tilde{G}_0^T) (J^P)^{-1}(\nabla_{\Lambda} \tilde{X} -\mathcal{I})\tilde{n}_0\\[2mm]
&+(\tilde{G}_{\varepsilon}-\tilde{G})\tilde{G}_0^T ((J^P(\tilde{X}))^{-1}-(J^P)^{-1})(\nabla_{\Lambda} \tilde{X} -\mathcal{I})\tilde{n}_0\\[2mm]
&+(\tilde{G}_{\varepsilon}-\tilde{G})\tilde{G}_0^T (J^P)^{-1}(\nabla_{\Lambda} \tilde{X} -\mathcal{I})\tilde{n}_0\\[2mm]
&+(\tilde{G}_{\varepsilon}-\tilde{G})(\tilde{G}^T-\tilde{G}_0^T) ((J^P(\tilde{X}))^{-1}-(J^P)^{-1})\tilde{n}_0\\[2mm]
&+(\tilde{G}_{\varepsilon}-\tilde{G})(\tilde{G}^T-\tilde{G}_0^T) (J^P)^{-1}\tilde{n}_0=\sum_{i=1}^6 d_i^{H_{\varepsilon}}.
\end{align*}
\medskip

\noindent We estimate $d_1^{H_{\varepsilon}}$ by using lemma \ref{lem3}, with $\eta>\frac{1}{2}$, lemma \ref{lem5} and lemma \ref{lem4}, with $\eta'<\frac{1}{2}$. Moreover we apply the trace theorem and lemma \ref{Jp-est} with the estimates \eqref{flux-estim} and \eqref{defgrad-estim}. In conclusion lemma \ref{lem2} give us the final result.

\begin{align*}
&\|d_1^{H_{\varepsilon}}\|_{H^{\frac{s}{2}-\frac{1}{4}}_{(0)}L^2}\leq\|(\tilde{G}_{\varepsilon}-\tilde{G})(\tilde{G}^T-\tilde{G}_0^T)\|_{H^{\frac{s}{2}-\frac{1}{4}}_{(0)}L^2}\| ((J^P(\tilde{X}))^{-1}-(J^P)^{-1})(\nabla_{\Lambda} \tilde{X} -\mathcal{I})\|_{H^{\frac{s}{2}-\frac{1}{4}}_{(0)}H^{\frac{1}{2}+\eta}}\\[2mm]
&\leq \|\tilde{G}_{\varepsilon}-\tilde{G}\|_{H^{\frac{s}{2}-\frac{1}{4}}_{(0)}H^{\frac{1}{2}-\eta'}}\|\tilde{G}^T-\tilde{G}_0^T\|_{H^{\frac{s}{2}-\frac{1}{4}}_{(0)}H^{\frac{1}{2}+\eta'}}\|J^P(\tilde{X})^{-1}-(J^P)^{-1}\|_{H^{\frac{s}{2}-\frac{1}{4}}_{(0)}H^{\frac{1}{2}+\eta}}\|\nabla_{\Lambda} \tilde{X} -\mathcal{I}\|_{H^{\frac{s}{2}-\frac{1}{4}}_{(0)}H^{\frac{1}{2}+\eta}}\\[2mm]
&\leq \|\tilde{G}_{\varepsilon}-\tilde{G}\|_{H^{\frac{s}{2}-\frac{1}{4}}_{(0)}H^{1-\eta'}}\|\tilde{G}^T-\tilde{G}_0^T\|_{H^{\frac{s}{2}-\frac{1}{4}}_{(0)}H^{1+\eta'}}\|J^P(\tilde{X})^{-1}-(J^P)^{-1}\|_{H^{\frac{s}{2}-\frac{1}{4}}_{(0)}H^{1+\eta}}\|\nabla_{\Lambda} \tilde{X} -\mathcal{I}\|_{H^{\frac{s}{2}-\frac{1}{4}}_{(0)}H^{1+\eta}}\\[2mm]
&\leq C(\tilde{v}_0,\tilde{G}_0) \left(\|\tilde{G}_{\varepsilon}-\tilde{G}+t(J^P-J^P_{\varepsilon})\nabla\tilde{v}_0\tilde{G}_0\|_{H^{\frac{s}{2}-\frac{1}{4}}_{(0)}H^{1-\eta'}}+\|t(J^P_{\varepsilon}-J^P)\nabla\tilde{v}_0\tilde{G}_0\|_{H^{\frac{s}{2}-\frac{1}{4}}_{(0)}H^{1-\eta'}}\right)
\end{align*}
\begin{align*}
&\leq C(\tilde{v}_0,\tilde{G}_0)\varepsilon+C(\tilde{v}_0,\tilde{G}_0)\left\|\partial_t\int_0^t \tilde{G}_{\varepsilon}-\tilde{G}+t(J^P-J^P_{\varepsilon})\nabla\tilde{v}_0\tilde{G}_0\right\|_{H^{\frac{s}{2}-\frac{1}{4}+\delta_1-\delta_1}_{(0)}H^{1-\eta'}}\\[2mm]
&\leq C(\tilde{v}_0,\tilde{G}_0)\varepsilon+C(\tilde{v}_0,\tilde{G}_0)T^{\delta_1}\|\tilde{G}_{\varepsilon}-\tilde{G}+t(J^P-J^P_{\varepsilon})\nabla\tilde{v}_0\tilde{G}_0\|_{\mathcal{F}^{s,\gamma-1}}.
\end{align*}
\medskip

\noindent For $i=2,\ldots,6$, $d_i^{H_{\varepsilon}}$ is exactly the same estimate. Now we rewrite in a better way $I_6$.

\begin{align*}
I_6&=(J^P(\tilde{X})^{-1}-J^P(\tilde{X}_{\varepsilon})^{-1}-((J^P)^{-1}-(J^P_{\varepsilon})^{-1} )(\nabla_{\Lambda}\tilde{X}_{\varepsilon}-\mathcal{I})\tilde{n}_0+((J^P)^{-1}-(J^P_{\varepsilon})^{-1} )(\nabla_{\Lambda}\tilde{X}_{\varepsilon}-\mathcal{I})\tilde{n}_0\\[2mm]
&+(J^P(\tilde{X})^{-1}-J^P(\tilde{X}_{\varepsilon})^{-1}-((J^P)^{-1}-(J^P_{\varepsilon})^{-1} )\tilde{n}_0=\sum_{7}^{10} d_{10}^{H_{\varepsilon}}.
\end{align*}
\medskip

\noindent We observe that in $d_{10}^{H_{\varepsilon}}$ there is the presence of the term of $\tilde{h}_{\phi}^L-\tilde{h}_{\phi,\varepsilon}^L$.  By using lemma \ref{lem4}, with $\eta<\frac{1}{2}$, the trace theorem \ref{parabolic-trace}, lemma \ref{zeta-est}. Finally by splitting the terms with the flux in a right way we can apply lemma \ref{lem2} to get the difference of the flux.

\begin{align*}
&\|d_7^{H_{\varepsilon}}\|_{H^{\frac{s}{2}-\frac{1}{4}}_{(0)}L^2}\leq \|J^P(\tilde{X})^{-1}-J^P(\tilde{X}_{\varepsilon})^{-1}-((J^P)^{-1}-(J^P_{\varepsilon})^{-1}\|_{H^{\frac{s}{2}-\frac{1}{4}}_{(0)}H^{\frac{1}{2}+\eta}}\|\nabla_{\Lambda}\tilde{X}_{\varepsilon}-\mathcal{I}\|_{H^{\frac{s}{2}-\frac{1}{4}}_{(0)}H^{\frac{1}{2}-\eta}}\\[2mm]
&\leq \|J^P(\tilde{X})^{-1}-J^P(\tilde{X}_{\varepsilon})^{-1}-((J^P)^{-1}-(J^P_{\varepsilon})^{-1}\|_{H^{\frac{s}{2}-\frac{1}{4}}_{(0)}H^{1+\eta}}\|\nabla_{\Lambda}\tilde{X}_{\varepsilon}-\mathcal{I}\|_{H^{\frac{s}{2}-\frac{1}{4}}_{(0)}H^{1-\eta}}\\[2mm]
&\leq C(\tilde{v}_0) \left(\|(J^P(\tilde{X}_{\varepsilon}))^{-1}-(J^P(\tilde{X}+\varepsilon b))^{-1}\|_{H^{\frac{s}{2}-\frac{1}{4}}_{(0)}H^{1-\eta}}\right.\\
&\left.\hspace{2cm}+\|(J^P(\tilde{X}+\varepsilon b))^{-1}-(J^P(\tilde{X}))^{-1}-(J^P_{\varepsilon})^{-1}+(J^P)^{-1}\|_{H^{\frac{s}{2}-\frac{1}{4}}_{(0)}H^{1-\eta}}\right)\\[2mm]
&\leq C(\tilde{v}_0) \left(\|\tilde{X}_{\varepsilon}-\tilde{X}-\varepsilon b\|_{H^{\frac{s}{2}-\frac{1}{4}}_{(0)}H^{1-\eta}}+\varepsilon\right)\\[2mm]
&\leq C(\tilde{v}_0)\varepsilon+ C(\tilde{v}_0)\left(\|\tilde{X}-\tilde{X}_{\varepsilon}+\varepsilon b-t(J^P-J^P_{\varepsilon})\tilde{v}_0\|{H^{\frac{s}{2}-\frac{1}{4}}_{(0)}H^{1-\eta}}+\|t(J^P-J^P_{\varepsilon})\tilde{v}_0\|_{H^{\frac{s}{2}-\frac{1}{4}}_{(0)}H^{1-\eta}}\right)\\[2mm]
&\leq C(\tilde{v}_0)\varepsilon+ C(\tilde{v}_0)\left\|\partial_t\int_0^t\tilde{X}-\tilde{X}_{\varepsilon}+\varepsilon b-t(J^P-J^P_{\varepsilon})\tilde{v}_0\right\|_{H^{\frac{s}{2}-\frac{1}{4}+\delta_2-\delta_2}_{(0)}H^{1-\eta}}\\[2mm]
&\leq C(\tilde{v}_0)\varepsilon+ C(\tilde{v}_0) T^{\delta_2}\|\tilde{X}-\tilde{X}_{\varepsilon}+\varepsilon b-t(J^P-J^P_{\varepsilon})\tilde{v}_0\|_{\mathcal{F}^{s+1,\gamma}}.
\end{align*}
\medskip

\noindent With this final estimate we can resume here that all the splitting terms give the following

\begin{align*}
\sum_{i=1}^6\|I_i\|_{H^{\frac{s}{2}-\frac{1}{4}}_{(0)}L^2}\leq C(\tilde{v}_0,\tilde{G}_0)\varepsilon+ C(\tilde{v}_0,\tilde{G}_0)T^{\delta}&\left(\|\tilde{X}-\tilde{X}_{\varepsilon}+\varepsilon b -t(J^P-J^P_{\varepsilon})\tilde{v}_0\|_{\mathcal{F}^{s+1,\gamma}}\right.\\
&\left.+\|\tilde{G}-\tilde{G}_{\varepsilon}-t(J^P-J^P_{\varepsilon})\nabla\tilde{v}_0\tilde{G}_0\|_{\mathcal{F}^{s,\gamma-1}}\right).
\end{align*}

\noindent By putting together the estimate for $\tilde{F}_{\varepsilon}$ in $\mathcal{K}^{s-1}_{(0)}$, the estimate for $\tilde{K}_{\varepsilon}$ in $\mathcal{\bar{K}}^s_{(0)}$, which we skipped because it can be found in \cite[Lemma 6.2]{CCFGG2} and the estimate for $\tilde{H}_{\varepsilon}$, then the proposition holds.
\end{proof}

\section{Existence of splash singularity}\label{sec6}
\noindent As we state in the Introduction, the choice of the initial velocity plays a fundamental role to have splash type singularities. Indeed by taking a positive normal component of the velocity, as represented below in fig. 3, we get that the unperturbed domain evolves to create a self-intersecting domain $P^{-1}(\tilde{\Omega}(t))$, for a suitable time $t>0$.

\begin{figure}[htbp]
\includegraphics[scale=0.46] {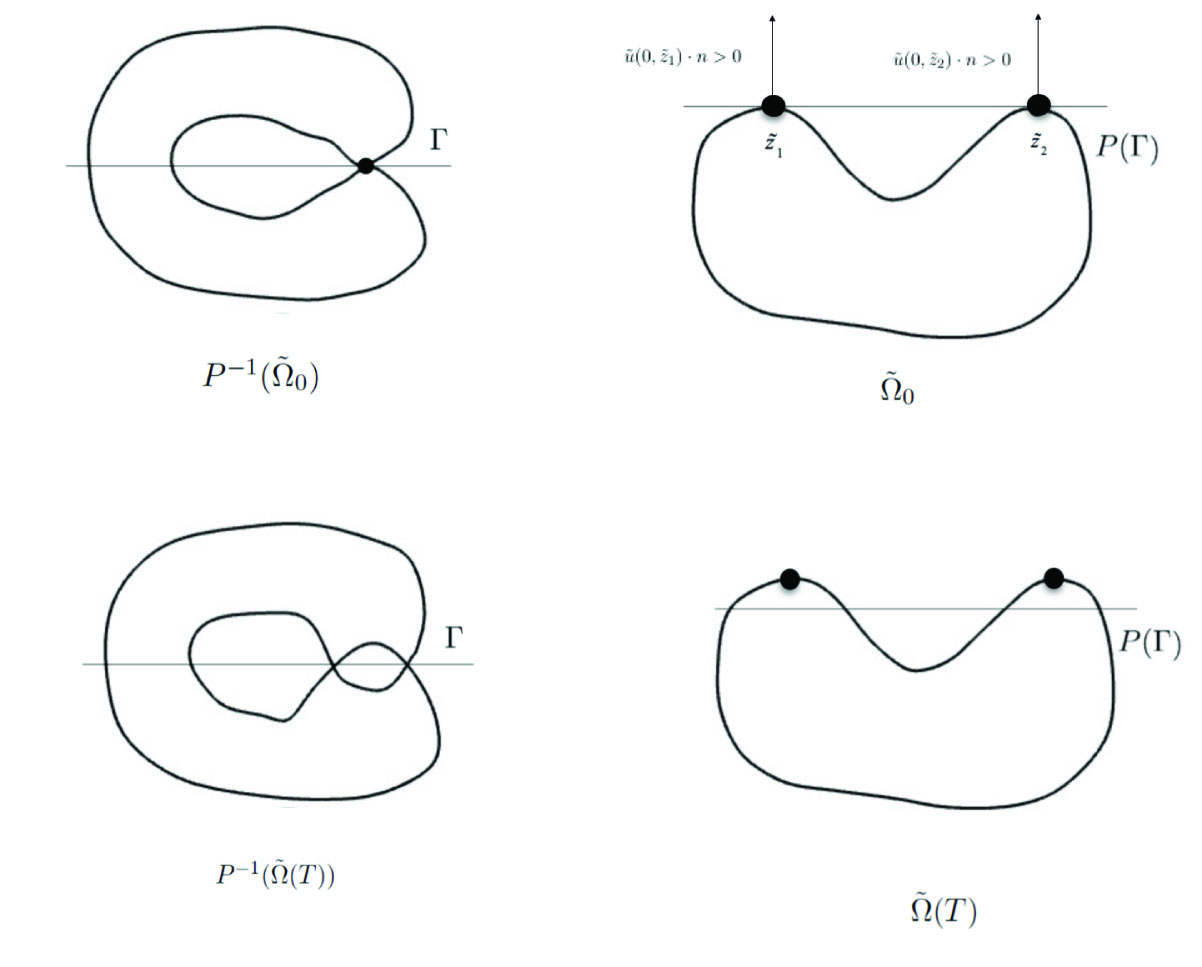}
\caption{}
\end{figure}

\subsection{Suitable choice of the initial velocity}

\noindent We are looking for initial data that must satisfy the compatibility conditions (\ref{compatibility}). Let's start the analysis by taking into account the Navier-Stokes system, without the presence of the deformation gradient $F$. In this case the compatibility condition that the initial velocity $u_0$ has to satisfy on the boundary $\partial\Omega$ is the following
\begin{equation}\label{NS-compcond}
\theta (\nabla u_0+\nabla u_0^T) n=0,
\end{equation}
\noindent where $\theta, n$ are the tangential and normal unit vectors, respectively.\\
\noindent For our problem, we extend the analysis for the choice of the initial velocity already made in \cite[Sec. 7]{CCFGG2}. For convenience of the reader we recall here the arguments in \cite{CCFGG2}. Let us consider a neighborhood $U$ of the boundary $\partial\Omega$, we can use  a coordinates system $(s,\lambda)$ given by $x(s,\lambda)=z(s)+\lambda z_s^{\perp}(s)$ and define a stream function $\psi$ by using the following quadratic expansion

\begin{equation}\label{stream}
\psi(x(s,\lambda))=\bar{\psi}(s,\lambda)=\psi_0(s)+\lambda\psi_1(s)+\frac{1}{2}\lambda^2\psi_2(s).
\end{equation}
\medskip

\noindent Consequently we extend on $U$ both $\theta$ and $n$ in the following way

\begin{equation*}
\left\{\begin{array}{lll}
\Theta(s,\lambda)=x_s(s,\lambda)=z_s(s)+\lambda z_{ss}^{\perp}(s)=(1-\lambda k(s))z_s(s)\\[3mm]
N(s,\lambda)=x_{\lambda}=z_s^{\perp},
\end{array}\right.
\end{equation*}
\noindent where $k(s)=z_{ss}\cdot z_s^{\perp}$  is the scalar curvature.
\medskip

\noindent We define $u_0=\nabla^{\perp}\psi$ and then we  compute  the compatibility condition by means of the stream function. Therefore the compatibility condition  for the Navier-Stokes (\ref{NS-compcond}) is equivalent to

\begin{equation}\label{cc}
\Theta(\nabla u_0+\nabla u_0) N=\partial_s^2\psi_0(s)-\psi_2(s)=0.
\end{equation}

\noindent As $u_0\cdot N=\partial_s \psi_0(s)$, first of all we take $\psi_0(s)$ in order to choose a positive normal component of the velocity and consequently $\psi_2(s)$ in such a way that condition (\ref{cc}) is satisfied.  Now, since the normal component of the velocity depends only on the stream function and does not depend on the boundary conditions, it suggests for the viscoelastic problem, that $u_0\cdot n$ does not depend on the deformation gradient, so in a similar way as before, the compatibility condition (\ref{compatibility}) 

\begin{equation}\label{newcompcond}
\partial_s^2\psi_0(s)- \psi_2(s)=-(\Theta(F_0 F_0^T-\mathcal{I})N)_{|\lambda=0},
\end{equation}
\noindent  At this point we can state the following proposition.

\begin{proposition}
For a given $\psi_0\in C^2(U)$ such that $\partial_s\psi_0>0$ and for any divergence free deformation tensor $F_0\in C^1(U)$, there exists a stream function of the type (\ref{stream}) such that $\bar{u}_0(s,\lambda)=u_0(x(s,\lambda))=\nabla^{\perp}\bar{\psi}(s,\lambda)$ and (\ref{newcompcond}) is satisfied.
\end{proposition}
\begin{proof}
Since $u_0=\nabla^{\perp}\psi$ it follows
\begin{equation}\label{v-normal}
\displaystyle\begin{array}{lll}
\bar{u}_0\cdot N=\left(\frac{1}{2(1-\lambda k(s))}\partial_s\bar{\psi}(s,\lambda)z_s^{\perp}(s)-\frac{1}{2}\partial_{\lambda}\bar{\psi}(s,\lambda)z_s(s)\right)z_s^{\perp}(s)\\[3mm]
\displaystyle\hspace{1.1cm}=\frac{1}{2(1-\lambda k(s))}\partial_s\bar{\psi}(s,\lambda).
\end{array}
\end{equation}
\medskip

\noindent Then by substituting $u_0, \Theta, N$ we have

\begin{equation*}
\Theta(\nabla u_0+\nabla u_0) N=\partial_s^2\psi_0(s)-\psi_2(s)-k(s)\psi_1(s),
\end{equation*}
\noindent consequently

\begin{equation}\label{newvisc-compcond}
\partial_s^2\psi_0(s)-\psi_2(s)-k(s)\psi_1(s)=-(\Theta(F_0 F_0^T-\mathcal{I}))N)_{\lambda=0}.
\end{equation}
\medskip

\noindent Thus for a given $\psi_0$ such that $\partial_s\psi_0>0$ and for any $F_0$, there exist  $\psi_1$ and $\psi_2$, such that (\ref{newvisc-compcond}) is satisfied and in particular, by choosing $\psi_1=0$, we get that (\ref{newcompcond}) is satisfied.
\end{proof}
\medskip

\noindent The two main ingredients for proving the existence of a splash singuarity are the stability result, already shown in Theorem \ref{main-stab}, and the construction of the initial data $(u_0,F_0)$ such that $u_0\cdot n>0$, determined in the previous proposition. This choice of the initial velocity allows us to obtain a domain $\tilde{\Omega}(\bar{t})$, such that $P^{-1}(\partial\tilde{\Omega}(\bar{t}))$ is a self-intersecting domain, for a suitable positive time. Hence by (\ref{Xcloseness}) we have that $P^{-1}(\partial\tilde{\Omega}_{\varepsilon}(\bar{t}))$ is also self-intersecting. In particular, at time $t=0$ we have $P^{-1}(\partial\tilde{\Omega}_{\varepsilon}(0))$, that is regular and for a latter time we end up in a self-intersecting domain, then the continuity argument guarantees the existence of a splash time $t^*\in(0,\bar{t})$. Thus we state the following theorem.

\begin{theorem}\label{existence-splash}
There exist  $\Omega_0$ bounded domain with a sufficient smooth boundary and for $k$ big enough, $u_0\in H^k(\Omega_0)$, such that for any divergence free $F_0\in H^k(\Omega_0)$, $ \det F_0=1 $, there exist $t^*>0$ and a regular solution $\{X-\alpha-t v_0, w, p, F-F_0-t\nabla v_0 F_0\}\in\mathcal{F}^{s+1,\gamma}\times\mathcal{K}^{s+1}_{(0)}\times\mathcal{K}^s_{pr(0)}\times\mathcal{F}^{s,\gamma-1}$ in $[0,t^*)$, $2<s<\frac{5}{2}, 1<\gamma<s-1$, such that the interface $\partial\Omega(t^*)$ self-intersects at least in one point and creates a splash singularity.
\end{theorem}

\subsection*{Acknowledgements:} 
The authors would like to thank the anonymous referee for her/his comments and suggestions.

\appendix\section{}
\subsection{Further estimates}
In this section we resume all the estimates  we use throught the proofs and we need also to figure out some estimates for the flux and the deformation gradient. In particular during the proofs we need the following results for $\tilde{X}-\tilde{\alpha}$.

\begin{equation}\label{flux-estim}
\begin{split}
&1.\hspace{0.5cm} \|\tilde{X}-\tilde{\alpha}\|_{L^{\infty}H^{s+1}}\leq \|\tilde{X}-\tilde{\alpha}-tJ^p \tilde{v}_0\|_{L^{\infty}H^{s+1}}+\|tJ^p \tilde{v}_0\|_{L^{\infty}H^{s+1}}\\[2mm]
&\hspace{3.5cm}\leq T^{\frac{1}{4}}\|\tilde{X}-\tilde{\alpha}-tJ^p \tilde{v}_0\|_{L^{\infty}_{\frac{1}{4}}H^{s+1}}+T\|J^p \tilde{v}_0\|_{H^{s+1}}\\[2mm]
&\hspace{3.5cm}\leq C(\tilde{v}_0)T^{\frac{1}{4}}\|\tilde{X}-\tilde{\alpha}-tJ^p \tilde{v}_0\|_{\mathcal{F}^{s+1,\gamma}}+C(\tilde{v}_0)T.\\[5mm]
&2.\hspace{0.5cm} \|\tilde{X}-\tilde{\alpha}\|_{H^1_{(0)}H^{\gamma}}\leq \|\tilde{X}-\tilde{\alpha}-tJ^p \tilde{v}_0\|_{H^1_{(0)}H^{\gamma}}+\|tJ^p \tilde{v}_0\|_{H^1_{(0)}H^{\gamma}}\\[2mm]
&\hspace{3.3cm}\leq \left\|\int_0^t\partial_t(\tilde{X}-\tilde{\alpha}-tJ^p \tilde{v}_0)\right\|_{H^{1+\eta-\delta}_{(0)}H^{\gamma}}+C(\tilde{v}_0)\|t\|_{H^1_{(0)}}\\[2mm]
&\hspace{3.3cm}\leq C(\tilde{v}_0)T^{\delta}\|\tilde{X}-\tilde{\alpha}-tJ^p \tilde{v}_0\|_{H^{1+\eta}_{(0)}H^{\gamma}}+C(\tilde{v}_0)T^{\frac{1}{2}}\\[2mm]
&\hspace{3.3cm}\leq C(\tilde{v}_0)T^{\delta}\|\tilde{X}-\tilde{\alpha}-tJ^p \tilde{v}_0\|_{\mathcal{F}^{s+1,\gamma}}+C(\tilde{v}_0)T^{\frac{1}{2}}.\\[5mm]
&3.\hspace{0.5cm} \|\tilde{X}-\tilde{\alpha}\|_{H^{\frac{s-1}{2}}_{(0)}H^{2+\mu}}\leq \|\tilde{X}-\tilde{\alpha}-tJ^p \tilde{v}_0\|_{H^{\frac{s-1}{2}}_{(0)}H^{2+\mu}}+\|tJ^p \tilde{v}_0\|_{H^{\frac{s-1}{2}}_{(0)}H^{2+\mu}}\\[2mm]
&\hspace{3.8cm}\leq \left\|\int_0^t\partial_t(\tilde{X}-\tilde{\alpha}-tJ^p \tilde{v}_0)\right\|_{H^{\frac{s-1}{2}+\delta-\delta}_{(0)}H^{2+\mu}}+C(\tilde{v}_0)\|t\|_{H^{\frac{s-1}{2}}_{(0)}}\\[2mm]
&\hspace{3.8cm}\leq C(\tilde{v}_0)T^{\delta}\|\tilde{X}-\tilde{\alpha}-tJ^p \tilde{v}_0\|_{H^{\frac{s-1}{2}+\delta}_{(0)}H^{2+\mu}}+C(\tilde{v}_0)T^{\frac{1}{2}}\\[2mm]
&\hspace{3.8cm}\leq C(\tilde{v}_0)T^{\delta}\|\tilde{X}-\tilde{\alpha}-tJ^p \tilde{v}_0\|_{\mathcal{F}^{s+1,\gamma}}+C(\tilde{v}_0)T^{\frac{1}{2}}.
\end{split}
\end{equation}
\medskip

\noindent And the following results for $\tilde{G}-\tilde{G}_0$.

\begin{equation}\label{defgrad-estim}
\begin{split}
&1.\hspace{0.5cm} \|\tilde{G}-\tilde{G}_0\|_{L^{\infty}H^{s}}\leq \|\tilde{G}-\tilde{G}_0-tJ^p\nabla\tilde{v}_0\tilde{G}_0\|_{L^{\infty}H^{s}}+\|tJ^p\nabla\tilde{v}_0\tilde{G}_0\|_{L^{\infty}H^{s}}\\[2mm]
&\hspace{3.5cm}\leq T^{\frac{1}{4}}\|\tilde{G}-\tilde{G}_0-tJ^p\nabla\tilde{v}_0\tilde{G}_0\|_{L^{\infty}_{\frac{1}{4}}H^{s}}+T\|J^p\nabla\tilde{v}_0\tilde{G}_0\|_{H^{s}}\\[2mm]
&\hspace{3.5cm}\leq C(\tilde{v}_0,\tilde{G}_0) T^{\frac{1}{4}}\|\tilde{G}-\tilde{G}_0-tJ^p\nabla\tilde{v}_0\tilde{G}_0\|_{\mathcal{F}^{s,\gamma-1}}+C(\tilde{v}_0,\tilde{G}_0)T.\\[5mm]
&2.\hspace{0.5cm} \|\tilde{G}-\tilde{G}_0\|_{H^1_{(0)}H^{\gamma-1}}\leq \|\tilde{G}-\tilde{G}_0-tJ^p\nabla\tilde{v}_0\tilde{G}_0\|_{H^1_{(0)}H^{\gamma-1}}+\|tJ^p\nabla\tilde{v}_0\tilde{G}_0\|_{H^1_{(0)}H^{\gamma-1}}\\[2mm]
&\hspace{3.8cm}\leq \left\|\int_0^t\partial_t(\tilde{G}-\tilde{G}_0-tJ^p\nabla\tilde{v}_0\tilde{G}_0)\right\|_{H^{1+\eta-\delta}_{(0)}H^{\gamma-1}}+C(\tilde{v}_0,\tilde{G}_0)\|t\|_{H^1_{(0)}}\\[2mm]
&\hspace{3.8cm}\leq C(\tilde{v}_0,\tilde{G}_0)T^{\delta}\|\tilde{G}-\tilde{G}_0-tJ^p\nabla\tilde{v}_0\tilde{G}_0\|_{H^{1+\eta}_{(0)}H^{\gamma-1}}+ C(\tilde{v}_0,\tilde{G}_0)T^{\frac{1}{2}}\\[2mm]
&\hspace{3.8cm}\leq C(\tilde{v}_0,\tilde{G}_0) T^{\delta}\|\tilde{G}-\tilde{G}_0-tJ^p\nabla\tilde{v}_0\tilde{G}_0\|_{\mathcal{F}^{s,\gamma-1}}+C(\tilde{v}_0,\tilde{G}_0)T^{\frac{1}{2}}.\\[5mm]
&3.\hspace{0.5cm} \|\tilde{G}-\tilde{G}_0\|_{H^{\frac{s-1}{2}}_{(0)}H^{1+\mu}}\leq \|\tilde{G}-\tilde{G}_0-tJ^p\nabla\tilde{v}_0\tilde{G}_0\|_{H^{\frac{s-1}{2}}_{(0)}H^{1+\mu}}+\|tJ^p\nabla\tilde{v}_0\tilde{G}_0\|_{H^{\frac{s-1}{2}}_{(0)}H^{1+\mu}}\\[2mm]
&\hspace{4cm}\leq \left\|\int_0^t\partial_t(\tilde{G}-\tilde{G}_0-tJ^p\nabla\tilde{v}_0\tilde{G}_0)\right\|_{H^{\frac{s-1}{2}+\delta-\delta}_{(0)}H^{1+\mu}}+C(\tilde{v}_0,\tilde{G}_0)\|t\|_{H^{\frac{s-1}{2}}_{(0)}}\\[2mm]
&\hspace{4cm}\leq C(\tilde{v}_0,\tilde{G}_0)T^{\delta}\|\tilde{G}-\tilde{G}_0-tJ^p\nabla\tilde{v}_0\tilde{G}_0\|_{H^{\frac{s-1}{2}+\delta}_{(0)}H^{1+\mu}}+C(\tilde{v}_0,\tilde{G}_0)T^{\frac{1}{2}}\\[2mm]
&\hspace{4cm}\leq C(\tilde{v}_0,\tilde{G}_0)T^{\delta}\|\tilde{G}-\tilde{G}_0-tJ^p\nabla\tilde{v}_0\tilde{G}_0\|_{\mathcal{F}^{s,\gamma-1}}+C(\tilde{v}_0,\tilde{G}_0)T^{\frac{1}{2}}.
\end{split}
\end{equation}
\noindent In these estimates we use Lemma \ref{lem2}, with $0<\delta<\eta$.   Furthermore, all the estimates for proving the local existence and the stability results make use of the following lemmas, for details see \cite{CCFGG2}.

\begin{lemma}\label{Jp-est}
Let $2<s<\frac{5}{2}$ and $\tilde{X}-\tilde{\alpha}-tJ^P\tilde{v}_0\in\mathcal{F}^{s+1,\gamma}$. Then, for $T$ small enough we have
\begin{align*}
&1.\hspace{0.5cm}\|J^P(\tilde{X})\|_{L^{\infty}H^{s+1}}\leq C(M, \tilde{v}_0, \|\tilde{\alpha}\|_{L^2},\|\tilde{X}-\tilde{\alpha}-tJ^P\tilde{v}_0\|_{\mathcal{F}^{s+1,\gamma}})\\[3mm]
&2.\hspace{0.5cm}\|J^P(\tilde{X})-J^P\|_{L^{\infty}H^{s+1}}\leq C(M, \tilde{v}_0, \|\tilde{\alpha}\|_{L^2},\|\tilde{X}-\tilde{\alpha}-tJ^P\tilde{v}_0\|_{\mathcal{F}^{s+1,\gamma}})\left(\|\tilde{X}-\tilde{\alpha}-tJ^P\tilde{v}_0\|_{L^{\infty}H^{s+1}}\right.\\[2mm]
&\hspace{13cm}\left.+\|tJ^P\tilde{v}_0\|_{L^{\infty}H^{s+1}}\right)\\[3mm]
&3.\hspace{0.5cm}\|J^P(\tilde{X})-J^P\|_{H^1_{(0)}H^{\gamma}}\leq  C(M, \tilde{v}_0,\|\tilde{X}-\tilde{\alpha}-tJ^P\tilde{v}_0\|_{\mathcal{F}^{s+1,\gamma}})\|\tilde{X}-\tilde{\alpha}\|_{H^1_{(0)}H^{\gamma}}\\[3mm]
&4.\hspace{0.5cm}\|J^P(\tilde{X})-J^P\|_{H^{\frac{s-1}{2}}_{(0)}H^{1+\delta}}\leq  C(M, \tilde{v}_0,\|\tilde{X}-\tilde{\alpha}-tJ^P\tilde{v}_0\|_{\mathcal{F}^{s+1,\gamma}})\left(\|\tilde{X}-\tilde{\alpha}-tJ^P\tilde{v}_0\|_{H^{\frac{s-1}{2}+\eta}_{(0)}H^{1+\delta}}+T\right),
\end{align*}
with 

$$M=\frac{1}{\inf_{\tilde{\alpha}}|\tilde{\alpha}|-C(\tilde{v}_0)T-T^{\frac{1}{4}}\|\tilde{X}-\tilde{\alpha}-tJ^P\tilde{v}_0\|_{\mathcal{F}^{s+1,\gamma}}}.$$
\end{lemma}

\begin{proof}
These results  have been obtained by using the definition of $J^P_{kj}=\partial_{X_j}P_k(P^{-1}(\tilde{X}))$, which contains terms as $\frac{\tilde{X}_i}{|\tilde{X}|}$ and by using estimates \eqref{flux-estim}. 
\end{proof}
\bigskip

\begin{lemma}\label{Jp-dif-est}
Let $2<s<\frac{5}{2}$ and $\tilde{X}-\tilde{\alpha}-tJ^P\tilde{v}_0, \tilde{Y}-\tilde{\alpha}-tJ^P\tilde{v}_0 \in\mathcal{F}^{s+1,\gamma}$. Then, for $T$ small enough we have

\begin{align*}
&1.\hspace{0.5cm}\|J^P(\tilde{X})-J^P(\tilde{Y})\|_{L^{\infty}H^{s+1}}\leq C(M, \|\tilde{X}-\tilde{\alpha}-tJ^P\tilde{v}_0\|_{\mathcal{F}^{s+1,\gamma}}, \|\tilde{Y}-\tilde{\alpha}-tJ^P\tilde{v}_0\|_{\mathcal{F}^{s+1,\gamma}}) \|\tilde{X}-\tilde{Y}\|_{L^{\infty}H^{s+1}}\\[3mm]
&2.\hspace{0.5cm}\|J^P(\tilde{X})-J^P(\tilde{Y})\|_{H^1_{(0)}H^{\gamma}}\leq C(M, \|\tilde{X}-\tilde{\alpha}-tJ^P\tilde{v}_0\|_{\mathcal{F}^{s+1,\gamma}}, \|\tilde{Y}-\tilde{\alpha}-tJ^P\tilde{v}_0\|_{\mathcal{F}^{s+1,\gamma}}) \|\tilde{X}-\tilde{Y}\|_{H^1_{(0)}H^{\gamma}},
\end{align*}
where 

$M=\max\left\lbrace\frac{1}{\inf_{\tilde{\alpha}}|\tilde{\alpha}|-C(\tilde{v}_0)T-T^{\frac{1}{4}}\|\tilde{X}-\tilde{\alpha}-tJ^P\tilde{v}_0\|_{\mathcal{F}^{s+1,\gamma}}},\frac{1}{\inf_{\tilde{\alpha}}|\tilde{\alpha}|-C(\tilde{v}_0)T-T^{\frac{1}{4}}\|\tilde{Y}-\tilde{\alpha}-tJ^P\tilde{v}_0\|_{\mathcal{F}^{s+1,\gamma}}}\right\rbrace.$
\end{lemma}
\medskip

\begin{remark}
The same estimates we obtain for $J^P(\tilde{X})$ and $J^P(\tilde{X})-J^P(\tilde{Y})$ hold also for $Q^2(\tilde{X})$ and $Q^2(\tilde{X})-Q^2(\tilde{Y})$.
\end{remark}
\bigskip

\begin{lemma}\label{zeta-est}
Let $2<s<\frac{5}{2}$, $\tilde{X}-\tilde{\alpha}-tJ^P\tilde{v}_0\in\mathcal{F}^{s+1,\gamma}$ and $\tilde{\zeta}=(\nabla\tilde{X})^{-1}$. Then for $T$ small enough, we have

\begin{align*}
&1.\hspace{0.5cm}\|\tilde{\zeta}\|_{L^{\infty}H^s}+\sum_{i=1}^{2}\|\partial_i\tilde{\zeta}\|_{L^{\infty}H^{s}}\leq C(M, \|\tilde{X}-\tilde{\alpha}-tJ^P\tilde{v}_0\|_{\mathcal{F}^{s+1,\gamma}})\\[3mm]
&2.\hspace{0.5cm}\|\tilde{\zeta}-\mathcal{I}\|_{L^{\infty}H^s}\leq C(M, \|\tilde{X}-\tilde{\alpha}-tJ^P\tilde{v}_0\|_{\mathcal{F}^{s+1,\gamma}})\|\tilde{X}-\tilde{\alpha}\|_{L^{\infty}H^{s+1}}\\[3mm]
&3.\hspace{0.5cm}\|\tilde{\zeta}-\mathcal{I}\|_{H^{\frac{s-1}{2}+\delta}_{(0)}H^{1+\eta}}\leq C(M, \|\tilde{X}-\tilde{\alpha}-tJ^P\tilde{v}_0\|_{\mathcal{F}^{s+1,\gamma}}) \|\tilde{X}-\tilde{\alpha}\|_{H^{\frac{s-1}{2}+\delta}_{(0)}H^{2+\eta}}\\[3mm]
&4.\hspace{0.5cm}\|\tilde{\zeta}-\mathcal{I}\|_{H^1_{(0)}H^{\gamma-1}}\leq C(M, \|\tilde{X}-\tilde{\alpha}-tJ^P\tilde{v}_0\|_{\mathcal{F}^{s+1,\gamma}}) \|\tilde{X}-\tilde{\alpha}\|_{H^1_{(0)}H^{\gamma}},
\end{align*}

where 
\newline

$$M=\frac{1}{1-C(\tilde{v}_0)T-CT^{\frac{1}{4}}\|\tilde{X}-\tilde{\alpha}-tJ^P\tilde{v}_0\|_{\mathcal{F}^{s+1,\gamma}}-CT^{\frac{1}{2}}|\tilde{X}-\tilde{\alpha}-tJ^P\tilde{v}_0\|_{\mathcal{F}^{s+1,\gamma}}^2}.$$
\end{lemma}
\medskip

\begin{proof}
1. Since $\tilde{\zeta}=\frac{1}{\det(\nabla\tilde{X})}(\cof(\nabla(\tilde{X})))^T$,  we need to estimate the  $\det(\nabla\tilde{X})=1+\dive(\tilde{X}-\tilde{\alpha})+\det(\nabla(\tilde{X}-\tilde{\alpha}))$ from below, by using \eqref{flux-estim}.\\ \\
2. Since $\tilde{\zeta}-\mathcal{I}=\tilde{\zeta}(\mathcal{I}-\nabla\tilde{X})=\tilde{\zeta}\nabla(\tilde{\alpha}-\tilde{X})$ then by using the previous result and \eqref{flux-estim} we obtain also the second estimate.\\ \\
3. and 4. We use the previous definition of $\tilde{\zeta}$, $\det(\nabla(\tilde{X}))$ and $\tilde{\zeta}-\mathcal{I}$ in the right spaces.
\end{proof}
\bigskip

\begin{lemma}\label{zeta-dif-est}
Let $2<s<\frac{5}{2}$ and $\tilde{X}^{(n)}-\tilde{\alpha}-tJ^P\tilde{v}_0, \tilde{X}^{(n-1)}-\tilde{\alpha}-tJ^P\tilde{v}_0\in\mathcal{F}^{s+1,\gamma}$. Then for $T$ small enough, we have

\begin{align*}
&1.\hspace{0.5cm}\|\tilde{\zeta}^{(n)}-\tilde{\zeta}^{(n-1)}\|_{H^{\frac{s-1}{2}+\delta}_{(0)}H^{1+\eta}}\leq C(M, \tilde{v}_0)\|\tilde{X}^{(n)}-\tilde{X}^{(n-1)}\|_{H^{\frac{s-1}{2}+\delta}_{(0)}H^{2+\eta}}\\[3mm]
&2.\hspace{0.5cm}\|\tilde{\zeta}^{(n)}-\tilde{\zeta}^{(n-1)}\|_{H^1_{(0)}H^{\gamma-1}}\leq C(M, \tilde{v}_0)\|\tilde{X}^{(n)}-\tilde{X}^{(n-1)}\|_{H^1_{(0)}H^{\gamma}}
\end{align*}
where

$$M=\max_{m=n-1,n}\frac{1}{1-C(\tilde{v}_0)T-CT^{\frac{1}{4}}\|\tilde{X}^{(m)}-\tilde{\alpha}-tJ^P\tilde{v}_0\|_{\mathcal{F}^{s+1,\gamma}}-CT^{\frac{1}{2}}\|\tilde{X}^{(m)}-\tilde{\alpha}-tJ^P\tilde{v}_0\|_{\mathcal{F}^{s+1,\gamma}}^2}.$$
\end{lemma}

\end{document}